\documentclass{article}

\usepackage[authoryear]{natbib}
\PassOptionsToPackage{authoryear,open={((},close={))}}{natbib}
\usepackage[final]{neurips_2022}

\usepackage[T1]{fontenc}    
\usepackage[utf8]{inputenc} 

\usepackage[active]{srcltx}
\usepackage{float}
\usepackage{amsmath}
\usepackage{amsthm}
\usepackage{amssymb}

\PassOptionsToPackage{normalem}{ulem}
\usepackage{ulem}

\makeatletter

\floatstyle{ruled}
\newfloat{algorithm}{tbp}{loa}
\providecommand{\algorithmname}{Algorithm}
\floatname{algorithm}{\protect\algorithmname}

\theoremstyle{plain}
\newtheorem{thm}{\protect\theoremname}
\theoremstyle{remark}
\newtheorem{rem}[thm]{\protect\remarkname}
\theoremstyle{plain}
\newtheorem{cor}[thm]{\protect\corollaryname}
\theoremstyle{plain}
\newtheorem{prop}[thm]{\protect\propositionname}
\theoremstyle{plain}
\newtheorem{lem}[thm]{\protect\lemmaname}

\usepackage{hyperref}       
\usepackage{url}            
\usepackage{booktabs}       
\usepackage{amsfonts}       
\usepackage{nicefrac}       
\usepackage{microtype}      
\usepackage{xcolor}         

\usepackage{enumitem}
\usepackage{breqn}
\usepackage{bbm} 
\usepackage{algorithm,algorithmicx,algpseudocode,caption}



\DeclareMathOperator{\Tr}{Tr}

\allowdisplaybreaks

\global\long\def\s[#1]{\textnormal{\scriptsize #1}}
\global\long\def\st[#1]{\textnormal{\tiny #1}}

\global\long\def\eqd{\stackrel{d}{=}}

\global\long\def\Dkl{\mathrm{D_{KL}}}

\global\long\def\Dchis{\chi^2}
\global\long\def\dtv{\mathrm{d_{{TV}}}}

\global\long\def\loss{\mathsf{loss}}

\global\long\def\M{\mathsf{M}}

\global\long\def\P{\mathbb{P}}
\global\long\def\E{\mathbb{E}}
\global\long\def\V{\mathbb{V}}

\global\long\def\I{\mathbbm{1}}
\global\long\def\T{\top} 
\global\long\def\v[#1]{\mathbf{#1}} 
\global\long\def\m[#1]{\boldsymbol{#1}} 

\global\long\def\r[#1]{#1}
\global\long\def\d{\mathrm{d}}

\global\long\def\dfn{:=}

\global\long\def\trre[#1,#2]{\overset{{\scriptstyle (#2)}}{#1}} 

\makeatother

\usepackage{babel}
\providecommand{\corollaryname}{Corollary}
\providecommand{\lemmaname}{Lemma}
\providecommand{\propositionname}{Proposition}
\providecommand{\remarkname}{Remark}
\providecommand{\theoremname}{Theorem}

\renewcommand\[{\begin{equation}}
\renewcommand\]{\end{equation}}

\title{Mean Estimation in High-Dimensional Binary Markov Gaussian Mixture Models}

\author{%
  Yihan Zhang \\ 
  Institute of Science and Technology, Austria \\
  \texttt{zephyr.z798@gmail.com} \\
  \And
  Nir Weinberger \\
  Technion - Israel Institute of Technology \\
  \texttt{nirwein@technion.ac.il} \\
}

\begin{document}

\maketitle

\begin{abstract}
We consider a high-dimensional mean estimation problem over
a binary hidden Markov model, which illuminates the interplay between
memory in data, sample size, dimension, and signal strength in statistical
inference. In this model, an estimator observes $n$ samples of a
$d$-dimensional parameter vector $\theta_{*}\in\mathbb{R}^{d}$,
multiplied by a random sign $ S_i $ ($1\le i\le n$), and corrupted
by isotropic standard Gaussian noise. The sequence of signs $\{S_{i}\}_{i\in[n]}\in\{-1,1\}^{n}$
is drawn from a stationary homogeneous Markov chain with flip probability
$\delta\in[0,1/2]$. As $\delta$ varies, this model smoothly interpolates
two well-studied models: the Gaussian Location Model for which $\delta=0$
and the Gaussian Mixture Model for which $\delta=1/2$. Assuming that
the estimator knows $\delta$, we establish a nearly minimax optimal
(up to logarithmic factors) estimation error rate, as a function of
$\|\theta_{*}\|,\delta,d,n$. We then provide an upper bound to
the case of estimating $\delta$, assuming a (possibly inaccurate)
knowledge of $\theta_{*}$. The bound is proved to be tight when $\theta_{*}$
is an accurately known constant. These results are then combined to
an algorithm which estimates $\theta_{*}$ with $\delta$ unknown
a priori, and theoretical guarantees on its error are stated. 
\end{abstract}




\section{Introduction}

Memory between data samples is ubiquitous in practical applications, as data collected from networks or sampled time series inevitably inherit spatial or temporal statistical dependencies. Numerous examples arise in imaging, meteorology, health care, finance, social science, and so on \citep{glaeser1996crime,bertrand2000network,sacerdote2001peer,duflo2003role,christakis2013social}. In principle, prior knowledge of the existence of such memory can be used to improve the performance of statistical estimators compared to the performance obtained in memoryless models. Development and analysis of statistical inference algorithms for models with memory has been extensively explored from the algorithmic perspective, and computationally efficient algorithms such as Baum-Welch and message-passing were developed \citep{ephraim2002hidden, van2008hidden, MacKay-book, wain-jordan-graphical}. It was also extensively studied from the theoretical perspective in the classical, fixed-dimensional and asymptotic regime, e.g., \citet[Chapter 27 and references therein]{gyorfi-book}. However, much less is understood about the high dimensional, non-asymptotic regime, which is of paramount importance in modern applications, and the focus of current extensive research \citep{vershynin2018high, wainwright2019high}. As we exemplify in this paper, the error in such estimation problems depends in an intricate way on the interplay between the number of samples, the dimension of the vector parameters to be estimated, the noise level (signal-to-noise ratio), and the level of memory between the samples.

An ever popular and fundamental model is the Gaussian mixture model, in which memory exists between samples whenever the latent variables determining the component of each sample are dependent. Numerous recent papers, e.g., 
\citet{balakrishnan2017statistical}, 
\citet{xu2016global}, 
\citet{klusowski2016statistical}, 
\citet{jin2016local}, 
\citet{dwivedi2018singularity}, 
\citet{dwivedi2018theoretical}, 
\citet{dwivedi2020sharp}, 
\citet{zhao2018statistical}, 
\citet{yan2017convergence}, 
\citet{weinberger2021algorithm} 
have focused on high-dimensional \emph{memoryless} models, and analyzed computationally efficient estimation algorithms, most notably the expectation maximization (EM) algorithm. Specifically, the seminal \citet{balakrishnan2017statistical} has provided theoretical guarantees for EM in memoryless models, though without proving minimax optimality, which was subsequently established in \citet{wu2019EM}. As a notable exception, \citet{yang2015statistical} has generalized the analysis of \citet{balakrishnan2017statistical} to a \emph{hidden Markov model}  (HMM) that has memory, yet again, without determining how the minimax error rate depends on the number of samples, dimension, noise level, and the amount of memory. In this paper, we address the question of precise characterization of the minimax error rate in terms of these parameters, in the context of a high-dimensional Gaussian HMM. We next turn to formally define this model and the estimation problem, describe known results, and then present our contributions. Our obtained results illuminate the opportunities and challenges associated with optimal inference in high dimensional models with memory.

\subsection{Problem formulation }

Let $S_{0}^{n}\dfn(S_{0},S_{1},\ldots,S_{n})$ be the following homogeneous
binary symmetric Markov chain, $S_{i}\in\{-1,1\}$, $\P[S_{0}=1]=1/2$
and 
\begin{equation}
S_{i}=\begin{cases}
S_{i-1}, & \text{w.p. }1-\delta\\
-S_{i-1}, & \text{w.p. }\delta
\end{cases}\label{eq: Markov sign model}
\end{equation}
for $i\in[n]\dfn\{1,\ldots,n\}$, and where $\delta\in[0,1]$ is the flip
probability of the binary Markov chain. We also denote $\rho\dfn1-2\delta\in[-1,1]$
which is the correlation between adjacent samples $\rho=\E[S_{i}S_{i+1}]$.
At each time point $i\in[n]$, a sample of a $d$-dimensional Gaussian
mixture model is observed
\begin{equation}
X_{i}=S_{i}\theta_{*}+Z_{i},\label{eq: Gaussian Markov model}
\end{equation}
where $Z_{i}\sim N(0,I_{d})$ is an i.i.d. sequence, independent
of $S_{0}^{n}$, and where $\theta_{*}\in\mathbb{R}^{d}$, $d\geq1$.
At its two extremes, this model degenerates to one of two fundamental
models, which are well studied. When $\delta=0$, the memory length is infinite, and the
sign $S_{0}=S_{1}\cdots=\cdots=S_{n}$ is fixed. Thus, up to
this sign ambiguity, the model (\ref{eq: Gaussian Markov model})
is the standard \emph{Gaussian location model} (GLM), which is essentially
a memoryless model (and exactly so if $S_{0}$ is known). When $\delta=\frac{1}{2}$,
the signs $S_{0}^{n}$ are i.i.d. and have no memory at all.
The model (\ref{eq: Gaussian Markov model}) is then a \emph{Gaussian mixture  model} (GMM) with two symmetric components, which is also a memoryless model.
In all other cases, $0<\delta<\frac{1}{2}$ (or $\frac{1}{2}<\delta<1$),
the model is a simple version of a HMM. 

The inference problem we consider in this paper is the estimation
of $\theta_{*}\in\mathbb{R}^{d}$, under the loss function 
\begin{equation}
\loss(\hat{\theta},\theta_{*})\dfn\min\{\|\hat{\theta}-\theta_{*}\|,\|\hat{\theta}+\theta_{*}\|\},\label{eq: loss function for means}
\end{equation}
that is, the Euclidean distance error under a possible sign ambiguity.\footnote{Similar bounds can be derived for the squared loss by trivial extensions. }
An intermediate goal (or an additional problem) is to estimate $\delta$,
under the regular absolute error loss function $|\hat{\delta}-\delta|$. 
The fundamental limits of this estimation problem will be gauged by the 
\emph{local} minimax rate, which is the maximal decrease rate of the
loss possible for any estimator, given $n$ samples, at dimension
$d\geq2$, for signal strength $\|\theta_{*}\|=t$, and under flip
probability $\delta$. Specifically, for $d\geq2$ it is defined as
\[
\M(n,d,\delta,t)\dfn\inf_{\hat{\theta}(X_{1}^{n})}\sup_{\|\theta_{*}\|=t}\E\left[\loss(\theta_{*},\hat{\theta}(X_{1}^{n}))\right].
\]
For general $d\geq1$, the \emph{global} minimax rate is defined with
the condition $\|\theta_{*}\|=t$  replaced by $\|\theta_{*}\|\leq t$ (this condition trivializes the estimator for $d=1$). 

\subsection{Known minimax estimation errors rates for GLM and GMM}

Before delving into models with memory ($0<\delta<1/2$), we review known results on the minimax rates in memoryless high dimensional Gaussian models -- the GLM ($\delta=0$) and the GMM ($\delta=1/2$). We refer the reader to note that the minimax error rate in these models may undergo two possible phase transitions -- one as $t$ increases,
and the other one as $d$ increases. We also remark that the regime of interest is that of low-separation ($t\lesssim1$), in which accurately detecting the components is impossible, yet parameter estimation with vanishing loss is possible.

At the first extreme, the local minimax rate for the GLM ($\delta=0$)
is the usual parametric error rate 
\begin{equation}
\M_{\text{GLM}}(n,d,t)\dfn\M(n,d,0,t)\asymp\begin{cases}
t, & t\leq\sqrt{\frac{d}{n}}\\
\sqrt{\frac{d}{n}}, & t\geq\sqrt{\frac{d}{n}}
\end{cases}.\label{eq: minimax rates -- Gaussian location model}
\end{equation}
This is achieved by the trivial estimator $\hat{\theta}=0$ if $t\leq\sqrt{\frac{d}{n}}$
and the simple empirical average estimator $\hat{\theta}=\frac{1}{n}\sum_{i=1}^{n}X_{i}$
if $t\geq\sqrt{\frac{d}{n}}$ (see \citep[Sec.\ III-9]{wu2017lecture}). The rate $\Theta(\sqrt{\frac{d}{n}})$
is then the \emph{global} minimax rate, i.e., the largest error over $t=\|\theta_{*}\|>0$.
This model does not have a phase transition with dimension. At the
other extreme, the GMM model ($\delta=\frac{1}{2}$) undergoes a phase
transition at the dimension $d=n$. At low dimension, $d\leq n$,
the minimax rate was neatly shown in \citep[Appendix B]{wu2019EM}
to be
\begin{equation}
\M_{\text{GMM}}(n,d,t)\equiv\M\left(n,d,\frac{1}{2},t\right)\asymp\begin{cases}
t, & t\leq\left(\frac{d}{n}\right)^{1/4}\\
\frac{1}{t}\sqrt{\frac{d}{n}}, & \left(\frac{d}{n}\right)^{1/4}\leq t\leq1\\
\sqrt{\frac{d}{n}}, & t>1
\end{cases},\label{eq: minimax rates Gaussian mixture low dimension}
\end{equation}
whereas at high dimension $d\geq n$, it is as for the Gaussian location
model in (\ref{eq: minimax rates -- Gaussian location model}), i.e.,
$\M_{\text{GMM}}(n,d,t)=\M_{\text{GLM}}(n,d,t)$. Hence, 
at this high dimensional regime, the loss does not vanishes by increasing the signal strength $t$, or by increasing the number of samples $n$. The loss in (\ref{eq: minimax rates Gaussian mixture low dimension})
is achieved by the trivial estimator $\hat{\theta}=0$ if $t\leq\left(\frac{d}{n}\right)^{1/4}$
and by an estimator given by a properly scaled and shifted principal component
of the empirical covariance matrix of $X_{1}^{n}\dfn(X_{1},X_{2},\ldots,X_{n})$,
if $t\geq\left(\frac{d}{n}\right)^{1/4}$. For the GMM at low dimension,
$d\leq n$, the global minimax rate is $(\frac{d}{n})^{1/4}$, which
is worse than the minimax rate of the Gaussian location model $\sqrt{\frac{d}{n}}.$ 

Therefore, the GMM has worse estimation performance compared to the
GLM from three aspects: First, at low dimension, $d\leq n$, it has
a larger global minimax rate $(\frac{d}{n})^{1/4}$ compared to the
parametric error rate of the GLM, $\sqrt{\frac{d}{n}}$; Second, at
low dimension, $d\leq n$, parametric error rate is achieved only
for constant separation $t\geq1$; Third, the transition to the high
dimension regime occurs at $d=n$. 

As is intuitively appealing from a ``data-processing'' reasoning,
a Markov model with flip probability $\delta'$ should allow for lower
estimation error of $\theta_{*}$ compared to a Markov model with
$\delta>\delta'$. Indeed, and as a specific simple example, any Markov
model with $\delta<\frac{1}{2}$ can be easily transformed to a GMM
model by randomizing the signs of each of the samples by an independent
Rademacher variable. Thus we may deduce, e.g., that
since at high dimension ($d\geq n$) the GLM and the GMM have the
same minimax rates, the minimax rates for $d\geq n$ are in fact as
in (\ref{eq: minimax rates -- Gaussian location model}) for \emph{any} $\delta\in[0,1]$.
We thus henceforth exclusively focus on the regime $d\leq n$.
As we show, it is generally true that the improvement in
estimation error when $\delta$ is reduced is less profound as the dimension
increases. 

\subsection{Contributions} \label{sec: contributions}

We first consider the case in which $\delta$ is known to the estimator
of $\theta_{*}$. For this case, we analyze the loss of an estimator that is based on a computation of the principal
component of a properly chosen empirical covariance matrix. We show (Theorem \ref{thm: mean estimation known delta upper bound})
that, at low dimension, $d\leq\delta n$, it achieves a local minimax rate of 
\begin{equation}
\M(n,d,\delta,t)\lesssim\begin{cases}
t, & t\leq\left(\frac{\delta d}{n}\right)^{1/4}\\
\frac{1}{t}\sqrt{\frac{\delta d}{n}}, & \left(\frac{\delta d}{n}\right)^{1/4}\leq t\leq\sqrt{\delta}\\
\sqrt{\frac{d}{n}}, & t\geq\sqrt{\delta}
\end{cases}, \label{eq: minimax loss at low dimension Markov contribution statement}
\end{equation}
and at high dimension, $d\geq\delta n$, it is as for the Gaussian location
model in (\ref{eq: minimax rates -- Gaussian location model}). The
rate of this estimator is then further shown to be asymptotically
optimal (up to a logarithmic factor) via a minimax lower bound (Theorem \ref{thm: impossibility lower bound for mean estimation}).
Evidently, the loss  in (\ref{eq: minimax loss at low dimension Markov contribution statement}) smoothly interpolates the rates of the GLM in (\ref{eq: minimax rates -- Gaussian location model})
and the GMM in (\ref{eq: minimax rates Gaussian mixture low dimension}). Moreover, it is evident that the loss is improved with the decrease of
$\delta$ from all three aspects previously mentioned. First, at low dimension, $d\leq\delta n$, the global minimax rate is $\left(\frac{\delta d}{n}\right)^{1/4}$ (obtained by equating the first and second cases in (\ref{eq: minimax loss at low dimension Markov contribution statement})). Hence, reducing the flip probability from $\delta=\Theta(1)$ to $\delta=\frac{d}{n}$ smoothly reduces the global minimax rate from the parametric error rate $\Theta(\sqrt{\frac{d}{n}})$ of the GLM to the $ \Theta( \left(\frac{d}{n}\right)^{1/4}) $ of the GMM. 
Second, at low dimension, $d\leq\delta n$, the minimal signal strength required to obtain parametric error rate is $t = \Theta(\sqrt{\delta})$ (obtained from the third case in (\ref{eq: minimax loss at low dimension Markov contribution statement})). This, again, improves by decreasing $\delta$, and matches the extremes $t = \Theta(\sqrt{\frac{d}{n}})$ of GLM and $t = \Theta(1)$ of GMM.
Third, the transition to the high dimension regime occurs at $d=\delta n$, which is again better with lower $\delta$, and matches the transition point of the GMM given by $d=n$. This lower transition
point allows us to achieve the error rate of the GLM for any signal strength,
even in a regime in which the loss $O(\sqrt{\frac{d}{n}})$ vanishes with $n\to\infty$ (unlike
for GMM, $\delta=\frac{1}{2}$); specifically, this occurs whenever
$d\geq \delta n$ yet $d=o(n)$. Beyond the formal proof, Appendix \ref{app:heuristic} provides a heuristic justification 
for why the minimax error is naturally expected to scale as in \eqref{eq: minimax loss at low dimension Markov contribution statement}. 

Second, as a step towards the removal of the assumption that $\delta$
is known to the estimator, we consider the complementary problem of
estimating $\delta$ whenever an estimate $\theta_{\sharp}$ of $\theta_{*}$
is available (which can be either exact $\theta_{\sharp}=\theta_{*}$,
or inaccurate $\theta_{\sharp}\neq\theta_{*}$). We propose a simple estimator for $\delta$, and analyze
its error in case of a mismatch (Theorem \ref{thm: estimation error of delta for mismatched delta}).
We then specify this result to the matched case $\theta_{\sharp}=\theta_{*}$
and show that in the non-trivial regime ($\|\theta_{*}\|\lesssim1$)
its error rate is $\tilde{O}(\frac{1}{\|\theta_{*}\|^{2}}\sqrt{\frac{1}{n}})$.
We then proceed to show an impossibility lower bound of $\Omega(\sqrt{\frac{1}{n}})$
(Proposition \ref{prop: impossibility lower bound for flip probability estimation})
for this error rate. The precise dependence of the estimation error rate of $\delta$ on $\|\theta_*\|$ is therefore not precisely determined, and we discuss the challenges in settling this matter.  

Third, we consider the case in which the estimator of $\theta_{*}$
has no prior knowledge of $\delta$. We propose a three-step algorithm
for this case (Algorithm \ref{alg:Mean estimation with unknown delta}).
First, a (possibly) gross estimate $\hat{\theta}^{(A)}$of $\theta_{*}$
is computed based on third of the samples, assuming the worst case
of $\delta=\frac{1}{2}$. Then, an estimate $\hat{\delta}^{(B)}$
of $\delta$ is computed using another third of the samples, assuming
the estimate $\hat{\theta}^{(A)}$. Finally, a refined estimate $ \hat{\theta}^{(C)} $ of
$\theta_{*}$ is obtained by (essentially) assuming that $\delta$
is $\hat{\delta}^{(B)}$. At each of the steps above, the algorithm
may stop and decide to return its current estimate when it determines
that no further improvement is possible by moving on to the next steps.
We analyze the loss of this algorithm (Theorem \ref{thm: Estimation error for algorithm}),
and show that this algorithm is capable of partially achieving the
gains associated with the case of known $\delta$. 

Our technique of partitioning the samples to blocks according to the dependence structure of the Markov chain renders the possibility of understanding the minimax rate of more general models, such as mixtures with multiple components \citep{doss2020location} and memory, Ising models \citep{daskalakis-dependent-regression}, Boltzmann machines \citep{bresler2019boltzmann}, Markov random fields, etc. Our findings in the mean estimation setting also stand in contrast to the \emph{data wastage} phenomenon observed in the linear regression setting in the prior work \citet{bresler-markov-linear-regression}.

\subsection{Additional related work}

Both the GLM \citep{johnstone2002function,tsybakov2008introduction}
and GMM \citep{lindsay1995mixture,mclachlan2019finite} are classic
models which were well-explored from numerous perspectives. 
In the last few years, there is a
surge of interest in the non-asymptotic performance analysis of computationally
efficient estimation algorithms for this estimation task. For example,
\citet{moitra2010settling,kalai2010efficiently,anandkumar2014tensor,hardt2015tight,wu2020optimal}
have analyzed method-of-moments-based algorithms, and various other papers
considered the EM algorithm \citep{balakrishnan2017statistical,xu2016global,jin2016local,klusowski2016statistical,weinberger2021algorithm,dwivedi2018singularity,dwivedi2020sharp,dwivedi2018theoretical,zhao2018statistical,yan2017convergence}.
Specifically, the local minimax rate for GMM in (\ref{eq: minimax rates Gaussian mixture low dimension})
was determined in \citep{wu2019EM} as a benchmark for the operation
of the EM algorithm. 

The model (\ref{eq: Gaussian Markov model}) is a simple instance
of a HMM \citep{ephraim2002hidden,van2008hidden} in high dimensions.
Parameter estimation in such models is practically performed via the
Baum-Welch algorithm \citep{baum1970maximization}, which is a computationally
efficient version of EM for HMMs. 
To the best of our knowledge, there were hardly any attempts to characterize
the minimax rates in such models, with the exception of \citep{aiylam2018parameter} and \citep{yang2015statistical}, previously mentioned. In \citep{aiylam2018parameter}, a local version of the Baum-Welch algorithm was proposed, and vanishing error of the convergence of the estimate to the true parameter was established for both a population version as well as a finite-sample version. In \citep{yang2015statistical}, general bounds on the performance of Baum-Welch algorithm were specified to
the Gaussian model with Markov signs (\ref{eq: Gaussian Markov model})
considered here. It was shown that the Baum-Welch algorithm achieves a parametric error rate, and
converges in a finite number of iterations \citep[Corollary 2]{yang2015statistical}. 
However, the qualifying
condition for the estimation error bound of \citep{yang2015statistical} is that $t=\|\theta_{*}\|\gtrsim\log\frac{1}{1-(1-2\delta)^{2}}$,
that is, a non-trivial separation when $\delta$ is constant, which
further blows up as $\delta\downarrow0$. By contrast, in this paper,
our goal is to characterize the estimation error in the regime of
$\delta$ and $t=\|\theta_{*}\|$ in which the minimax rate is
affected by these parameters, and this requires analyzing vanishing
$\delta$ and $t$. 

More broadly, there is a growing interest in advancing the quantitative understanding of the performance of statistical learning and inference with dependent data. 
\citet{bresler-markov-linear-regression} studied linear regression with Markovian covariates and characterized the minimax error rate in terms of the mixing time of the Markov chain. 
A stochastic gradient descent-style algorithm adapted to the Markov setting was shown to be minimax optimal. 
Statistical estimation problems including linear and logistic regression with more general network dependencies among response variables were studied by \citet{daskalakis-dependent-regression} and \citet{estimation-dependent}. 
Learnability and generalization bounds were derived by \citet{learning-weak-dependent} for dependent data satisfying the so-called Dobrushin's condition. 

\subsection{Notation conventions}

For a vector $v\in\mathbb{R}^{d}$, $\|v\|$ is the Euclidean norm.
For a
positivedefinite matrix $A$, $\lambda_{\max}(A)$ and $v_{\max}(A)$
are the maximal eigenvalue and the associated eigenvector (of unit
norm) of $A$. Unless
otherwise stated, the constants involved in Bachmann-Landau notation
are numerical, and do not depend on the 
parameters $(n,d,\delta,t)$. 
It holds that $ a\gtrsim b $ (resp.\ $ a\lesssim b $) if there exists a constant $ c>0 $ (resp.\ $ C>0 $) such that $ a\ge cb $ (resp.\ $ a\le Cb $). 
If $ a\lesssim b $ and $ a\gtrsim b $. then $ a\asymp b $. 
Integer constraints (ceiling and floor) on large quantities that do not affect the results are omitted for brevity. 
For a real numbers $ a,b $ the shorthand notation $ a\vee b \dfn \max\{a,b\}, a\wedge b \dfn \min\{a,b\} $  
and $ (a)_+ \dfn a\vee0 $ is used.
A sequence of objects $ X_1,\cdots,X_n $ is denoted by $ X_1^n $. 
Expectation, variance and probability are denoted by $ \mathbb{E},\mathbb{V}$ and $\mathbb{P} $, respectively. 
Equality in distribution of random variables $ X $ and $ Y $ is denoted by $ X\eqd Y $. All logarithms $\log$ are to the base $e$.

\section{Mean estimation for a known flip probability\label{sec:Mean-estimation-for}}

In this section, we consider the problem of estimating $\theta_{*}$
whenever $\delta$ is exactly known to the estimator. In that case,
it may be assumed w.l.o.g. that $\delta\in[0,\frac{1}{2}]$, as otherwise
one may negate each of the even samples to obtain an equivalent model
with $\delta$ replaced with $1-\delta$. Hence also $\rho\in[0,1]$.
We next describe an estimator for this task, state a bound on its
performance, and then show that it matches (up to a logarithmic factor) an impossibility lower bound. 

The estimator operationally interpolates and therefore simultaneously generalizes the empirical average estimator \eqref{eqn:glm-estimator} and the (properly scaled) principal component estimator \eqref{eqn:gmm-estimator} analyzed in \citep[Appendix B]{wu2019EM}. 
It degenerates to the latter estimators if $ \delta\downarrow0 $ or $ \delta\uparrow\frac{1}{2} $. 
Specifically, the estimator partitions the sample into blocks of equal length $k$ each (which will later be set to $ k = \frac{1}{8\delta} $, according to the mixing time of the Markov chain $S_0^n$).
Let $\ell$ denote the number of blocks respectively, so that
$k\ell=n$. Let ${\cal I}_{i}=\{(i-1)k+1,(i-1)k+2,\cdots,ik\}$ denote
the indices of the $i$th block. Further, let $\{R_{i}\}_{i\in[\ell]}$
be an i.i.d. Rademacher sequence ($R_{i}\sim\text{Uniform}\{-1,1\}$),
and let 
\[\overline{X}_{i}\dfn R_{i}\cdot\frac{1}{k}\sum_{j\in{\cal I}_{i}}X_{j}=\overline{S}_{i}\theta_{*}+\overline{Z}_{i}\]
denote the average of the samples in the $i$th block (randomized
with a sign $R_{i}$), where 
\[\overline{S}_{i}\dfn R_i\cdot\frac{1}{k}\sum_{j\in{\cal I}_{i}}S_{j}\]
is the \textit{gain} (average of the signs) of the $i$th block, and 
\[\overline{Z}_{i}\dfn R_i\cdot\frac{1}{k}\sum_{j\in{\cal I}_{i}}Z_{j}\sim N\left(0,\frac{1}{k}\cdot I_{d}\right)\]
is the average noise of the $i$th block. Due to the sign randomization,
it holds that $\{\overline{S}_{i}\}_{i\in[\ell]}$ is an i.i.d. sequence.
Since $\{\overline{Z}_{i}\}_{i\in[\ell]}$ is also an i.i.d. sequence,
then so is $\{\overline{X}_{i}\}_{i\in[\ell]}$. For notational simplicity
we will omit the block index $i$ of a generic block. For block length
$k$, we denote by
\[\xi_{k}\dfn\E[\overline{S}^{2}]=\E\left[\left(\frac{1}{k}\sum_{j=1}^{k}S_{j}\right)^{2}\right]\]
the second moment of the gain $\overline{S}$. 
Note that $ \xi_k\in[\frac{1}{k},1] $ for any $ \delta\in[0,\frac{1}{2}]$, and, in particular, it is always positive. 
For a sequence of samples
$X_{1}^{n}=(X_{1},\ldots,X_{n})$, we define by $\hat{\Sigma}_{n,k}(X_{1}^{n})$
the empirical covariance matrix of the averaged samples over blocks
$\{\overline{X}_{i}\}_{i\in[\ell]}$, that is 
\[\hat{\Sigma}_{n,k}(X_{1}^{n})\dfn\frac{1}{\ell}\sum_{i=1}^{\ell}\overline{X}_{i}\overline{X}_{i}^{\T},\]
whose population average is $\Sigma_{n,k}(\theta_{*})$, where
\[\Sigma_{n,k}(\theta_*)\dfn\E[\overline{X}\,\overline{X}^{\T}]=\xi_{k}\theta_*\theta_*^{\T}+\frac{1}{k}I_{d}.\]
We note that $\theta$ is the principal component of $\Sigma_{n,k}(\theta)$,
that is, $\lambda_{\text{max}}(\Sigma_{n,k}(\theta))=\xi_{k}\|\theta\|^{2}+\frac{1}{k}$
and the corresponding eigenvector is $v_{\text{max}}(\Sigma_{n,k}(\theta))=\theta$.
We thus consider the following estimator for $\theta_{*}$, from a
sequence $X_{1}^{n}$, and with a block length of $k$ 
\begin{equation}
\hat{\theta}_{\text{cov}}(X_{1}^{n};k)\dfn\sqrt{\frac{1}{\xi_{k}}\left(\lambda_{\max}(\hat{\Sigma}_{n,k}(X_{1}^{n}))-\frac{1}{k}\right)_{+}}\cdot v_{\max}\left(\hat{\Sigma}_{n,k}(X_{1}^{n})\right).\label{eq: general PCA estimation rule}
\end{equation}
The estimator is thus constructed from two types of averages: First,
a coherent average of the samples at each block, to obtain $\ell$
block-samples $\overline{X}_{i}$ with gain $\overline{S}_{i}$ and 
noise variance reduced by a factor of $k$. Second, an incoherent average
of the ``square'' of the $\ell$ block-samples $\overline{X}_{i}\overline{X}_{i}^{\T}$,
which resolves the remaining sign ambiguity between blocks. This balance two extreme cases: If $\delta=0$, then
this reduces the problem to the GLM (with a sign
ambiguity) and $k=n$ is an optimal choice. If $\delta=\frac{1}{2}$,
then this reduces the problem to the GMM, in which coherent averaging is non-beneficial and $k=1$ is rate optimal. Generally, the optimal choice of the block length $k$ is proportional to the mixing
time of the Markov chain $\Theta(\frac{1}{\delta})$. This choice assures that the random gain $\overline{S}$ is $\pm1$ with a (constant)
high probability. In fact, an elementary, yet crucial, part of the analysis establishes that the random gain $\bar{S}$ has constant variance for this choice of block length
(see Lemma \ref{lem:analysis of the random gain} in Appendix \ref{app:proof-achievability-estimate-theta}).
On the other hand, if $k=\Omega(\frac{1}{\delta})$,
then the random gain $\overline{S}$ will not be $\pm1$ (or not even bounded
away from zero) with high probability, and such choice is never efficient.
Specifically, we consider the estimator in (\ref{eq: general PCA estimation rule})
with $k=\frac{1}{8\delta}$. 
The above estimation procedure is depicted in Figure \ref{fig:block} in Appendix \ref{app:proof-achievability-estimate-theta}.

Let us denote 
\[\beta(n,d,\delta)\dfn\sqrt{\frac{d}{n}}\vee\left(\frac{\delta d}{n}\right)^{1/4},\]
which will actually be the global minimax rate.
\begin{thm}
\label{thm: mean estimation known delta upper bound}Assume that $\delta\geq\frac{1}{n}$
and $d\leq n$, and set $\hat{\theta}\equiv\hat{\theta}_{\text{\emph{cov}}}(X_{1}^{n};k)$
with $k=\frac{1}{8\delta}$. Then, there exist numerical constants
$c_{0},c_{1},c_{2}>0$ such that for every $ \theta_*\in\mathbb{R}^d $
\[
\E\left[\loss(\hat{\theta},\theta_{*})\right]\leq c_{0}\cdot\begin{cases}
\beta(n,d,\delta), & \|\theta_{*}\|\leq\beta(n,d,\delta)\\
\sqrt{\frac{d}{n}}+\frac{1}{\|\theta_{*}\|}\sqrt{\frac{\delta d}{n}}+\frac{1}{\|\theta_{*}\|}\cdot\frac{d}{n}, & \beta(n,d,\delta)\leq\|\theta_{*}\|
\end{cases}
\label{eqn:cov-estimator-rate}
\]
and 
\[\loss(\hat{\theta},\theta_{*})\leq c_{1}\cdot\log(n)\cdot\E\left[\loss(\hat{\theta},\theta_{*})\right]\]
with probability larger than $1-\frac{c_{2}}{n}$. 
\end{thm}

Theorem \ref{thm: mean estimation known delta upper bound} is proved in Appendix \ref{app:proof-achievability-estimate-theta}. 
Evidently, Theorem \ref{thm: mean estimation known delta upper bound} implies that the upper bound on the minimax rate stated in (\ref{eq: minimax loss at low dimension Markov contribution statement})
above holds in low dimension, $d\leq\delta n$, whereas $\M(n,d,\delta,t)\asymp\M_{\text{GLM}}(n,d,t)$
holds in high dimension $d\geq\delta n$.\footnote{When $ \|\theta_*\|\le\beta(n,d,\delta) $, the estimator $ \hat{\theta}_{\text{cov}}(X_1^n;k) $ in Theorem \ref{thm: mean estimation known delta upper bound} only achieves a rate $ \beta(n,d,\delta) $ which is larger than the promised rate $ \|\theta_*\| $ in \eqref{eq: minimax loss at low dimension Markov contribution statement}. 
However, since the estimator is assumed to know $t$ (but not the direction of $ \theta_* $; a common formulation in high-dimensional statistics), then it can output the zero vector. 
It then incurs loss $ \|\theta\|_* $ for any $ \theta_*\in\mathbb{R}^d $, matching the promised rate when $ \|\theta_*\|\le\beta(n,d,\delta) $. 
To summarize, for any value of $t$, the minimax rate is achieved by the minimum rate of $ \hat{\theta}_{\text{cov}}(X_1^n;k) $ and $ \hat{\theta}_0(X_1^n) \equiv 0 $.
}
We remark that the condition $ \delta\ge\frac{1}{n} $ is mild as otherwise the model \eqref{eq: Gaussian Markov model} is essentially equivalent to GLM. 
See Remark \ref{rk:remark-delta-at-least-1-over-n} in Appendix \ref{app:proof-achievability-estimate-theta}. 
Numerical validation of the performance of the estimator $ \hat{\theta}_{\text{cov}}(X_1^n;k) $ is shown 
in Appendix \ref{app:numerical-validation}.

We next consider an impossibility result. As we have seen, at high
dimension, $d\geq\delta n$, the minimax error rates achieved are the
same as for the Gaussian location model, and thus clearly cannot be
improved. We thus next focus on the low dimensional regime $d\leq\delta n$. 
\begin{thm}
\label{thm: impossibility lower bound for mean estimation}Assume
that $2\le d\leq\delta n$ and $n\geq\frac{128}{d}$. 
Then the local minimax rate is bounded as
\[
\M(n,d,\delta,t)\gtrsim\frac{1}{\sqrt{\log(n)}}\cdot\begin{cases}
t, & t\leq\left(\frac{\delta d}{n}\right)^{1/4}\\
\frac{1}{t}\sqrt{\frac{\delta d}{n}}, & \left(\frac{\delta d}{n}\right)^{1/4}\leq t\leq\sqrt{\delta}\\
\sqrt{\frac{d}{n}}, & t\geq\sqrt{\delta}
\end{cases}.
\]
\end{thm}

Hence, the minimax rates achieved by the estimator in Theorem \ref{thm: mean estimation known delta upper bound}
are nearly asymptotically optimal, up to a $\sqrt{\log(n)}$ factor. 
The full proof of Theorem \ref{thm: impossibility lower bound for mean estimation} together with a summary of the main ideas used in the proof is presented in Appendix \ref{app:proof-converse-estimate-theta}. 

\begin{rem}[Relaxation of the noise distribution assumption]
\label{rk:extensions}
For the sake of clarity of exposition, we have assumed in Theorems  \ref{thm: mean estimation known delta upper bound} and \ref{thm: impossibility lower bound for mean estimation} that the noise samples $ \{Z_i\}_{i=1}^n $ are i.i.d.\ isotropic Gaussians. 
There are two straightforward relaxations of this assumption. First, our minimax upper bound (Theorem \ref{thm: mean estimation known delta upper bound}) can be proved to any subGaussian noise distribution, simply because all the concentration bounds for Gaussian random variables used in the proof admit subGaussian analogues. The  impossibility result  (converse, Theorem \ref{thm: impossibility lower bound for mean estimation}) trivially holds for subGaussian noise since Gaussians are special case of subGaussians. Second, the isotropic assumption can be relaxed to anisotropic noise with known covariance $\Sigma$ by simple standardization: If $ Z_i\sim{N}(0_d,\Sigma) $ are i.i.d.\ for some \emph{known} $\Sigma\succ0$, then the estimator will multiply the samples by $ \Sigma^{-1/2} $ and reduce the problem back to the isotropic setting. After applying the estimator we propose for the isotropic case, the estimator will obtain its final estimate by multiplying its isotropic estimate by $ \Sigma^{1/2} $. The loss of the estimator will then be gauged by the Mahalanobis distance, parameterized by $\Sigma$.
We refer the reader to Appendix \ref{app:open} for a discussion on more challenging directions in which the isotropic Gaussian assumption can be relaxed. 

\end{rem}

\section{Flip probability estimation for a given estimator of $\theta_{*}$
\label{sec:Flip-probability-estimation}}

In this section, we consider the problem of estimating $\delta$ whenever
$\theta_{*}$ is approximately known to be $\theta_{\sharp}$. We
propose a simple estimator, and then discuss the importance of the
accuracy of $\theta_{*}$. We then derive an impossibility result
for the matched case, $\theta_{\sharp}=\theta_{*}$. 

First note that an estimator for $\delta$ can be easily obtained from an estimator
for $\rho$, with essentially the same error rate, via $\hat{\delta}=\frac{1}{2}(1-\hat{\rho})$. Thus we focus on estimating $\rho$.
Assume for simplicity that $n$ is even. 
Observing that $\E[X_{2i}^\top X_{2i+1}] = \rho\|\theta_*\|^2 $, we propose the following natural estimator for $\rho$, which replaces the population average with empirical average:
\begin{equation}
\hat{\rho}_{\text{corr}}(X_{1}^{n};\theta_{\sharp})=\frac{1}{\|\theta_{\sharp}\|^{2}}\cdot\frac{2}{n}\sum_{i=1}^{n/2}X_{2i}^{\T}X_{2i-1}.\label{eq: estimator for rho}
\end{equation}
That is, the estimator is based on evaluating the correlation of each
of two adjacent samples $X_{2i}$ and $X_{2i-1}$. We first
state a general bound on the estimation error of this estimator. We
then consider the case in which $\theta_{*}$ is known,
and show how the estimation error is improved in this case.
\begin{thm}
\label{thm: estimation error of delta for mismatched delta}Assume
that $d\leq n$. 
Let $ \theta_*\in\mathbb{R}^d $ and let $ \theta_\sharp $ be an estimate of $ \theta_* $. 
Set $ \hat{\rho}\equiv\hat{\rho}_{\text{\emph{corr}}}(X_1^n;\theta_\sharp) $ and $ \hat{\delta} = \frac{1}{2}(1-\hat{\rho}) $. 
Then, it holds with probability $1-\frac{8}{n}$
that 
\begin{align}
\left|\hat{\delta}-\delta\right| & =\frac{1}{2}\left|\hat{\rho}-\rho\right| 
 \leq\frac{\left|\|\theta_{*}\|^{2}-\|\theta_{\sharp}\|^{2}\right|}{\|\theta_{\sharp}\|^{2}}+16\log(n)\left[\sqrt{\frac{\delta}{n}}+\frac{1}{\|\theta_{\sharp}\|}\sqrt{\frac{1}{n}}+\frac{1}{\|\theta_{\sharp}\|^{2}}\sqrt{\frac{d}{n}}\right].\label{eq: high probability bound on delta estimation}
\end{align}
\end{thm}

The proof of Theorem \ref{thm: estimation error of delta for mismatched delta} appears in Appendix \ref{app:proof-achievability-estimate-delta}. 
Note that Theorem \ref{thm: estimation error of delta for mismatched delta}
states a high-probability bound, suitable to its usage later on in
Section \ref{sec:Mean-estimation-under}. A bound on the expectation
of the error can be obtained by the standard method of integrating tails.

\paragraph*{The effect of knowledge of $\theta_{*}$}

If $ \theta_* $ is known up to a sign, i.e., $ \theta_\sharp = \pm\theta_* $, then for the purpose of $ \rho $ (or equivalently $ \delta $) estimation, the model \eqref{eq: Gaussian Markov model} can be reduced to a one-dimensional model by rotational invariance of isotropic Gaussian (See additional details in Appendix \ref{app:proof-converse-converse-estimate-delta}). It then immediately follows from Theorem \ref{thm: estimation error of delta for mismatched delta} that:
\begin{cor}
\label{cor: estimation error of delta for mismatched delta known mean}Assume that $ d\le n $, 
$\|\theta_{*}\|\leq1$ and $\theta_{\sharp}=\pm\theta_{*}$. 
Let $ U_1^n $ be defined in as $ U_i \dfn \|\theta_*\|\cdot S_i+W_i $ where $ W_i\sim N(0,1) $ i.i.d., $ \hat{\rho}\equiv\hat{\rho}_{\text{\emph{corr}}}(U_1^n;\theta_\sharp) $ and $ \hat{\delta} = \frac{1}{2}(1-\hat{\rho}) $. 
Then
it holds with probability $1-\frac{8}{n}$ that 
\begin{equation}
\left|\hat{\delta}-\delta\right|=\frac{1}{2}\left|\hat{\rho}-\rho\right|\leq\frac{18\log(n)}{\|\theta_{*}\|^{2}}\sqrt{\frac{1}{n}}.\label{eq: high probability bound on delta estimation known mean}
\end{equation}
\end{cor}

Numerical validation of the performance of the estimator $ \hat{\delta}_{\text{corr}}(X_1^n;\theta_\sharp) = \frac{1}{2}(1 - \hat{\rho}_{\text{corr}}(X_1^n;\theta_\sharp)) $ in the mismatched (Theorem \ref{thm: estimation error of delta for mismatched delta}) and matched (Corollary \ref{cor: estimation error of delta for mismatched delta known mean}) cases is provided 
in Appendix \ref{app:numerical-validation}. 

We next consider an impossibility lower bound. 
\begin{prop}
\label{prop: impossibility lower bound for flip probability estimation}Suppose
that $\theta_{\sharp}=\theta_{*}$ and $ \|\theta_*\|\leq\frac{1}{\sqrt{2}} $. 
Then
\[
\inf\limits_{\hat{\delta}(U_{1}^{n})}\sup\limits_{\delta\in[0,1]}\E[|\delta-\hat{\delta}(U_{1}^{n})|]\geq\frac{1}{32\sqrt{n}},
\]
where the infimum is over any estimator $\hat{\delta}(U_{1}^{n})$ based on the model $ U_i = \|\theta_*\| S_i + W_i $ where each $ W_i $ is i.i.d.\ $ N(0,1) $.
\end{prop}

The proof of Proposition \ref{prop: impossibility lower bound for flip probability estimation} is presented in Appendix \ref{app:proof-converse-converse-estimate-delta}. 
According to Corollary \ref{cor: estimation error of delta for mismatched delta known mean} and Proposition \ref{prop: impossibility lower bound for flip probability estimation}, in estimating $ \delta $ with a known $ \theta_* $, though the dependence $ \Theta(\frac{1}{\sqrt{n}}) $ of the minimax error rate on the sample size is shown to be nearly optimal, it is unclear what the optimal dependence on the signal strength should be. 
This is left as an interesting open question and we discuss the challenges associated with this problem in Appendix \ref{app:proof-converse-converse-estimate-delta}.



\section{Mean estimation under an unknown flip probability \label{sec:Mean-estimation-under}}

As we have seen, if an estimator for $\theta_{*}$ knows the value
of $\delta$, and if both $\delta\leq\frac{1}{2}$ and $d\leq\delta n$
hold, then the estimator can achieve improved error rates over the
GMM case ($\delta=\frac{1}{2}$). In this section, we assume that
both $\theta_{*}$ and $\delta$ are unknown, and so the estimator
is required to estimate $\delta$ in order to use this knowledge for
an estimator of $\theta_{*}$. We propose an estimation procedure
of three steps based on sample splitting of $3n$ samples. We mention
at the outset that the regime in which improvement is possible will
be for low signal strength $\|\theta_{*}\|\lesssim1$ (low separation
between the components), and up to a dimension which depends on $\delta$.
Of course the estimation procedure does not know $(\theta_{*},\delta)$
in advance, and so it is required to identify if $(\theta_{*},\delta)$
are in this regime during its operation.

We now begin with an overview of the steps of the estimation algorithm.
At Step A, the algorithm estimates $\theta_{*}$ based on $X_{1}^{n}$
assuming a Gaussian mixture model $(\delta=\frac{1}{2})$ to obtain
an estimate $\hat{\theta}^{(A)}$. Then, based on $\|\hat{\theta}^{(A)}\|$,
the algorithm decides whether improvement is potentially possible
had $\delta$ was known. There are two cases. The first case is that
$\|\hat{\theta}^{(A)}\|$ is too low, and then its estimate is not
sufficiently accurate to be used in the next steps. Essentially, this
happens when the norm is below the global minimax rate $(\frac{d}{n})^{1/4}$,
and the estimation error of the norm on the same scale as the norm of $\|\theta_{*}\|$.
A trivial estimator of $\hat{\theta}=0$ is then optimal in terms
of error rates. It can be already noted at this step that while the global minimax
rate for the known $\delta$ case is $(\frac{\delta d}{n})^{1/4}$,
here the algorithm already stops and estimates $\hat{\theta}=0$ even if just $\|\theta_{*}\|\lesssim(\frac{d}{n})^{1/4}$, leading to larger global minimax rate.
The second case is that $\|\theta_{*}\|$ is larger than a constant.
In this case, the estimation based on a GMM already achieves the optimal
parametric $O(\sqrt{\frac{d}{n}})$ error rate of the Gaussian location
model, and so no further estimation steps are necessary. Otherwise,
an improvement in the estimation is possible. The algorithm proceeds
to Step B, and uses $X_{n+1}^{2n}$ to obtain an estimate $\hat{\delta}^{(B)}$
of $\delta$ based on the mismatched $\theta_{\sharp}\equiv\hat{\theta}^{(A)}$.
Then, based on the estimate $\hat{\delta}^{(B)}$ the algorithm decides
whether the accuracy of $\hat{\delta}^{(B)}$ is sufficient to be
used in an refined estimation of $\theta_{*}$. If the accuracy of
$\hat{\delta}^{(B)}$ is not good enough, then the algorithm outputs
the estimate from Step A, that is $\hat{\theta}^{(A)}$. Otherwise,
it proceeds to Step C, in which $\theta_{*}$ is re-estimated using
$X_{2n+1}^{3n}$, based on a mismatched choice of $k$, that is $k\asymp\frac{1}{\hat{\delta}^{(B)}}$
instead of $k=\frac{1}{8\delta}$. Intuitively, the estimated value
$\hat{\delta}^{(B)}$ should be larger than $\delta$ so the resulting
block size $k\asymp\frac{1}{\hat{\delta}^{(B)}}$ will be such that
the gain in the block is still close to $1$ with high probability.
On the other hand, it is desired that $\hat{\delta}^{(B)}$ will be
on the same scale as $\delta$ so that the estimation rate (\ref{eq: mean estimation error for joint estimation -- low dimension}) (see also \eqref{eq: minimax loss at low dimension Markov contribution statement}) -- which now essentially holds with $\hat{\delta}^{(B)}$ instead of
$\delta$ -- would be as small as possible. Thus, if the algorithm
has assured in Step B that $\hat{\delta}^{(B)}\asymp\delta$, then
at Step C it will achieve the error rate indicated in (\ref{eq: mean estimation error for joint estimation -- low dimension}).

\begin{algorithm}
\caption{Mean estimation for unknown $\delta$ \label{alg:Mean estimation with unknown delta}}

\begin{algorithmic}[1]

\State \textbf{input:} Parameters $\lambda_{\theta},\lambda_{\delta}>0$
(from (\ref{eq: mean estimation error for joint estimation -- low dimension})
(\ref{eq: mean estimation error for joint estimation -- high dimension})
(\ref{eq: delta estimation error for joint estimation})), $3n$ data
samples $X_{1}^{3n}$ from the model (\ref{eq: Gaussian Markov model})

\State \textbf{step A:} Estimate $\theta_{*}$ assuming a Gaussian
mixture model: 
\[\hat{\theta}^{(A)}\equiv\hat{\theta}_{\text{cov}}(X_{1}^{n};k=1)\]

\If{ $\|\hat{\theta}^{(A)}\|\leq2\lambda_{\theta}\cdot\log(n)\cdot(\frac{d}{n})^{1/4}$} 

\State\Return $\hat{\theta}=0$ \Comment{ No further improvement
can be guaranteed}

\ElsIf { $\|\hat{\theta}^{(A)}\|\geq\frac{1}{2}$ } 

\State\Return $\hat{\theta}=\hat{\theta}^{(A)}$ \Comment{
No further improvement is possible}

\EndIf

\State \textbf{step B:} Estimate $\delta$ assuming a mismatched
mean value $\hat{\theta}^{(A)}$: 
\[\hat{\delta}^{(B)}=\hat{\delta}_{\text{corr}}(X_{n+1}^{2n};\hat{\theta}^{(A)})\]

\If{ $\hat{\delta}^{(B)}\leq64\lambda_{\delta}\lambda_{\theta}\frac{\log(n)}{\|\hat{\theta}^{(A)}\|^{2}}\sqrt{\frac{d}{n}}$} 

\State\Return $\hat{\theta}=\hat{\theta}^{(A)}$ \Comment{
No further improvement can be guaranteed}

\EndIf

\State \textbf{step C}: Estimate $\theta_{*}$ assuming a mismatched flip probability $\hat{\delta}^{(B)}$: 
\[\hat{\theta}^{(C)}=\hat{\theta}_{\text{cov}}\left(X_{2n+1}^{3n};k=\dfrac{1}{16\hat{\delta}^{(B)}}\right)\]

\State \textbf{return }$\hat{\theta}=\hat{\theta}^{(C)}$.

\end{algorithmic}

\end{algorithm}

The formal description of the estimation algorithm 
is provided in Algorithm \ref{alg:Mean estimation with unknown delta}. 
We remark that refining the estimation of $\delta$ can be
easily incorporated as a fourth step of this algorithm, but we do
not present this in order to keep the statement of the result simple.
The error of the estimator output by Algorithm \ref{alg:Mean estimation with unknown delta}
is as follows:
\begin{thm}
\label{thm: Estimation error for algorithm}There exist constants $c_{1},c_{2}\geq0$ and $ \lambda_\theta,\lambda_\delta\ge1 $ such that if $d\leq\frac{n}{4\lambda_{\theta}^2\log^{2}(n)\wedge16}$ then the output $\hat{\theta}$
of Algorithm \ref{alg:Mean estimation with unknown delta} satisfies for any $ \theta_*\in\mathbb{R}^d $, with probability
$1-O(\frac{1}{n})$: \\
If $d\leq\frac{1}{64\lambda_{\delta}^2\lambda_{\theta}^2\log^{2}(n)}\delta^{4}n$
then 
\begin{equation}
\loss(\hat{\theta},\theta_{*})\leq c_{1}\log n\cdot\begin{cases}
\|\theta_{*}\|, & \|\theta_{*}\|\leq\lambda_{\theta}\log(n)\left(\frac{d}{n}\right)^{1/4}\\
\frac{1}{\|\theta_{*}\|}\sqrt{\frac{d}{n}}, & \lambda_{\theta}\log(n)\left(\frac{d}{n}\right)^{1/4}\leq\|\theta_{*}\|\leq\sqrt{8\lambda_{\delta}\lambda_{\theta}\log(n)}\left(\frac{d}{\delta^{2}n}\right)^{1/4}\\
\frac{1}{\|\theta_{*}\|}\sqrt{\frac{\delta d}{n}}, & \sqrt{8\lambda_{\delta}\lambda_{\theta}\log(n)}\left(\frac{d}{\delta^{2}n}\right)^{1/4}\leq\|\theta_{*}\|\leq\sqrt{\delta}\\
\sqrt{\frac{d}{n}}, & \|\theta_{*}\|\geq\sqrt{\delta}
\end{cases};\label{eq: high probability mean estimation error joint very very low dimension}
\end{equation}
If $d\geq\frac{1}{64\lambda_{\delta}^2\lambda_{\theta}^2\log^{2}(n)}\delta^{4}n$
then 
\begin{equation}
\loss(\hat{\theta},\theta_{*})\leq c_{2}\log n\cdot\begin{cases}
\|\theta_{*}\|, & \|\theta_{*}\|\leq\lambda_{\theta}\log(n)\left(\frac{d}{n}\right)^{1/4}\\
\frac{1}{\|\theta_{*}\|}\sqrt{\frac{d}{n}}, & \lambda_{\theta}\log(n)\left(\frac{d}{n}\right)^{1/4}\leq\|\theta_{*}\|\leq\sqrt{8\lambda_{\delta}\lambda_{\theta}\log(n)}\left(\frac{d}{\delta^{2}n}\right)^{1/4}\\
\sqrt{\frac{d}{n}}, & \|\theta_{*}\|\geq\sqrt{8\lambda_{\delta}\lambda_{\theta}\log(n)}\left(\frac{d}{\delta^{2}n}\right)^{1/4}
\end{cases}.\label{eq: high probability mean estimation error joint very low dimension}
\end{equation}
\end{thm}

Theorem \ref{thm: Estimation error for algorithm} implies that up to logarithmic factors, the error rates of the known $\delta$ case are recovered for low enough dimension
$d\lesssim\delta^{4}n$ and $\|\theta_{*}\|\gtrsim\left(\frac{d}{\delta^{2}n}\right)^{1/4}$. The analysis of Algorithm \ref{alg:Mean estimation with unknown delta} appears in Appendix \ref{sec: proofs for joint estimation}. Numerical validation of the performance of Algorithm \ref{alg:Mean estimation with unknown delta} can be found 
in Appendix \ref{app:numerical-validation}.

\paragraph*{The impact of unknown $\delta$ on the estimation error of $ \theta_* $}

Comparing Theorems \ref{thm: impossibility lower bound for mean estimation} (known $\delta$) and \ref{thm: Estimation error for algorithm} (unknown $\delta$) reveals the deterioration in the estimation of $\theta_*$ due to the lack
of knowledge of $\delta$ from the three aspects mentioned in Section \ref{sec: contributions} (ignoring logarithmic
factors): First, the global minimax rate is $O(\frac{d}{n})^{1/4}$ as
for the GMM, instead of the rate $O(\frac{\delta d}{n})^{1/4}$ for
the Markov model case with known $\delta$. Second, at the regime
$(\frac{d}{n})^{1/4}\lesssim\|\theta_{*}\|\lesssim(\frac{d}{\delta^{2}n})^{1/4}$
the error rate is $O(\frac{1}{\|\theta_{*}\|}\sqrt{\frac{d}{n}})$
instead of the lower $O(\frac{1}{\|\theta_{*}\|}\sqrt{\frac{\delta d}{n}})$. 
Third, the algorithm is only effective when the dimension is as
low as $d\lesssim\delta^{4}n$. For higher dimensions, the rates of
the GMM are achieved, which can be achieved even without the knowledge
of $\delta$. 



\section{Conclusion and future work}

In this paper, we have considered an elementary, yet fundamental, high-dimensional
model with memory. We have obtained a sharp bound on the minimax rate
of estimation in case the underlying statistical dependency (flip
probability) is known, and proposed a three-step estimation algorithm
when it is unknown. This has revealed the gains possible in estimation
rates due to the memory between the samples, and smoothly interpolated
between the extreme cases of GLM and GMM. An interesting open problem is to either characterize the optimality
of the algorithm or improving in the unknown $\delta$ case, which requires understanding optimal
estimation of the flip probability.

Naturally, as the model considered in this paper is basic, there is
an ample of possibilities to generalize this model. These include,
a larger number of components in the mixture, statistical dependency
with a more complicated graphical structure between the data samples,
existence of nuisance parameters such as the noise variance, sharp
finite-sample/finite-iteration analysis of specific practical algorithms
such as Baum-Welch, location-scale model with anisotropic noise, heavy-tailed noise, and so on.


\begin{ack}

Part of this work was done when YZ was a postdoc at Technion where he received funding from the European Union’s Horizon 2020 research and innovation programme under grant agreement No 682203-ERC-[Inf-Speed-Tradeoff].
The work of of NW was supported in part by the Israel Science Foundation (ISF) under Grant 1782/22.
NW is grateful to Guy Bresler for introducing him to this problem, for the initial ideas that led to this research, and for many helpful discussions on the topic.

\end{ack}

\newpage

\bibliographystyle{plainnat}
\bibliography{HM_GM}

\section*{Checklist}


\begin{enumerate}

\item For all authors...
\begin{enumerate}
  \item Do the main claims made in the abstract and introduction accurately reflect the paper's contributions and scope?
    \answerYes{}
  \item Did you describe the limitations of your work?
    \answerYes{}
  \item Did you discuss any potential negative societal impacts of your work?
    \answerNA{}
  \item Have you read the ethics review guidelines and ensured that your paper conforms to them?
    \answerYes{}
\end{enumerate}

\item If you are including theoretical results...
\begin{enumerate}
  \item Did you state the full set of assumptions of all theoretical results?
    \answerYes{}
        \item Did you include complete proofs of all theoretical results?
    \answerYes{}
\end{enumerate}

\item If you ran experiments...
\begin{enumerate}
  \item Did you include the code, data, and instructions needed to reproduce the main experimental results (either in the supplemental material or as a URL)?
    \answerNA{}
  \item Did you specify all the training details (e.g., data splits, hyperparameters, how they were chosen)?
    \answerNA{}
        \item Did you report error bars (e.g., with respect to the random seed after running experiments multiple times)?
    \answerNA{}
        \item Did you include the total amount of compute and the type of resources used (e.g., type of GPUs, internal cluster, or cloud provider)?
    \answerNA{}
\end{enumerate}

\item If you are using existing assets (e.g., code, data, models) or curating/releasing new assets...
\begin{enumerate}
  \item If your work uses existing assets, did you cite the creators?
    \answerNA{}
  \item Did you mention the license of the assets?
    \answerNA{}
  \item Did you include any new assets either in the supplemental material or as a URL?
    \answerNA{}
  \item Did you discuss whether and how consent was obtained from people whose data you're using/curating?
    \answerNA{}
  \item Did you discuss whether the data you are using/curating contains personally identifiable information or offensive content?
    \answerNA{}
\end{enumerate}

\item If you used crowdsourcing or conducted research with human subjects...
\begin{enumerate}
  \item Did you include the full text of instructions given to participants and screenshots, if applicable?
    \answerNA{}
  \item Did you describe any potential participant risks, with links to Institutional Review Board (IRB) approvals, if applicable?
    \answerNA{}
  \item Did you include the estimated hourly wage paid to participants and the total amount spent on participant compensation?
    \answerNA{}
\end{enumerate}

\end{enumerate}


\newpage{}

\appendix

In Appendix \ref{app:heuristic} we provide heuristic justification for the scaling of the optimal minimax error rate in \eqref{eq: minimax loss at low dimension Markov contribution statement}. 
In Appendix \ref{sec:Proofs for Mean estimation} we provide the proofs
for Theorems \ref{thm: mean estimation known delta upper bound} and \ref{thm: impossibility lower bound for mean estimation}.
In Appendix \ref{sec:Proofs for delta estimation} we provide
the proofs for Theorem \ref{thm: estimation error of delta for mismatched delta} and Proposition \ref{prop: impossibility lower bound for flip probability estimation}. 
In Appendix \ref{sec: proofs for joint estimation}
we provide the proofs for Theorem \ref{thm: Estimation error for algorithm}. 
In Appendix \ref{sec:Useful-mathematical-tools}
we include some useful results for the sake of completeness. 

\paragraph*{Additional notation}
For a matrix $A\in\mathbb{R}^{d_{1}\times d_{2}}$, $\|A\|_{\text{op}}$
is the operator norm (with respect to Euclidean norms), and $\|A\|_{F}$ is the Frobenius norm of $A$.
$\mathbb{S}^{d-1}\dfn\{\theta\in\mathbb{R}^{d}\colon\|\theta\|=1\}$
is the unit sphere and $\mathbb{B}^{d}\dfn\{\theta\in\mathbb{R}^{d}\colon\|\theta\|\leq1\}$
is the unit ball in the $d$-dimensional Euclidean space.
For a pair
of probability distributions $P$ and $Q$ on a common alphabet ${\cal X}$
with densities $p$ and $q$ w.r.t. to a base measure $\nu$, we denote
the total variation distance by $\dtv(P,Q)\dfn\frac{1}{2}\int|p-q|\d\nu$,
the Kullback-Leibler (KL) divergence by $\Dkl(P\mid\mid Q)\dfn\int p\log\frac{p}{q}\d\nu$,
and the chi-square divergence by $\Dchis(P\mid\mid Q)\dfn\int\frac{(p-q)^{2}}{q}\d\nu=\int\frac{p^{2}}{q}\d\nu-1$.

\section{Heuristic justification of minimax rate \eqref{eq: minimax loss at low dimension Markov contribution statement}}
\label{app:heuristic}

The main intuition behind the HMM considered in this paper comes from the correlation decay phenomenon in graphical model. 
Indeed, one can view the Markov chain $ X_1\to X_2\to\cdots\to X_n $ sampled from the model \eqref{eq: Gaussian Markov model} as a simple graphical model whose conditional independence structure is expressed by a line graph on $n$ nodes. 
Each node $ X_i $ is conditionally independent of other samples given its neighbors $ X_{i-1} $ and $ X_{i+1} $. 
Furthermore, a pair of nodes $ X_i $ and $ X_j $ of large graph distance (i.e., $ |i - j| $) is approximately independent. 
Informally, we expect that there is one sign flip (i.e., $ S_i = -S_{i+1} $) per $ \approx\frac{1}{\delta} $ samples. 
Therefore, signs $ S_i $ and $ S_j $ are likely to have the same value if $ |i - j|\lesssim\frac{1}{\delta} $ and are approximately independent if $ |i - j|\gtrsim\frac{1}{\delta} $. 
This is the leading guideline for the design of our estimator \eqref{eq: general PCA estimation rule} for upper bound and the reduction to the genie-aided model \eqref{eq: Gaussian Markov model genie} for lower bound. 

Recall that for the Gaussian location model 
\begin{align}
X = \theta_* + Z \label{eqn:glm-model}
\end{align}
with $ Z\sim N(0,I_d) $, the minimax rate is given by
\begin{align}
\M_{\text{GLM}}(n,d,t) &\asymp t\wedge\sqrt{\frac{d}{n}} . \label{eqn:glm-rate}
\end{align}
Note that this is a compact way of writing the rate \eqref{eq: minimax rates -- Gaussian location model}. 
Given $ n $ i.i.d.\ samples $ X_1^n $ from model \eqref{eqn:glm-model}, the estimator  
\begin{align}
\hat{\theta} = \frac{1}{n}\sum_{i=1}^nX_i \label{eqn:glm-estimator}
\end{align}
achieves the rate \eqref{eqn:glm-rate}. 
For the Gaussian mixture model $ X = S\theta_* + Z $ with $ S\sim\text{Unif}\{-1,1\} $ and $ Z\sim N(0,I_d) $, the minimax rate \eqref{eq: minimax rates Gaussian mixture low dimension} \citep[Appendix B]{wu2019EM} can be compactly written as
\begin{align}
\M_{\text{GMM}}(n,d,t) &\asymp \left[\frac{1}{t}\left(\sqrt{\frac{d}{n}} + \frac{d}{n}\right) + \sqrt{\frac{d}{n}}\right] \wedge t . \label{eqn:gmm-rate}
\end{align}
The above rate is attained by the estimator 
\begin{align}
\hat{\theta} = \sqrt{(\lambda_{\max}(\hat{\Sigma}) - 1)_+}\cdot v_{\max}(\hat{\Sigma}) \label{eqn:gmm-estimator}
\end{align}
where $ \hat{\Sigma} = \frac{1}{n}\sum_{i=1}^nX_iX_i^\top $. 
More generally, we can also consider the GMM 
\begin{align}
\tilde{X} = S\theta_* + \tilde{Z} \label{eqn:gmm-general}
\end{align}
with $ S\sim\text{Unif}\{-1,1\} $ and $ \tilde{Z}\sim N(0,\sigma^2\cdot I_d) $, for some $ \sigma>0 $, which is equivalent in distribution to $ \tilde{X} = \sigma(S\theta_*/\sigma + Z) $ where $ Z\sim N(0,I_d) $. 
Generalizing \eqref{eqn:gmm-rate}, it is straightforward to check that the optimal minimax rate for model \eqref{eqn:gmm-general} is given by
\begin{align}
\tilde{\M}_{\text{GMM}}(n,d,t,\sigma) &\asymp \sigma \left[ \left(\frac{1}{t/\sigma}\left(\sqrt{\frac{d}{n}} + \frac{d}{n}\right) + \sqrt{\frac{d}{n}}\right) \wedge \frac{t}{\sigma} \right] \\
&= \left[ \frac{\sigma^2}{t}\left(\sqrt{\frac{d}{n}} + \frac{d}{n}\right) + \sigma\sqrt{\frac{d}{n}} \right] \wedge t . \label{eqn:rate-gmm-variance}
\end{align}
Now, for the HMM at hand, the rationale is that the original model \eqref{eq: Gaussian Markov model} is equivalent to 
$\tilde{X} = \overline{S}\cdot\overline{X}$, 
where $ \overline{S}\sim\text{Unif}\{-1,1\} $ is the coherent sign of a block of $\frac{1}{\delta}$ raw samples from model \eqref{eq: Gaussian Markov model}, 
$ \overline{X} = \frac{1}{1/\delta} \sum_{i=1}^{1/\delta} X_i $ represents the block-sample obtained by applying the estimator \eqref{eqn:glm-estimator} to a block of $ \frac{1}{\delta} $ raw samples, hence, 
$\overline{X} \eqd \theta_* + N(0,\delta\cdot I_d) $. 
Furthermore, the signs across different blocks are essentially independent. 
Therefore
\begin{align}
\tilde{X} \eqd \overline{S}\theta_* + \overline{Z} \label{eqn:hmm-effective}
\end{align}
where $ \overline{Z}\sim N(0,\delta\cdot I_d) $, and we have $ \frac{n}{1/\delta} = n\delta $ i.i.d.\ block-samples from model \eqref{eqn:hmm-effective}. 
According to \eqref{eqn:rate-gmm-variance}, the optimal minimax rate of the HMM \eqref{eq: Gaussian Markov model} should be 
\begin{align}
\M_{\text{HMM}}(n,d,\delta,t) &\asymp \tilde{\M}_{\text{GMM}}(n\delta,d,t,\sqrt{\delta}) \\
&\asymp \left[ \frac{\delta}{t}\left(\sqrt{\frac{d}{n\delta}} + \frac{d}{n\delta}\right) + \sqrt{\delta}\sqrt{\frac{d}{n\delta}} \right] \wedge t \\
&= \left[ \frac{1}{t}\left(\sqrt{\frac{d\delta}{n}} + \frac{d}{n}\right) + \sqrt{\frac{d}{n}} \right] \wedge t . \label{eqn:hmm-rate-heuristic}
\end{align}
Evaluating \eqref{eqn:hmm-rate-heuristic} immediately yields \eqref{eq: minimax loss at low dimension Markov contribution statement}. 

The optimal minimax error rates for GLM (cf.\ \eqref{eq: minimax rates -- Gaussian location model}), GMM (cf.\ \eqref{eq: minimax rates Gaussian mixture low dimension}, originally proved in \citep[Appendix B]{wu2019EM}) and HMM (cf.\ \eqref{eq: minimax loss at low dimension Markov contribution statement}, implied by our results Theorems \ref{thm: mean estimation known delta upper bound} and \ref{thm: impossibility lower bound for mean estimation}) are plotted in Figure \ref{fig:rate}. 
\begin{figure}[htbp]
  \centering
  \includegraphics[width=0.45\textwidth]{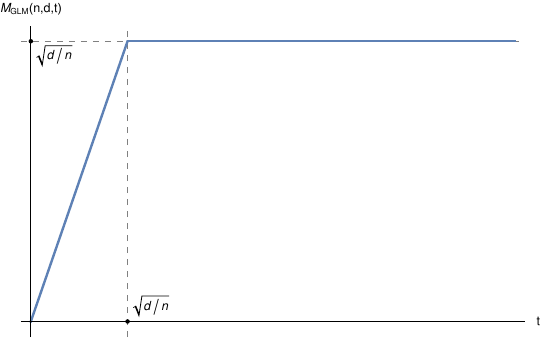} \;
  \includegraphics[width=0.45\textwidth]{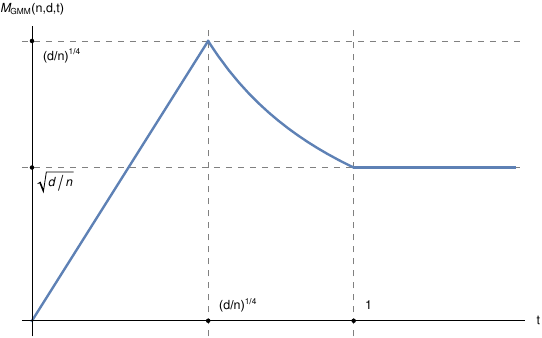} \\ 
  \includegraphics[width=0.45\textwidth]{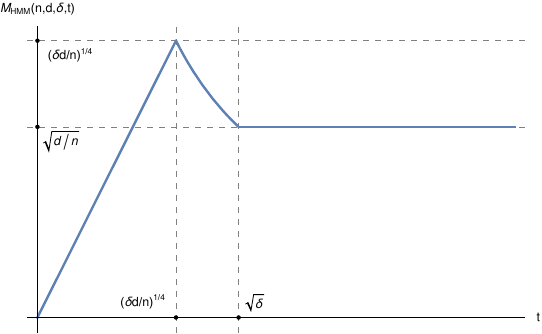} \;
  \includegraphics[width=0.45\textwidth]{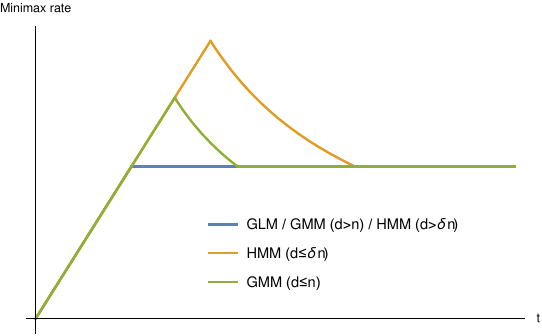} 
  \caption{Plots of minimax rates of the Gaussian Location Model, the Gaussian Mixture Model and the Hidden Markov Model. 
  The top left, top right and bottom left figures show the scaling of the minimax rates of GLM (cf.\ \eqref{eq: minimax rates -- Gaussian location model}), GMM with $ d\le n $ (cf.\ \eqref{eq: minimax rates Gaussian mixture low dimension}) and HMM with $ d\le \delta n $ (cf.\ \eqref{eq: minimax loss at low dimension Markov contribution statement}), respectively. 
  All the above error rates are plotted in the bottom right figure in which one can clearly see how the rate varies as $\delta$  varies. }
  \label{fig:rate}
\end{figure}

\section{Proofs for Section \ref{sec:Mean-estimation-for}: Mean estimation
for known $\delta$ \label{sec:Proofs for Mean estimation}}

\subsection{Proof of Theorem \ref{thm: mean estimation known delta upper bound}:
Analysis of the estimator}
\label{app:proof-achievability-estimate-theta}

Our proposed procedure in Theorem \ref{thm: mean estimation known delta upper bound} for $ \theta_* $ estimation with a known flip probability $ \delta $ is depicted in Figure \ref{fig:block}. 
\begin{figure}[htbp]
  \centering
  \includegraphics[width=0.95\textwidth]{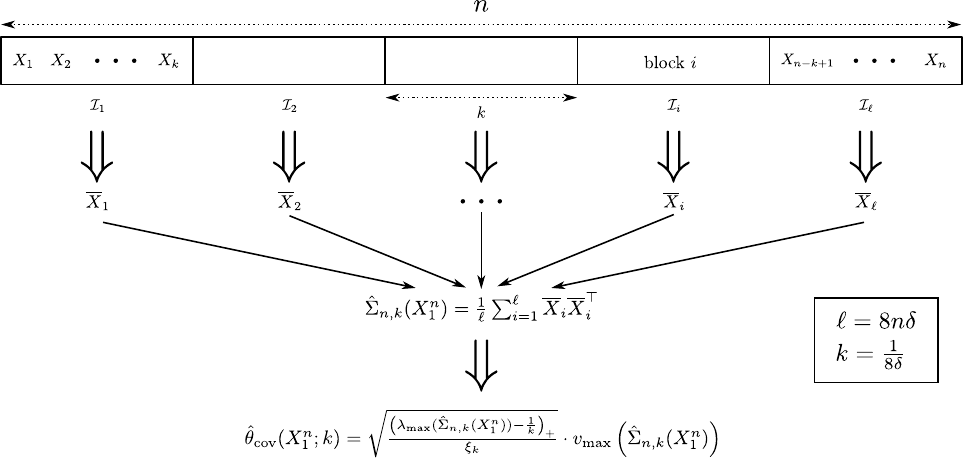}
  \caption{A visual illustration of the construction of the estimator $ \hat{\theta}_{\text{cov}}(X_1^n;k) $ in \eqref{eq: general PCA estimation rule} given $n$ samples $ X_1^n $ from the model \eqref{eq: Gaussian Markov model} with a known flip probability $ \delta $. The block size $k$ is chosen to be $ \frac{1}{8\delta} $. }
  \label{fig:block}
\end{figure}

\begin{rem}
\label{rk:remark-delta-at-least-1-over-n}
The condition $\delta\geq\frac{1}{n}$ in Theorem \ref{thm: mean estimation known delta upper bound}
is inconsequential. Indeed, if $\delta\leq\frac{1}{n}$, then also
$d\geq\delta n$ and we are at the high dimension regime. In that
case, the estimator can artificially increase the flip probability
(by sign randomization) to $\frac{1}{n}$, which results the local
minimax rate of the GLM, that is, $\M(n,d,\delta,t)\asymp\M_{\text{GLM}}(n,d,t)$,
which is optimal.
We therefore assume $ \delta\le\frac{1}{n} $ throughout the proof. 
\end{rem}

To begin with the analysis of the estimator in Figure \ref{fig:block}, the following lemma is a simple, yet key tool for the proof. It
establishes the variance of the random gain $\overline{S}$. The proof
relies on a sort of \emph{self-bounding} property (cf.\ \citep[Chapter 3.3]{boucheron2013concentration}), which stems from
the fact that $|\overline{S}|\leq1$. 
\begin{lem}
\label{lem:analysis of the random gain}It holds that $0\leq\overline{S}^{2}\leq1$
(with probability $1$) and that 
\[
\V(\overline{S}^{2})\leq\E[1-\overline{S}^{2}]\leq4\delta k.
\]
\end{lem}

\begin{proof}
Trivially $0\leq\overline{S}^{2}\leq1$ and so the variance is bounded
as
\[
\V(\overline{S}^{2})=\V(1-\overline{S}^{2})\leq\E[(1-\overline{S}^{2})^{2}]\leq\E[1-\overline{S}^{2}].
\]
The expected value is then upper bounded as 
\begin{align}
\E[1-\overline{S}^{2}] & =\E\left[1-\left(\frac{1}{k}\sum_{j=1}^{k}S_{j}\right)^{2}\right]\\
 & =\E\left[1-\left(\frac{1}{k}\sum_{j=1}^{k}S_{j}\right)^{2}\,\middle|\, S_{0}=1\right]\\
 & =\E\left[\left(1-\left(\frac{1}{k}\sum_{j=1}^{k}S_{j}\right)\right)\left(1+\left(\frac{1}{k}\sum_{j=1}^{k}S_{j}\right)\right)\,\middle|\, S_{0}=1\right]\\
 & \leq2\cdot\E\left[1-\left(\frac{1}{k}\sum_{j=1}^{k}S_{j}\right)\,\middle|\, S_{0}=1\right]\\
 & =2-\frac{2}{k}\sum_{j=1}^{k}\E\left[S_{j}\,\middle|\, S_{0}=1\right]\\
 & \trre[=,a]2-\frac{2}{k}\sum_{j=1}^{k}\rho^{j}\\
 & =2-\frac{2}{k}\sum_{j=1}^{k}(1-2\delta)^{j}\\
 & \trre[\leq,b]2-\frac{2}{k}\sum_{j=1}^{k}(1-2j\delta)\\
 & =\frac{4\delta}{k}\sum_{j=1}^{k}j\\
 & =\frac{2\delta(k^{2}+k)}{k}\\
 & \leq4\delta k,\label{eq: bound on expected one minus gain}
\end{align}
where $(a)$ follows from 
\begin{align}
 & \E[S_{j}\mid S_{0}=1]\nonumber \\
 & =\P[S_{j-1}=1\mid S_{0}=1]\E[S_{j}\mid S_{0}=1,S_{j-1}=1]\nonumber \\
 & \hphantom{==}+\P[S_{j-1}=-1\mid S_{0}=1]\E[S_{j}\mid S_{0}=1,S_{j-1}=-1]\\
 & =\P[S_{j-1}=1\mid S_{0}=1]\E[S_{j}\mid S_{j-1}=1]+\P[S_{j-1}=-1\mid S_{0}=1]\E[S_{j}\mid S_{j-1}=-1]\\
 & =\P[S_{j-1}=1\mid S_{0}=1]\left(\frac{1+\rho}{2}-\frac{1-\rho}{2}\right)+\P[S_{j-1}=-1\mid S_{0}=1]\left(-\frac{1+\rho}{2}+\frac{1-\rho}{2}\right)\\
 & =\P[S_{j-1}=1\mid S_{0}=1]\cdot\rho-\P[S_{j-1}=-1\mid S_{0}=1]\cdot\rho\\
 & =\rho\cdot\E[S_{j-1}\mid S_{0}=1]\\
 & =\cdots=\rho^{j},
\end{align}
$(b)$ follows from Bernoulli's inequality $(1-x)^{r}\geq1-rx$ for
$x\in[0,1]$ and $r\geq1$. 
\end{proof}
The next lemma summarizes concentration results and bounds on the
expected value of various empirical quantities needed for the rest
of the analysis. For a sequence of samples $(S_{1}^{n},Z_{1}^{n})$,
let us denote the following events: 
\[
{\cal E}_{n,\delta,k}^{'}\dfn\left\{ S_{1}^{n}\colon\left|\frac{1}{\ell}\sum_{i=1}^{\ell}\overline{S}_{i}^{2}-\xi_{k}\right|\leq5\sqrt{\frac{\delta k^{2}\log(n)}{n}}\right\} ,
\]

\[
{\cal E}_{n,d,k}^{''}\dfn\left\{ (S_{1}^{n},Z_{1}^{n})\colon\left\Vert \frac{1}{\ell}\sum_{i=1}^{\ell}\overline{S}_{i}\overline{Z}_{i}\right\Vert \leq7\sqrt{\frac{d}{n}}\right\} ,
\]

\[
{\cal E}_{n,d,k}^{'''}\dfn\left\{ Z_{1}^{n}\colon\left\Vert \frac{1}{\ell}\sum_{i=1}^{\ell}\overline{Z}_{i}\overline{Z}_{i}^{\T}-\frac{1}{k}\cdot I_{d}\right\Vert _{\text{op}}\leq4\sqrt{\frac{d}{kn}}+\frac{4d}{n}\right\} , 
\]
and let 
\[
{\cal E}_{n,d,\delta,k}\dfn{\cal E}_{n,\delta,k}^{'}\cap{\cal E}_{n,d,k}^{''}\cap{\cal E}_{n,d,k}^{'''}.
\]
 Note that these events depend on the choice of the block length $k$. 
\begin{lem}
\label{lem:concentration of empirical terms}
Assume that $\delta\geq\frac{1}{n}$
and $d\leq n$. 
It holds that
$\P[{\cal E}_{n,d,\delta,k}]\geq1-\frac{5}{n}$ and 
\begin{equation}
\E\left[\left|\frac{1}{\ell}\sum_{i=1}^{\ell}\overline{S}_{i}^{2}-\xi_{k}\right|\right]\leq2\sqrt{\frac{\delta k^{2}}{n}},\label{eq: expectation of gain average over all blocks}
\end{equation}
\begin{equation}
\E\left[\left\Vert \frac{1}{\ell}\sum_{i=1}^{\ell}\overline{S}_{i}\overline{Z}_{i}\right\Vert \right]\leq\sqrt{\frac{d}{n}},\label{eq: expectation of gain times noise over all blocks}
\end{equation}
and 
\begin{equation}
\E\left[\left\Vert \frac{1}{\ell}\sum_{i=1}^{\ell}\overline{Z}_{i}\overline{Z}_{i}^{\T}-\frac{1}{k}\cdot I_{d}\right\Vert _{\text{op}}\right]\leq13\sqrt{\frac{d}{nk}}+10\frac{d}{n}.\label{eq: expectation of noise square over all blocks}
\end{equation}
\end{lem}

\begin{proof}
We analyze the expected value and concentration of each of the three
terms. 

First, the random variable $\overline{S}_{i}^{2}-\xi_{k}$ has zero
mean, it is bounded in $[-1,1]$, and its variance is bounded by $4\delta k$
according to Lemma \ref{lem:analysis of the random gain}. Since $\{\overline{S}_{i}^{2}-\xi_{k}\}_{i=1}^{\ell}$ are
i.i.d., Bernstein's inequality for bounded distributions implies that
(e.g., \eqref{eqn:bernstein} from \citet[Proposition 2.14]{wainwright2019high}, by setting in the
notation therein $X_{i}=\overline{S}_{i}^{2}-\xi_{k}$ which is a
zero-mean random variable bounded in $[-1,1]$)
\[
\P\left[\left|\frac{1}{\ell}\sum_{i=1}^{\ell}\overline{S}_{i}^{2}-\xi_{k}\right|\geq t\right]\leq2\exp\left(-\frac{\ell^{2}t^{2}}{4\ell\delta k+\ell t/3}\right).
\]
Requiring that the r.h.s. is at most $\frac{2}{n}$, and using $n=\ell k$,
this implies that 
\[
\left|\frac{1}{\ell}\sum_{i=1}^{\ell}\overline{S}_{i}^{2}-\xi_{k}\right|\leq\sqrt{\frac{8\delta k^{2}\log(n)}{n}}+\frac{\frac{3}{2}k\log(n)}{n}
\]
with probability $1-\frac{2}{n}$. Under the assumption $\delta\geq\frac{1}{n}$,
we may further upper bound
\begin{align}
\sqrt{\frac{8\delta k^{2}\log(n)}{n}}+\frac{\frac{3}{2}k\log(n)}{n} & \leq\log(n)\cdot\left[\sqrt{\frac{8\delta k^{2}}{n}}+\frac{\frac{3}{2}k}{n}\right]\\
 & \leq\log(n)\cdot\left[\sqrt{\frac{8\delta k^{2}}{n}}+\frac{3}{2}\sqrt{\frac{\delta k^{2}}{n}}\right]\\
 & \leq5\log(n)\cdot\sqrt{\frac{\delta k^{2}}{n}},
\end{align}
and this implies $\P[{\cal E}_{n,\delta,k}^{'}]\geq1-\frac{2}{n}$.
In addition, since $\overline{S}_{i}^{2}-\xi_{k}$ are zero mean i.i.d.,
Lemma \ref{lem:analysis of the random gain} implies that 
\[
\E\left[\left|\frac{1}{\ell}\sum_{i=1}^{\ell}\overline{S}_{i}^{2}-\xi_{k}\right|\right]\leq\sqrt{\V\left[\frac{1}{\ell}\sum_{i=1}^{\ell}\overline{S}_{i}^{2}-\xi_{k}\right]}\leq\sqrt{\frac{4\delta k}{\ell}}=2\sqrt{\frac{\delta k^{2}}{n}}
\]
which proves (\ref{eq: expectation of gain average over all blocks}). 

Second, conditioned on any fixed $\{\overline{S}_{i}\}_{i=1}^{\ell}$,
it holds that $\overline{S}_{i}\overline{Z}_{i}\sim N(0,\frac{\overline{S}_{i}^{2}}{k}I_{d})$.
Hence, $\frac{1}{\ell}\sum_{i=1}^{\ell}\overline{S}_{i}\overline{Z}_{i}\sim N(0,\sigma^{2}\cdot I_{d})$
with 
\[
\sigma^{2}\dfn\frac{1}{k\ell^{2}}\sum_{i=1}^{\ell}\overline{S}_{i}^{2}\leq\frac{1}{k\ell}=\frac{1}{n}.
\]
So, by standard concentration of 
norm of Gaussian (or, more generally
subGaussian) random vectors \eqref{eqn:norm-subgaussian} from \citep[Theorem 1.19]{rigollet2019high}
(or a degenerate case of \citep[Example 6.2]{wainwright2019high}),
it holds with probability $1-\epsilon$ that 
\[
\left\Vert \frac{1}{\ell}\sum_{i=1}^{\ell}\overline{S}_{i}\overline{Z}_{i}\right\Vert \leq4\sigma\sqrt{d}+2\sigma\sqrt{2\log\left(\frac{1}{\epsilon}\right)}\leq4\sqrt{\frac{d}{n}}+\sqrt{\frac{8}{n}\log\left(\frac{1}{\epsilon}\right)}.
\]
Since the r.h.s. does not depend on $\{\overline{S}_{i}\}_{i=1}^{\ell}$,
the same bound holds unconditionally. Setting $\epsilon=\frac{1}{n}$
and further upper bounding 
\[
4\sqrt{\frac{d}{n}}+\sqrt{\frac{8}{n}\log(n)}\leq\sqrt{\frac{32d+16\log(n)}{n}}\leq7\sqrt{\frac{d}{n}},
\]
using $\sqrt{a}+\sqrt{b}\leq\sqrt{2(a+b)}$ and $d\geq2\log(n)$,
results $\P[{\cal E}_{n,d,k}^{''}]\geq1-\frac{1}{n}$. In addition,
since $\{\overline{S}_{i}\overline{Z}_{i}\}_{i=1}^{\ell}$ are i.i.d.,
it holds that 
\[
\E\left[\left\Vert \frac{1}{\ell}\sum_{i=1}^{\ell}\overline{S}_{i}\overline{Z}_{i}\right\Vert \right]\leq\sqrt{\E\left[\left\Vert \frac{1}{\ell}\sum_{i=1}^{\ell}\overline{S}_{i}\overline{Z}_{i}\right\Vert ^{2}\right]}=\sqrt{d\sigma^{2}}\leq\sqrt{\frac{d}{n}}
\]
which proves (\ref{eq: expectation of gain times noise over all blocks}). 

Third, by Gaussian covariance estimation from \citep[Example 6.2]{wainwright2019high}, it holds for any
$\eta>0$ that 
\[
V\dfn k\left\Vert \frac{1}{\ell}\sum_{i=1}^{\ell}\overline{Z}_{i}\overline{Z}_{i}^{\T}-\frac{1}{k}I_{d}\right\Vert _{\text{op}}\leq\left[2\sqrt{\frac{d}{\ell}}+2\eta+\left(\sqrt{\frac{d}{\ell}}+\eta\right)^{2}\right]
\]
with probability larger than $1-2e^{-\ell\eta^{2}/2}$. 

We begin with a high probability event. Setting $\epsilon=2e^{-\ell\eta^{2}/2}$,
it holds with probability larger than $1-\epsilon$ that 
\begin{align}
\left\Vert \frac{1}{\ell}\sum_{i=1}^{\ell}\overline{Z}_{i}\overline{Z}_{i}^{\T}-\frac{1}{k}\cdot I_{d}\right\Vert _{\text{op}} & \leq2\sqrt{\frac{d}{k^{2}\ell}}+\sqrt{\frac{8\log\left(\frac{2}{\epsilon}\right)}{k^{2}\ell}}+\left(\sqrt{\frac{d}{k\ell}}+\sqrt{\frac{2\log\left(\frac{2}{\epsilon}\right)}{k\ell}}\right)^{2}\\
 & =2\sqrt{\frac{d}{kn}}+\sqrt{\frac{8\log\left(\frac{2}{\epsilon}\right)}{kn}}+\left(\sqrt{\frac{d}{n}}+\sqrt{\frac{2\log\left(\frac{2}{\epsilon}\right)}{n}}\right)^{2}\\
 & \trre[\leq,a]2\sqrt{\frac{d}{kn}}+\sqrt{\frac{8\log\left(\frac{2}{\epsilon}\right)}{kn}}+\frac{2d+4\log\left(\frac{2}{\epsilon}\right)}{n},
\end{align}
where $(a)$ follows from $(a+b)^{2}\leq2a^{2}+2b^{2}$. Setting $\epsilon=\frac{2}{n}$
and further upper bounding 
\begin{align}
2\sqrt{\frac{d}{kn}}+\sqrt{\frac{8\log(n)}{kn}}+\frac{2d+4\log(n)}{n} & \leq\sqrt{\frac{8d+16\log(n)}{kn}}+\frac{2d+4\log(n)}{n}\\
 & \leq4\sqrt{\frac{d}{kn}}+\frac{4d}{n},
\end{align}
using $\sqrt{a}+\sqrt{b}\leq\sqrt{2(a+b)}$ and $d\geq2\log(n)$, results
$\P[{\cal E}_{n,d,k}^{'''}]\geq1-\frac{2}{n}$. For the bound on the
expectation, we further upper bound 
\[
V\leq2\sqrt{\frac{d}{\ell}}+2\eta+\frac{2d}{\ell}+2\eta^{2}.
\]
Let $\alpha=2\sqrt{\frac{d}{\ell}}+2\frac{d}{\ell}$. Then, $U\dfn V-\alpha$
satisfies that $\P\left[U\geq2\eta+2\eta^{2}\right]\leq2e^{-\ell\eta^{2}/2}$.
Then, 
\begin{align}
\E\left[U\right] & =\E\left[U\cdot\I\{U<0\}\right]+\E\left[U\cdot\I\{U\geq0\}\right]\\
 & \leq\E\left[U\cdot\I\{U\geq0\}\right]\\
 & =\int_{0}^{\infty}\P\left[U\cdot\I\{U\geq0\}>t\right]\d t\\
 & =\int_{0}^{\infty}\P\left[U>t\right]\d t\\
 & \trre[=,a]\int_{0}^{\infty}\left(2+4\eta\right)\P\left[U>2\eta+2\eta^{2}\right]\d\eta\\
 & \leq\int_{0}^{\infty}\left(2+4\eta\right)2e^{-\ell\eta^{2}/2}\d\eta\\
 & =4\int_{0}^{\infty}e^{-\ell\eta^{2}/2}\d\eta+8\int_{0}^{\infty}\eta e^{-\ell\eta^{2}/2}\d\eta\\
 & \trre[=,b]4\sqrt{\frac{2\pi}{\ell}}+\frac{8}{\ell},
\end{align}
where $(a)$ follows from the change of variables $t=2\eta+2\eta^{2}$,
and $(b)$ follows from Gaussian integration (the first term). Hence,
\[
\E\left[\frac{V}{k}\right]\leq\frac{1}{k}\E[U+\alpha]\leq13\sqrt{\frac{d}{nk}}+10\frac{d}{n}
\]
which proves (\ref{eq: expectation of noise square over all blocks}).
The claim $\P[{\cal E}_{n,d,\delta,k}]\geq1-\frac{5}{n}$ then follows
from the union bound over the three events considered above. 
\end{proof}
The next lemma utilizes Lemma \ref{lem:concentration of empirical terms}
to establish concentration on the maximal eigenvalue and the associated
eigenvector.
\begin{lem}
\label{lem: concentration of eigenvalue and vector}
Let 
\begin{equation}
\psi(n,d,\delta,k)\dfn2\sqrt{\frac{\delta k^{2}}{n}}\cdot\|\theta_{*}\|^{2}+2\sqrt{\frac{d}{n}}\cdot\|\theta_{*}\|+13\sqrt{\frac{d}{nk}}+10\frac{d}{n}.\label{eq: definition of psi}
\end{equation}
Then, it holds that

\[
\E\left[\left|\lambda_{\text{\emph{max}}}\left(\hat{\Sigma}_{n,k}(X_{1}^{n})\right)-\left(\xi_{k}\|\theta_{*}\|^{2}+\frac{1}{k}\right)\right|\right]\leq\psi(n,d,\delta,k)
\]
and
\[
\E\left[\loss\left(v_{\text{\emph{max}}}\left(\hat{\Sigma}_{n,k}(X_{1}^{n})\right),\frac{\theta_{*}}{\|\theta_{*}\|}\right)\right]\leq\frac{8\cdot\psi(n,d,\delta,k)}{\|\theta_{*}\|^{2}}.
\]
If ${\cal E}_{n,d,\delta,k}$ holds then similar bounds hold (with probability $1$) with
$\psi(n,d,\delta,k)$ multiplied by $7\log(n)$. 
\end{lem}

\begin{proof}
Note that $\lambda_{\text{max}}(\Sigma_{n,k}(\theta_{*}))=\xi_{k}\cdot\|\theta_{*}\|^{2}+\frac{1}{k}$.
Then, 
\begin{align}
 & \left|\lambda_{\text{max}}\left(\hat{\Sigma}_{n,k}(X_{1}^{n})\right)-\left(\xi_{k}\|\theta_{*}\|^{2}+\frac{1}{k}\right)\right|\nonumber \\
 & =\left|\lambda_{\text{max}}\left(\hat{\Sigma}_{n,k}(X_{1}^{n})\right)-\lambda_{\text{max}}\left(\Sigma_{n,k}(\theta_{*})\right)\right|\\
 & \trre[\leq,a]\left\Vert \hat{\Sigma}_{n,k}(X_{1}^{n})-\Sigma_{n,k}(\theta_{*})\right\Vert _{\text{op}}\\
 & =\left\Vert \hat{\Sigma}_{n,k}(X_{1}^{n})-\xi_{k}\theta_{*}\theta_{*}^{\T}-\frac{1}{k}I_{d}\right\Vert _{\text{op}}\\
 & \leq\left|\frac{1}{\ell}\sum_{i=1}^{\ell}\overline{S}_{i}^{2}-\xi_{k}\right|\cdot\|\theta_{*}\|^{2}+\left\Vert \frac{2}{\ell}\sum_{i=1}^{\ell}\overline{S}_{i}\overline{Z}_{i}\right\Vert \cdot\|\theta_{*}\|+\left\Vert \frac{1}{\ell}\sum_{i=1}^{\ell}\overline{Z}_{i}\overline{Z}_{i}^{\T}-\frac{1}{k}\cdot I_{d}\right\Vert _{\text{op}},
\end{align}
where $(a)$ follows from Weyl's inequality. Taking expectation and
using Lemma \ref{lem:concentration of empirical terms} proves the
claimed bound. 

Next, note that $v_{\text{max}}(\Sigma_{n,k}(\theta_{*}))=\frac{\theta_{*}}{\|\theta_{*}\|}$,
its eigenvalue is $\xi_k\|\theta_{*}\|^{2}+\frac{1}{k}$, and that all
the $(d-1$) other eigenvalues of $\Sigma_{n,k}(\theta_{*})$ are
$\frac{1}{k}$. Thus, the eigen-gap of the maximal eigenvalue is
$\xi_k\|\theta_{*}\|^{2}$. Then, by Davis-Kahan's perturbation bound \eqref{eqn:davis-kahan} from \citep[Theorem 4.5.5]{vershynin2018high}
\begin{align*}
\loss\left(v_{\text{max}}\left(\hat{\Sigma}_{n,k}(X_{1}^{n})\right),\frac{\theta_{*}}{\|\theta_{*}\|}\right) & =\loss\left(v_{\text{max}}\left(\hat{\Sigma}_{n,k}(X_{1}^{n})\right),v_{\text{max}}\left(\Sigma_{n,k}(\theta_{*})\right)\right)\\
 & \leq4\frac{\left\Vert \hat{\Sigma}_{n,k}(X_{1}^{n})-\Sigma_{n,k}(\theta_{*})\right\Vert _{\text{op}}}{\xi_k\|\theta_{*}\|^{2}} \\
 &\trre[\le,a] 8\frac{\left\Vert \hat{\Sigma}_{n,k}(X_{1}^{n})-\Sigma_{n,k}(\theta_{*})\right\Vert _{\text{op}}}{\|\theta_{*}\|^{2}},
\end{align*}
where $(a)$ follows since $\xi_{k}=\E[\overline{S}^{2}]\geq1-4k\delta=\frac{1}{2}$
according to Lemma \ref{lem:analysis of the random gain}, and the
choice $k=\frac{1}{8\delta}$. 
The operator norm is upper bounded as for the maximal eigenvalue.
The proof then follows by taking expectation, and utilizing Lemma
\ref{lem:concentration of empirical terms}. 

Finally, assuming the event ${\cal E}_{n,d,\delta,k}$ holds, Lemma
\ref{lem:concentration of empirical terms} implies that the same
bound holds with $7\log(n)\cdot\psi(n,d,\delta,k)$ instead of $\psi(n,d,\delta,k)$. 
\end{proof}
The next lemma upper bounds $\hat{\theta}$ in expectation and in
high probability whenever $\|\theta_{*}\|$ is below the minimax rate
$\beta(n,d,\delta)=\Theta(\sqrt{\frac{d}{n}}\vee(\frac{\delta d}{n})^{1/4})$,
and show that both are (roughly) at most on the scale of $\beta(n,d,\delta)$. 
\begin{lem}
\label{lem:scale estimation}
If $\|\theta_{*}\|\leq\beta(n,d,\delta)$
and $k=\frac{1}{8\delta}$ then 

\begin{equation}
\E\left[\left\Vert \hat{\theta}_{\text{cov}}(X_{1}^{n};k)\right\Vert \right]\leq11\cdot\beta(n,d,\delta),\label{eq: expectation bound in case of low signal}
\end{equation}
and if the event ${\cal E}_{n,d,\delta,k}$ holds then 
\begin{equation}
\left\Vert \hat{\theta}_{\text{cov}}(X_{1}^{n};k)\right\Vert \leq77\log(n)\cdot\beta(n,d,\delta).\label{eq: correct scale in case of low snr}
\end{equation}
\end{lem}

\begin{proof}
It holds that 
\begin{align}
\|\hat{\theta}_{\text{cov}}(X_{1}^{n};k)\|^{2} & =\frac{\lambda_{\text{max}}\left(\hat{\Sigma}_{n,k}(X_{1}^{n})\right)-\frac{1}{k}}{\xi_{k}}\vee0\\
 &\trre[\le,a] 2\left[\lambda_{\text{max}}\left(\hat{\Sigma}_{n,k}(X_{1}^{n})\right)-\frac{1}{k}\right]\vee0\\
 & \le2\left[\lambda_{\text{max}}\left(\hat{\Sigma}_{n,k}(X_{1}^{n})\right)-\xi_{k}\|\theta_{*}\|^{2}-\frac{1}{k}\right]\vee0+2\xi_{k}\|\theta_{*}\|^{2}\\
 & \trre[\leq,b]2\left[\lambda_{\text{max}}\left(\hat{\Sigma}_{n,k}(X_{1}^{n})\right)-\xi_{k}\|\theta_{*}\|^{2}-\frac{1}{k}\right]\vee0+2\|\theta_{*}\|^{2},
\end{align}
where $(a)$ follows since $ \xi_k\ge\frac{1}{2} $ and $ (b) $ follows since $\xi_{k}\leq1$.
Taking expectation of both sides, and utilizing Lemma \ref{lem: concentration of eigenvalue and vector}
results 
\begin{align}
 & \E\left[\|\hat{\theta}_{\text{cov}}(X_{1}^{n};k)\|^{2}\right]\nonumber \\
 & \leq2\cdot\psi(n,d,\delta,k)+2\|\theta_{*}\|^{2}\\
 & \leq4\sqrt{\frac{\delta k^{2}}{n}}\cdot\|\theta_{*}\|^{2}+4\sqrt{\frac{d}{n}}\cdot\|\theta_{*}\|+26\sqrt{\frac{d}{nk}}+20\frac{d}{n}+2\|\theta_{*}\|^{2}\\
 & \trre[=,a]4\sqrt{\frac{1}{64\delta n}}\cdot\|\theta_{*}\|^{2}+4\sqrt{\frac{d}{n}}\cdot\|\theta_{*}\|+26\sqrt{\frac{8\delta d}{n}}+20\frac{d}{n}+2\|\theta_{*}\|^{2}\\
 & \trre[\leq,b]4\sqrt{\frac{1}{64\delta n}}\cdot\left(\frac{2d}{n}+2\sqrt{\frac{\delta d}{n}}\right)+4\sqrt{\frac{d}{n}}\cdot\left(\sqrt{\frac{d}{n}}+\left(\frac{\delta d}{n}\right)^{1/4}\right)\nonumber \\
 & \hphantom{==}+26\sqrt{\frac{8\delta d}{n}}+20\frac{d}{n}+2\left(\frac{2d}{n}+2\sqrt{\frac{\delta d}{n}}\right)\\
 & \leq4\sqrt{\frac{1}{64\delta n}}\cdot\left(\frac{2d}{n}+2\sqrt{\frac{\delta d}{n}}\right)+4\sqrt{\frac{d}{n}}\cdot\left(\frac{\delta d}{n}\right)^{1/4}+78\sqrt{\frac{\delta d}{n}}+28\frac{d}{n}\\
 & =\frac{d}{\sqrt{\delta}n^{3/2}}+\sqrt{\frac{d}{n^{2}}}+4\sqrt{\frac{d}{n}}\cdot\left(\frac{\delta d}{n}\right)^{1/4}+78\sqrt{\frac{\delta d}{n}}+28\frac{d}{n}\\
 & \trre[\leq,c]4\sqrt{\frac{d}{n}}\cdot\left(\frac{\delta d}{n}\right)^{1/4}+78\sqrt{\frac{\delta d}{n}}+30\frac{d}{n}\\
 & \trre[\leq,d]80\sqrt{\frac{\delta d}{n}}+32\frac{d}{n},
\end{align}
where $(a)$ follows by setting $k=\frac{1}{8\delta}$, and $(b)$
follows from the assumption $\|\theta_{*}\|\leq\beta(n,d,\delta)\leq\sqrt{\frac{d}{n}}+\left(\frac{\delta d}{n}\right)^{1/4}$
(and $(a+b)^{2}\leq2a^{2}+2b^{2}$), $(c)$ follows from $\delta\geq\frac{1}{n}$,
so that $\frac{d}{\sqrt{\delta}n^{3/2}}\leq\frac{d}{n}$, and $(d)$
follows from $ab\leq\frac{1}{2}a^{2}+\frac{1}{2}b^{2}$ applied to
the first term.

The proof is then completed by Jensen's inequality 
\begin{align}
\E\left[\|\hat{\theta}_{\text{cov}}(X_{1}^{n};k)\|\right] & \leq\sqrt{\E\left[\|\hat{\theta}_{\text{cov}}(X_{1}^{n};k)\|^{2}\right]}\\
 & \leq\sqrt{80\sqrt{\frac{\delta d}{n}}+32\frac{d}{n}}\\
 & \leq\sqrt{112\cdot\left(\sqrt{\frac{\delta d}{n}}\vee\frac{d}{n}\right)}\\
 & \leq11\cdot\beta(n,d,\delta).
\end{align}

If $\lambda_{\text{max}}(\hat{\Sigma}_{n,k}(X_{1}^{n}))-\frac{1}{k}\leq0$
then $\hat{\theta}_{\text{cov}}(X_{1}^{n};k)=0$ and the claim is trivial.
Otherwise, from the same reasoning as above, assuming that the event
${\cal E}_{n,d,\delta,k}$ holds, it also holds from Lemma \ref{lem: concentration of eigenvalue and vector} that $\|\hat{\theta}_{\text{cov}}(X_{1}^{n};k)\|\leq77\log(n)\cdot\beta(n,d,\delta)$.
\end{proof}
We may now prove Theorem \ref{thm: mean estimation known delta upper bound}. 
\begin{proof}[Proof of Theorem \ref{thm: mean estimation known delta upper bound}.]
 If $\|\theta_{*}\|\leq\beta(n,d,\delta)$ then Lemma \ref{lem:scale estimation}
implies that 
\[
\E\left[\left\Vert \hat{\theta}_{\text{cov}}(X_{1}^{n};k)\right\Vert \right]\leq11\cdot\beta(n,d,\delta).
\]
Otherwise, assume that $\|\theta_{*}\|\geq\beta(n,d,\delta)$. For
any estimator $\tilde{\theta}$ we may write $\tilde{\theta}=\|\tilde{\theta}\|\cdot v$
where $v\in\mathbb{S}^{d-1}$, it holds that 
\begin{align}
\loss(\tilde{\theta},\theta_{*}) & =\left\Vert \|\tilde{\theta}\|\cdot v-\theta_{*}\right\Vert \\
 & =\left\Vert \|\tilde{\theta}\|\cdot v-\|\theta_{*}\|\cdot v+\|\theta_{*}\|\cdot v-\theta_{*}\right\Vert \\
 & \leq\left|\|\tilde{\theta}\|-\|\theta_{*}\|\right|+\|\theta_{*}\|\cdot\loss\left(v,\frac{\theta_{*}}{\|\theta_{*}\|}\right),
\end{align}
where the last inequality follows from the triangle inequality. Specifying
this result to $\tilde{\theta}=\hat{\theta}\equiv\hat{\theta}_{\text{cov}}(X_{1}^{n};k)$,
results 
\begin{equation}
\loss(\hat{\theta},\theta_{*})\leq\left|\|\hat{\theta}\|-\|\theta_{*}\|\right|+\|\theta_{*}\|\cdot\loss\left(v_{\text{max}}\left(\hat{\Sigma}_{n,k}(X_{1}^{n})\right),\frac{\theta_{*}}{\|\theta_{*}\|}\right).\label{eq: loss as norm loss and eigen loss}
\end{equation}
 Lemma \ref{lem: concentration of eigenvalue and vector} implies
that the first term in (\ref{eq: loss as norm loss and eigen loss})
is bounded as
\begin{align}
\left|\|\hat{\theta}\|-\|\theta_{*}\|\right| & =\frac{\left|\|\hat{\theta}\|^{2}-\|\theta_{*}\|^{2}\right|}{\|\hat{\theta}\|+\|\theta_{*}\|}\\
 & \leq\frac{\left|\|\hat{\theta}\|^{2}-\|\theta_{*}\|^{2}\right|}{\|\theta_{*}\|}\\
 & =\left|\frac{\left(\lambda_{\max}(\hat{\Sigma}_{n,k}(X_{1}^{n}))-\frac{1}{k}\right)_{+}-\xi_{k}\|\theta_{*}\|^{2}}{\xi_{k}\|\theta_{*}\|}\right|\\
 & \leq\frac{\left|\lambda_{\max}(\hat{\Sigma}_{n,k}(X_{1}^{n}))-\frac{1}{k}-\xi_{k}\|\theta_{*}\|^{2}\right|}{\xi_{k}\|\theta_{*}\|}\\
 & \leq\frac{2\left|\lambda_{\max}(\hat{\Sigma}_{n,k}(X_{1}^{n}))-\frac{1}{k}-\xi_{k}\|\theta_{*}\|^{2}\right|}{\|\theta_{*}\|},\label{eq: norm difference empirical bound}
\end{align}
where the last inequality follows since $\xi_{k}\geq1-4k\delta$ (from
Lemma \ref{lem:analysis of the random gain}), and $k\leq\frac{1}{8\delta}$.
Taking expectation of both sides, and utilizing Lemma \ref{lem: concentration of eigenvalue and vector}
results 
\[
\E\left[\left|\|\hat{\theta}\|-\|\theta_{*}\|\right|\right]\leq\frac{2\cdot\psi(n,d,\delta,k)}{\|\theta_{*}\|}.
\]
Lemma \ref{lem: concentration of eigenvalue and vector} further implies
that the second term in (\ref{eq: loss as norm loss and eigen loss})
is bounded as 
\[
\E\left[\|\theta_{*}\|\cdot\loss\left(v_{\text{max}}\left(\hat{\Sigma}_{n,k}(X_{1}^{n})\right),\frac{\theta_{*}}{\|\theta_{*}\|}\right)\right]\leq\frac{8\cdot\psi(n,d,\delta,k)}{\|\theta_{*}\|}.
\]
Hence, (\ref{eq: loss as norm loss and eigen loss}) 
\begin{align}
\loss(\hat{\theta},\theta_{*}) & \leq\frac{10\cdot\psi(n,d,\delta,k)}{\|\theta_{*}\|}\\
 & \leq 20\sqrt{\frac{\delta k^{2}}{n}}\cdot\|\theta_{*}\|+20\sqrt{\frac{d}{n}}+\frac{130}{\|\theta_{*}\|}\sqrt{\frac{d}{nk}}+100\frac{1}{\|\theta_{*}\|}\frac{d}{n}\\
 & \trre[=,a]20\sqrt{\frac{1}{64\delta n}}\cdot\|\theta_{*}\|+20\sqrt{\frac{d}{n}}+\frac{130}{\|\theta_{*}\|}\sqrt{\frac{8\delta d}{n}}+\frac{100}{\|\theta_{*}\|}\frac{d}{n}\\
 & \trre[\leq,b]20\sqrt{\frac{d}{n}}+\frac{130}{\|\theta_{*}\|}\sqrt{\frac{8\delta d}{n}}+\frac{110}{\|\theta_{*}\|}\frac{d}{n},\label{eq: upper bound on the loss proof}
\end{align}
where $(a)$ follows from setting $k=\frac{1}{8\delta}$, and $(b)$
follows since for $\|\theta_{*}\|\geq\beta(n,d,\delta)$,
\begin{align}
20\sqrt{\frac{1}{64\delta n}}\cdot\|\theta_{*}\| & =20\frac{\|\theta_{*}\|^{2}}{\|\theta_{*}\|}\sqrt{\frac{1}{64\delta n}}\\
 & \leq20\cdot\frac{\left(\sqrt{\frac{d}{n}}+\left(\frac{\delta d}{n}\right)^{1/4}\right)^{2}}{\|\theta_{*}\|}\cdot\sqrt{\frac{1}{64\delta n}}\\
 & \leq5\cdot\frac{\frac{d}{n}+\sqrt{\frac{\delta d}{n}}}{\|\theta_{*}\|}\cdot\sqrt{\frac{1}{\delta n}}\\
 & =5\cdot\frac{1}{\|\theta_{*}\|}\left(\frac{d}{n^{3/2}\sqrt{\delta}}+\frac{\sqrt{d}}{n}\right)\\
 & \leq\frac{5}{\|\theta_{*}\|}\frac{d}{n^{3/2}\sqrt{\delta}}+\frac{5}{\|\theta_{*}\|}\frac{d}{n}\\
 & \leq\frac{10}{\|\theta_{*}\|}\frac{d}{n},
\end{align}
where the last inequality follows since $\delta\geq\frac{1}{n}$. 
\end{proof}

\subsection{Proof of Theorem \ref{thm: impossibility lower bound for mean estimation}:
Impossibility lower bound}
\label{app:proof-converse-estimate-theta}

\paragraph*{The proof's main ideas}

The estimation error for the Markov model cannot be better than a
genie-aided model for which the sign is known during a block whose
size is larger than the mixing time $\Theta(\frac{1}{\delta})$ of
the original Markov chain. The estimator can then align the signs
of the samples in each block, and thus reduce the noise variance by
a factor of $\delta$. This effectively reduces the problem into a
Markov model whose flip probability is asymptotically close to $1/2$,
that is, the resulting model is close to a GMM. This aforementioned
closeness is quantified by a uniform bound on the ratio between the
probability distributions. Based on the closeness of the models, the
known lower bound for the GMM model \citep[Appendxi B]{wu2019EM}
implies a bound on the estimation error in the Markov model.

The proof follows from an application of Fano's method \citep[Section 15.3]{wainwright2019high} \citep{yang1999information},
where, as usual, the main technical challenge follows from the bounding
of the mutual information term. We first begin with a brief description
of the lower bound on the estimation error, and then explain how it
is solved in the Gaussian mixture case in \citep[Appendix B]{wu2019EM}.
We then describe our analysis method for the Markov case. 

\paragraph*{Construction of a packing set}

The version of Fano's method that we use is based on constructing
a packing set $\Theta_{M}\dfn\{\theta_{m}\}_{m\in[M]}\subset\mathbb{R}^d $ and an additional
center vector $\theta_{0}\in\mathbb{R}^{d}$. Here, the set $\Theta_{M}$
packs points in a spherical cap of a fixed angle with a center at
$\theta_{0}$. To construct this set, we let $\Phi_{M}\dfn\{\phi_{m}\}_{m\in[M]}\subset\mathbb{B}^{d-2}$
be a $\frac{1}{16}$-packing set of $\mathbb{B}^{d-2}$ of size larger
than $M\geq16^{d-2}$ in the Euclidean distance, whose existence is
assured by a standard argument (e.g., \eqref{eqn:packing} from \citep[Lemma 5.7 and Example 5.8]{wainwright2019high}).
We then append to each $\phi_{m}$ another coordinate to obtain $\overline{\phi}_{m}\dfn(\phi_{m},\sqrt{1-\|\phi_{m}\|})\in\mathbb{S}^{d-2}$.
Then, $\overline{\Phi}_{M}\dfn\{\overline{\phi}_{m}\}_{m\in[M]}\subset\mathbb{S}^{d-2}$
is a $\frac{1}{16}$-packing set of $\mathbb{S}^{d-2}$ of size $M$.
Now, the packing set $\overline{\Phi}_{M}$ packs points in the Euclidean
distance, however, the loss function $\loss(\cdot,\cdot)$ in (\ref{eq: loss function for means})
is sign-insensitive, and so the packing set requires further dilution.
Specifically, there must exist an orthant in $\mathbb{R}^{d-1}$ which
contains at least a $2^{-(d-1)}$ fraction of the points in $\overline{\Phi}_{M}$.
So, there must also exist a rotation matrix $O\in\mathbb{R}^{(d-1)\times(d-1)}$
of $\overline{\Phi}_{M}$ so that $\tilde{\Phi}_{M}\dfn O\overline{\Phi}_{M}\cap\mathbb{R}_{+}^{d-1}=\{O\overline{\phi}_{m}\}_{m\in[M]}\cap\mathbb{R}_{+}^{d-1}$
has at least $\frac{|\overline{\Phi}_{M}|}{2^{d-1}}\geq\frac{1}{16}\cdot8^{d-1}$ points
(with $\mathbb{R}_{+}^{d-1}$ being the positive orthant). Furthermore,
for any distinct $\tilde{\phi}_{m},\tilde{\phi}_{m'}\in\tilde{\Phi}_{M}$
it holds that $\loss(\tilde{\phi}_{m},\tilde{\phi}_{m'})=\|\tilde{\phi}_{m}-\tilde{\phi}_{m'}\|$.
We now choose $\epsilon\in(0,1)$, set $\tilde{\theta}_{0}=[1,0,\ldots,0]\in\mathbb{R}^d $
to be the center vector (this choice is arbitrary and made for convenience),
and $\tilde{\theta}_{m}=[\sqrt{1-\epsilon^{2}},\epsilon\tilde{\phi}_{m}]\in\mathbb{R}^d $,
for $m\in[M]$. Evidently, the angle between any $\tilde{\theta}_{m}$
and $\tilde{\theta}_{0}$ is fixed for all $m\in[M]$. Finally, given
the prescribed norm $t$ in the statement of the theorem, we set the
packing set that will be next used in Fano's-inequality based argument
as $\Theta_{M}\dfn\{t\cdot\tilde{\theta}_{m}\}_{m\in[M]}$ and set
$\theta_{0}=t\tilde{\theta}_{0}$. To summarize, $\Theta_{M}$ satisfies
the following properties: (i) $|\Theta_{M}|\geq\frac{1}{16}\cdot8^{d-1}$, (ii) for
any $\theta_{m}\in\Theta_{M}$ it holds that $\|\theta_{m}\|=t$,
(iii) for any distinct $\theta_{m},\theta_{m'}\in\Theta_{M}$ it holds
that $\loss(\theta_{m},\theta_{m'})=\|\theta_{m}-\theta_{m'}\|\geq\frac{1}{16}\epsilon t$,
(iv) for any $\theta_{m}\in\Theta_{m}$ it holds that 
\[
\loss(\theta_{m},\theta_{0})=\|\theta_{m}-\theta_{0}\|=t\cdot\sqrt{(\sqrt{1-\epsilon^{2}}-1)^{2}+\epsilon^{2}}\leq t\cdot\left[\sqrt{1-\epsilon^{2}}-1+\epsilon\right]\leq2t\epsilon,
\]
where the first inequality is from $\sqrt{a^{2}+b^{2}}\leq a+b$ for
$a,b\in\mathbb{R}_{+}$ and the second inequality from $\sqrt{1-\epsilon^{2}}\leq1+\epsilon$. 

\paragraph*{Fano's inequality based lower bound}

Recall that $P_{\theta}^{(n)}$ is the probability distribution of
$X_{1}^{n}$ under the Markov model $X_{i}=S_{i}\theta+Z_{i}$, $i\in[n]$
with $S_{0}^{n}$ is as in (\ref{eq: Markov sign model}). Let $J\sim\text{Unif}[M]$
and assume that given a prescribed norm $t>0$, it holds that $X_{1}^{n}\mid J=j\sim P_{\theta_{j}}^{(n)}$.
Then, based on the four properties of the packing set $\Theta_{M}$,
Fano's method states that \citep[Proposition 15.2]{wainwright2019high}
\begin{align}
\M(n,d,\delta,t) & \geq\frac{\epsilon t}{32}\cdot\left[1-\frac{I(J;X_{1}^{n})+\log2}{\log M}\right] \label{eqn:apply-fano-1} \\
 & \geq\frac{\epsilon t}{32}\cdot\left[1-\frac{I(J;X_{1}^{n})+\log2}{(d-1)\log8}\right]\\
 & \trre[\geq,a]\frac{\epsilon t}{32}\cdot\left[\frac{1}{2}-4\cdot\frac{I(J;X_{1}^{n})}{d}\right]\\
 & \trre[\geq,b]\frac{\epsilon t}{32}\cdot\left[\frac{1}{2}-4\cdot\frac{\max_{m\in[M]}\Dkl(P_{\theta_{m}}^{(n)}\mid\mid P_{\theta_{0}}^{(n)})}{d}\right]\label{eq: Fano based lower bound for Markov case with KL}
\end{align}
where $I(J;X_{1}^{n})$ is the mutual information between $J$ and
$X_{1}^{n}$, and in $(a)$ we have used the assumption $d\geq3$ which implies $ (d-1)\log8 - \log16\ge \frac{d}{4} $,
and in $(b)$ we have used the standard ``information-radius'' bound
on the mutual information (e.g., (15.52) in \citep[Proof of Lemma 15.21 and Exercise 15.11]{wainwright2019high}).
The crux of the proof is to establish a bound of the form $\Dkl(P_{\theta_{m}}^{(n)}\mid\mid P_{\theta_{0}}^{(n)})\leq\frac{d}{16}$
for a given $\epsilon_{*}>0$ to obtain a lower bound of $\Omega(\epsilon_{*}t)$.
Note that due to symmetry of the packing set, the last KL divergence
is the same for all $m\in[M]$. 

Before continuing the proof for the Markov case with $\delta<\frac{1}{2}$,
we next describe the analysis of the Gaussian mixture model in \citep{wu2019EM}.
For the Gaussian mixture model (that is, a degenerate Markov model
with $\delta=\frac{1}{2}$), $P_{\theta}^{(n)}$ is an i.i.d. distribution,
and so the tensorization property of the KL divergence immediately
implies that $\Dkl(P_{\theta_{m}}^{(n)}\mid\mid P_{\theta_{0}}^{(n)})=n\cdot\Dkl(P_{\theta_{m}}\mid\mid P_{\theta_{0}})$
(where $P_{\theta}\equiv P_{\theta}^{(1)}$). Then, it was established
in \citep[Lemma 27]{wu2019EM} that $\Dkl(P_{\theta_{m}}\mid\mid P_{\theta_{0}})\lesssim t^{2}\cdot\|\theta_{0}-\theta_{m}\|^{2}$
as follows. First, it was noted that under $\theta_{0}=[t,0,\ldots,0]\in\mathbb{R}^d $,
$P_{\theta}=P_{t}\otimes N(0,I_{d-1})$, that is a product distribution,
with the first coordinate being one-dimensional Gaussian mixture with
means $\pm t$, and all the other $d-1$ coordinates being standard
Gaussian. Based on this and the chain rule of the KL divergence, the
KL divergence $\Dkl(P_{\theta_{m}}\mid\mid P_{\theta_{0}})$ was evaluated
separately by the KL divergence between the distribution of the
first coordinate, and the KL divergence between the distributions
of the other $d-1$ coordinates, conditioned on the first coordinates.
Each of these two KL divergences was bounded by the corresponding
chi-square divergences, and was further shown to be $O(t^{2}\cdot\|\theta_{0}-\theta_{m}\|^{2})$.
Finally, the minimax lower bound stated in (\ref{eq: minimax rates Gaussian mixture low dimension})
was obtained by setting $\epsilon=c\cdot\min\{1,\frac{1}{t^{2}}\sqrt{\frac{d}{n}}\}$
for some small enough $c>0$ in the construction of the packing set.

For the Markov model (with $\delta<\frac{1}{2}$), it seems rather
cumbersome to bound $\Dkl(P_{\theta_{m}}^{(n)}\mid\mid P_{\theta_{0}}^{(n)})$
since the model $P_{\theta}^{(n)}$ has memory. Thus we next propose
an indirect approach, which requires three preliminary steps. At the
first step, we reduce, via a genie-based argument, the original Gaussian-Markov
model with $n$ samples to a Gaussian-Markov model with $\ell$ samples,
where each of these $\ell$ sample is a coherent average of a block
of $k$ consecutive samples, $\ell=\frac{n}{k}$ (in a similar, yet
not identical, form to the estimator from Section \ref{sec:Mean-estimation-for}).
The sequences of signs underlying each of these $\ell$ samples also
forms a Markov model, where the block length $k$ is judiciously chosen
(roughly) as $\Theta(\frac{1}{\delta})$ so that the dependency between
the signs is much weaker, in the sense that the flip probability $ \overline{\delta} $ of
its signs \emph{asymptotically} tends to $\frac{1}{2}$ as $ \delta\to0 $. Since for
Gaussian mixture models the flip probability is \emph{exactly }$\frac{1}{2}$,
this Markov model is \emph{approximately} a Gaussian mixture model. 

At the second step, we develop a change-of-measure argument. As said,
the reduced model is only approximately Gaussian mixture ($\overline{\delta}\approx\frac{1}{2}$);
Had it was exactly a Gaussian mixture model $(\overline{\delta}=\frac{1}{2}$)
then the bound on the KL divergence from \citep[Lemma 27]{wu2019EM}
could have been directly used. Nonetheless, since $\overline{\delta}\to\frac{1}{2}$
with our choice of $k$, the probability distribution of the samples
tends to that of the memoryless Gaussian mixture. Since divergences
are essentially a continuous functions of the measured distributions,
we expect that the KL divergence for the Markov model with $\overline{\delta}$
is close to the KL divergence for the Gaussian mixture model. Lemma
\ref{lem: ratio between Markov with flip close to half and GM} forms
the basis of this change-of-measure argument in terms of the probability
distribution of the signs, and Lemma \ref{lem: change of measure bound for KL }
provides the corresponding change-of-measure bound for the KL divergence. 

At the third step, we bound the chi-square divergence for a $d$ dimensional
Gaussian mixture model. As described above, in \citep[Lemma 27]{wu2019EM}
the KL divergence was upper bounded for the Gaussian mixture model
by first separating to the KL divergence in the first coordinate and
the KL divergence in all other $d-1$ coordinates, exploiting the chain rule of the
KL divergence, and then bounding each of the two KL divergences by
a chi-square divergence. Here, we provide a refined proof which directly
bounds the chi-square divergence, without the need to separate the
first coordinate from the others. This is the result of Lemma \ref{lem: chi-square for d dimensional Gaussian mixture}.

The proof of Theorem \ref{thm: impossibility lower bound for mean estimation}
is then completed using the results of the three
preliminary steps. We next turn to the detailed proof. 

\textbf{\uline{First step (genie-aided reduction):}} In this step
we reduce the original model $P_{\theta}^{(n)}$ to a model $\tilde{P}_{\theta}^{(\ell)}$
which is aided by knowledge from a genie, and so estimation errors
in the new model are only lower. Inspired by the operation of our
proposed estimator in Section \ref{sec:Mean-estimation-for}, we consider
$\ell=\frac{n}{k}$ blocks of size $k$ each, where, for simplicity
of exposition, we assume that both $\ell,k$ are integers. A genie
informs the estimator with the sign changes $S_{j}S_{j+1}$ at the
time indices $j\in[n-1]\backslash\{ik\}_{i=1}^{\ell-1}$, that is,
at all times \emph{except for $j=k,2k,3k,\cdots,n$.} Hence, the estimator
can align the signs within each block ${\cal I}_{i}\dfn\{(i-1)k+1,(i-1)k+2,\cdots,ik\}$
of size $k$, but not between blocks. Consequently, a statistically
\emph{equivalent} model to the original model with genie information
is the model
\[
Y_{i}=R_{i}\theta_{*}+Z_{i} 
\]
where $ i\in[n] $, $Z_{i}\sim N(0,I_{d})$ are i.i.d. exactly as in the original
Markov model, and the signs $R_{i}\in\{-1,1\}$ are such that 
\[
R_{i+1}=R_{i}
\]
with probability $1$ if $i\in[n-1]\backslash\{jk\}_{j=1}^{\ell-1}$
and
\[
R_{jk+1}=\begin{cases}
R_{(j-1)k+1}, & \text{w.p. }\frac{1+\rho^{k}}{2}\\
-R_{(j-1)k+1}, & \text{w.p. }\frac{1-\rho^{k}}{2}
\end{cases}
\]
if $i\in\{jk\}_{j=1}^{\ell-1}$. Now, since the signs are fixed during
the blocks $\{{\cal I}_{i}\}_{i=1}^{\ell}$ of length $k$, the
mean of the samples in each block is a \emph{sufficient statistic}
for estimation of $\theta_{*}$. Thus, in the same spirit of the estimator
in Section \ref{sec:Mean-estimation-for}, the genie-aided model
is equivalent to 
\begin{equation}
\overline{Y}_{i}\dfn\frac{1}{k}\sum_{j\in{\cal I}_{i}}Y_{i}=\overline{R}_{i}\theta_{*}+\overline{Z}_{i}\label{eq: Gaussian Markov model genie}
\end{equation}
where $ i\in[\ell] $, $\P[\overline{R}_{0}=1]=\frac{1}{2}$, and
\begin{equation}
\overline{R}_{i+1}=\begin{cases}
\overline{R}_{i} & \text{w.p. }\frac{1+\rho^{k}}{2}\\
-\overline{R}_{i} & \text{w.p. }\frac{1-\rho^{k}}{2}
\end{cases}\label{eq: Markov model genie}
\end{equation}
is the gain (average of the signs) of the $i$th block, and 
\[
\overline{Z}_{i}\dfn\frac{1}{k}\sum_{j\in{\cal I}_{i}}Z_{j}\sim N\left(0,\frac{1}{k}\cdot I_{d}\right)
\]
is a averaged Gaussian noise. Thus, there are $\ell$ samples in the
equivalent model, and the flip probability is $\overline{\delta}\dfn\frac{1-\rho^{k}}{2}$,
which is closer to $\frac{1}{2}$ compared to the flip probability
in the original model, $\delta=\frac{1-\rho}{2}$. We note that there
are two differences compared to the model used by the estimator in
Section \ref{sec:Mean-estimation-for}. First, due to the information
supplied from the genie, here the average gain is $\overline{R}_{i}\in\{-1,1\}$
with probability $1$, to wit $|\overline{R}_{i}|$ never drops below
$1$; Second, there is no randomization between the blocks. In the
definition of the estimator in Section \ref{sec:Mean-estimation-for}
we indeed had the freedom to randomize the first sign in each block
in order to make the blocks statistically independent, whereas here,
in an impossibility bound, it is possible that statistical dependence
between the blocks may result improved estimation rates (though as
we shall see, this is essentially not the case). In what follows,
we denote by $\overline{P}_{\theta}^{(\ell)}$ the probability
distribution of $\overline{Y}_{1}^{\ell}$ under the model (\ref{eq: Gaussian Markov model genie})
with flip probability $\overline{\delta}$ (i.e., as with signs as
in (\ref{eq: Markov model genie})). At this point, the trade-off
involved in an optimal choice of $k$ is already apparent -- as the
block length $k$ increases, the dependence between the blocks, as
reflected by the flip probability $\overline{\delta}=\frac{1-\rho^{k}}{2}$,
decreases. On the other hand, the largest lower bound is attained
when the genie reveals minimal information, that is, when $k$ is
minimal. Our choice of $k$ is thus essentially the minimal value
required so that $\frac{1-\rho^{k}}{2}\to\frac{1}{2}$, and specifically
chosen as $k=\frac{\log(n)}{\delta}$.

\textbf{\uline{Second step (change of measure):}} We denote by
$p_{\delta}(s_{1}^{\ell})=\P[S_{1}^{\ell}=s_{1}^{\ell}]$ the probability
distribution of $\ell$ sign samples drawn according to the law of
a homogeneous binary symmetric Markov chain with $\P[S_{0}=1]=1/2$
and flip probability $\delta$. Note that $p_{1/2}(s_{1}^{\ell})$
is then an i.i.d. model, with uniform random signs. The next lemma
uniformly bounds the ratio between $p_{\delta}(s_{1}^{\ell})$ and
$p_{1/2}(s_{1}^{\ell})$ when $\delta$ is close to $1/2$, as was
obtained in the first step. 
\begin{lem}
\label{lem: ratio between Markov with flip close to half and GM}Let
$S_{0}^{\ell}\in\{-1,1\}^{\ell+1}$ be a homogeneous binary symmetric
Markov chain with $\P[S_{0}=1]=1/2$ and flip probability $\delta\in[0,\frac{1}{2}]$,
and let $p_{\delta}(s_{1}^{\ell})=\P[S_{1}^{\ell}=s_{1}^{\ell}]$.
Furthermore, for $\overline{\delta}=\frac{1-\rho^{k}}{2}$ where $\rho=1-2\delta$
and $k=\frac{\log(n)}{\delta}$ it holds that 
\[
1-\frac{1}{n}\leq\frac{p_{\overline{\delta}}(s_{1}^{\ell})}{p_{1/2}(s_{1}^{\ell})}\leq1+\frac{2}{n}.
\]
\end{lem}

\begin{proof}
Note that $p_{1/2}(s_{1}^{\ell})=\frac{1}{2^{\ell}}$ for all $s_{1}^{\ell}\in\{-1,1\}^{\ell}$.
Then, since $\delta\in[0,\frac{1}{2}]$, letting $\rho=1-2\delta$, 
\[
p_{\delta}(s_{1}^{\ell})\leq(1-\delta)^{\ell}=\left(\frac{1+\rho}{2}\right)^{\ell}=p_{1/2}(s_{1}^{\ell})\left(1+\rho\right)^{\ell}=p_{1/2}(s_{1}^{\ell})\left(2-2\delta\right)^{\ell}
\]
and 
\[
p_{\delta}(s_{1}^{\ell})\geq\delta^{\ell}=\left(\frac{1-\rho}{2}\right)^{\ell}=p_{1/2}(s_{1}^{\ell})\left(1-\rho\right)^{\ell}=p_{1/2}(s_{1}^{\ell})\left(2\delta\right)^{\ell} . 
\]
It this holds that 
\[
\left(2\delta\right)^{\ell}\leq\frac{p_{\delta}(s_{1}^{\ell})}{p_{1/2}(s_{1}^{\ell})}\leq\left(2-2\delta\right)^{\ell}.
\]
Substituting $\delta=\overline{\delta}=\frac{1-\rho^{k}}{2}$, the upper
bound on $\frac{p_{\overline{\delta}}(s_{1}^{\ell})}{p_{1/2}(s_{1}^{\ell})}$
is further upper bounded as 
\begin{align}
\left(1+\rho^{k}\right)^{\ell} & =\left(1+(1-2\delta)^{\frac{\log(n)}{\delta}}\right)^{\ell}\\
 & \trre[\leq,a]\left(1+\frac{1}{n^{2}}\right)^{\ell}\\
 & \leq\left(1+\frac{1}{n^{2}}\right)^{n}\\
 & \trre[\leq,a]e^{1/n}\\
 & \trre[\leq,b]1+\frac{2}{n},
\end{align}
where both transitions denoted $(a)$ follow from $1+x\leq e^{x}$
for $x\in\mathbb{R}$, and $(b)$ follows from $e^{x}\leq1+2x$ for
$x\in[0,1]$. Similarly, the lower bound on $\frac{p_{\overline{\delta}}(s_{1}^{\ell})}{p_{1/2}(s_{1}^{\ell})}$
is further lower bounded as 
\begin{align}
\left(1-\rho^{k}\right)^{\ell} & =\left(1-(1-2\delta)^{\frac{\log(n)}{\delta}}\right)^{\ell}\\
 & \trre[\geq,a]\left(1-\frac{1}{n^{2}}\right)^{\ell}\\
 & \geq\left(1-\frac{1}{n^{2}}\right)^{n}\\
 & \trre[\geq,b]1-\frac{1}{n},
\end{align}
where $(a)$ follows again from $1+x\leq e^{x}$ for $x\in\mathbb{R}$,
and $(b)$ follows from Bernoulli's inequality $(1-x)^{r}\geq1-rx$
for $x\in[0,1]$ and $r\geq1$.
\end{proof}
Let $\tilde{P}_{\theta}^{(\ell)}$ denote the Gaussian mixture model (i.e.,
with flip probability $1/2$) corresponding to $\ell$ samples with
means at $\pm\theta\in\mathbb{R}^{d}$, and note that it has the same
$\ell$ marginal distributions as the genie-aided reduced model $\overline{P}_{\theta}^{(\ell)}$(with
flip probability $\overline{\delta}$). As discussed, the probability
distributions $\tilde{P}_{\theta}^{(\ell)}$ and $\overline{P}_{\theta}^{(\ell)}$
are close. The following lemma provides a change-of-measure bound
on the KL divergence for this case.
\begin{lem}
\label{lem: change of measure bound for KL }Let $U_{1}^{\ell}$ (resp.
$V_{1}^{\ell}$) be distributed according to a probability distribution
$P\equiv P_{U_{1}\cdots U_{\ell}}$ (resp. $Q\equiv Q_{U_{1}\cdots U_{\ell}}$
) on $(\mathbb{R}^{d})^{\ell}$, and let $\tilde{P}\equiv\tilde{P}_{U_{1}\cdots U_{\ell}}\dfn\prod_{i=1}^{n}P_{U_{i}}$
and $\tilde{Q}\equiv\tilde{Q}_{V_{1}\cdots V_{\ell}}\dfn\prod_{i=1}^{n}Q_{V_{i}}$
be the product distributions of their marginal distributions. Let 
\[
\beta_{P}\dfn\sup_{u_{1}^{\ell}\in(\mathbb{R}^{d})^{\ell}}\left(\frac{P_{U_{1}\cdots U_{\ell}}(u_{1}^{\ell})}{\tilde{P}_{U_{1}\cdots U_{\ell}}(u_{1}^{\ell})}\right)
\]
and 
\[
\beta_{Q}\dfn\sup_{u_{1}^{\ell}\in(\mathbb{R}^{d})^{\ell}}\left(\frac{\tilde{Q}_{V_{1}\cdots V_{\ell}}(u_{1}^{\ell})}{Q_{V_{1}\cdots V_{\ell}}(u_{1}^{\ell})}\right).
\]
Then, 
\[
\Dkl\left(P\mid\mid Q\right)\leq\Dkl\left(\tilde{P}\mid\mid\tilde{Q}\right)+\log(\beta_{P}\cdot\beta_{Q}).
\]
\end{lem}

\begin{proof}
Since $\tilde{Q}$ is a product distribution, and since $P$ and $\tilde{P}$
have the same marginal distributions on $\mathbb{R}^{d}$ (at each
time point) then 
\begin{align}
 & \int P_{U_{1}\cdots U_{\ell}}(u_{1}^{\ell})\log\left(\frac{1}{\tilde{Q}_{V_{1}\cdots V_{\ell}}(u_{1}^{\ell})}\right)\d u_{1}^{\ell}\nonumber \\
 & =\sum_{i=1}^{\ell}\int P_{U_{1}\cdots U_{\ell}}(u_{1}^{\ell})\log\left(\frac{1}{\tilde{Q}_{V_{i}}(u_{i})}\right)\d u_{1}^{\ell}\\
 & =\sum_{i=1}^{\ell}\int P_{U_{i}}(u_{i})\log\left(\frac{1}{\tilde{Q}_{V_{i}}(u_{i})}\right)\d u_{i}\\
 & =\sum_{i=1}^{\ell}\int\tilde{P}_{U_{i}}(u_{i})\log\left(\frac{1}{\tilde{Q}_{V_{i}}(u_{i})}\right)\d u_{i}\\
 & =\sum_{i=1}^{\ell}\int\tilde{P}_{U_{1}\cdots U_{\ell}}(u_{1}^{\ell})\log\left(\frac{1}{\tilde{Q}_{V_{i}}(u_{i})}\right)\d u_{1}^{\ell}\\
 & =\int\tilde{P}_{U_{1}\cdots U_{\ell}}(u_{1}^{\ell})\log\left(\frac{1}{\tilde{Q}_{V_{1}\cdots V_{\ell}}(u_{1}^{\ell})}\right)\d u_{1}^{\ell}.\label{eq: change of measure for KL - first bound}
\end{align}
Using similar reasoning
\begin{align}
 & \int\tilde{P}_{U_{1}\cdots U_{\ell}}(u_{1}^{\ell})\log\left(\tilde{P}_{U_{1}\cdots U_{\ell}}(u_{1}^{\ell})\right)\d u_{1}^{\ell}\nonumber \\
 & =\sum_{i=1}^{\ell}\int\tilde{P}_{U_{1}\cdots U_{\ell}}(u_{1}^{\ell})\log\left(\tilde{P}_{U_{i}}(u_{i})\right)\d u_{1}^{\ell}\\
 & =\sum_{i=1}^{\ell}\int\tilde{P}_{U_{i}}(u_{i})\log\left(\tilde{P}_{U_{i}}(u_{i})\right)\d u_{i}\\
 & =\sum_{i=1}^{\ell}\int P_{U_{i}}(u_{i})\log\left(\tilde{P}_{U_{i}}(u_{i})\right)\d u_{i}\\
 & =\sum_{i=1}^{\ell}\int P_{U_{1}\cdots U_{\ell}}(u_{1}^{\ell})\log\left(\tilde{P}_{U_{i}}(u_{i})\right)\d u_{1}^{\ell}\\
 & =\int P_{U_{1}\cdots U_{\ell}}(u_{1}^{\ell})\log\left(\tilde{P}_{U_{1}\cdots U_{\ell}}(u_{1}^{\ell})\right)\d u_{1}^{\ell}\\
 & =\int P_{U_{1}\cdots U_{\ell}}(u_{1}^{\ell})\log\left(P_{U_{1}\cdots U_{\ell}}(u_{1}^{\ell})\right)\d u_{1}^{\ell}+\int P_{U_{1}\cdots U_{\ell}}(u_{1}^{\ell})\log\left(\frac{\tilde{P}_{U_{1}\cdots U_{\ell}}(u_{1}^{\ell})}{P_{U_{1}\cdots U_{\ell}}(u_{1}^{\ell})}\right)\d u_{1}^{\ell}.\label{eq: change of measure for KL - second bound}
\end{align}
Combining (\ref{eq: change of measure for KL - first bound}) and
(\ref{eq: change of measure for KL - second bound}) results the following
bound
\begin{align}
 & \Dkl\left(P\mid\mid Q\right)\nonumber \\
 & =\int P_{U_{1}\cdots U_{\ell}}(u_{1}^{\ell})\log\left(\frac{P_{U_{1}\cdots U_{\ell}}(u_{1}^{\ell})}{Q_{V_{1}\cdots V_{\ell}}(u_{1}^{\ell})}\right)\d u_{1}^{\ell}\\
 & =\int P_{U_{1}\cdots U_{\ell}}(u_{1}^{\ell})\log\left(\frac{P_{U_{1}\cdots U_{\ell}}(u_{1}^{\ell})}{\tilde{Q}_{V_{1}\cdots V_{\ell}}(u_{1}^{\ell})}\right)\d u_{1}^{\ell}+\int P_{U_{1}\cdots U_{\ell}}(u_{1}^{\ell})\log\left(\frac{\tilde{Q}_{V_{1}\cdots V_{\ell}}(u_{1}^{\ell})}{Q_{V_{1}\cdots V_{\ell}}(u_{1}^{\ell})}\right)\d u_{1}^{\ell}\\
 &= \int\tilde{P}_{U_{1}\cdots U_{\ell}}(u_{1}^{\ell})\log\left(\tilde{P}_{U_{1}\cdots U_{\ell}}(u_{1}^{\ell})\right)\d u_{1}^{\ell} - \int P_{U_{1}\cdots U_{\ell}}(u_{1}^{\ell})\log\left(\frac{\tilde{P}_{U_{1}\cdots U_{\ell}}(u_{1}^{\ell})}{P_{U_{1}\cdots U_{\ell}}(u_{1}^{\ell})}\right)\d u_{1}^{\ell} \nonumber \\
 &\hphantom{=} +\int P_{U_{1}\cdots U_{\ell}}(u_{1}^{\ell})\log\left(\frac{1}{\tilde{Q}_{V_{1}\cdots V_{\ell}}(u_{1}^{\ell})}\right)\d u_{1}^{\ell} + \int P_{U_{1}\cdots U_{\ell}}(u_{1}^{\ell})\log\left(\frac{\tilde{Q}_{V_{1}\cdots V_{\ell}}(u_{1}^{\ell})}{Q_{V_{1}\cdots V_{\ell}}(u_{1}^{\ell})}\right)\d u_{1}^{\ell} \\
 & \leq\Dkl\left(\tilde{P}\mid\mid\tilde{Q}\right)+\sup_{u_{1}^{\ell}\in(\mathbb{R}^{d})^{\ell}}\log\left(\frac{P_{U_{1}\cdots U_{\ell}}(u_{1}^{\ell})}{\tilde{P}_{U_{1}\cdots U_{\ell}}(u_{1}^{\ell})}\right)+\sup_{u_{1}^{\ell}\in(\mathbb{R}^{d})^{\ell}}\log\left(\frac{\tilde{Q}_{V_{1}\cdots V_{\ell}}(u_{1}^{\ell})}{Q_{V_{1}\cdots V_{\ell}}(u_{1}^{\ell})}\right),
\end{align}
as claimed by the lemma. 
\end{proof}
\textbf{\uline{Third step (bound on the chi-square divergence):}}
Let $\varphi(y;\theta,\sigma^{2})$ be the Gaussian PDF with mean
$\theta\in\mathbb{R}^{d}$ and covariance matrix $\sigma^{2}\cdot I_{d}$.
The next lemma bounds the chi-square divergence between a pair of
such distributions with different means. Originally, a bound of this
order was established in \citep[proof of Lemma 27]{wu2019EM} on the
KL divergence, by splitting the first coordinate (which is assumed,
w.l.o.g., to contain the signal) and the other $d-1$ coordinates,
using the chain rule, and then bounding each of the two KL terms with
the corresponding chi-square divergence. Here we provide a direct
upper bound on the chi-square divergence, which does not rely on splitting
between the coordinates of the mean vector, and which might be of
independent use in future works.
\begin{lem}
\label{lem: chi-square for d dimensional Gaussian mixture}Let $\tilde{P}_{\theta}=\frac{1}{2}N(\theta,\sigma^{2}\cdot I_d)+\frac{1}{2}N(-\theta,\sigma^{2}\cdot I_d)$
be a balanced Gaussian mixture with means at $\pm\theta\in\mathbb{R}^{d}$.
Then, if $\|\theta_{0}\|=\|\theta_{1}\|=t\leq\sigma$, then 
\begin{equation}
\Dchis(\tilde{P}_{\theta_{1}}\mid\mid\tilde{P}_{\theta_{0}})\leq\frac{8t^{2}}{\sigma^{4}}\cdot\|\theta_{0}-\theta_{1}\|^{2}.\label{eq: a bound on chi square for GM in the Fano argument lemma}
\end{equation}
\end{lem}

\begin{proof}
For any $y\in\mathbb{R}^{d}$, 
\begin{align}
\tilde{P}_{\theta}(y) & =\frac{1}{2}\frac{1}{(2\pi\sigma^{2})^{d/2}}e^{-\frac{\|y-\theta\|^{2}}{2\sigma^{2}}}+\frac{1}{2}\frac{1}{(2\pi\sigma^{2})^{d/2}}e^{-\frac{\|y+\theta\|^{2}}{2\sigma^{2}}}\\
 & =\varphi(y;0,\sigma^{2})\cdot e^{-\frac{\|\theta\|^{2}}{2\sigma^{2}}}\cdot\cosh\left(\frac{\theta^{\top}y}{\sigma^{2}}\right).
\end{align}
Then, 
\begin{align}
 & \Dchis\left(\tilde{P}_{\theta_{1}}\mid\mid\tilde{P}_{\theta_{0}}\right)\nonumber \\
 & =\int\frac{\left[\varphi(y;0,\sigma^{2})e^{-\frac{\|\theta_{1}\|^{2}}{2\sigma^{2}}}\cdot\cosh\left(\frac{\theta_{1}^{\top}y}{\sigma^{2}}\right)-\varphi(y;0,\sigma^{2})e^{-\frac{\|\theta_{0}\|^{2}}{2\sigma^{2}}}\cdot\cosh\left(\frac{\theta_{0}^{\top}y}{\sigma^{2}}\right)\right]^{2}}{\varphi(y;0,\sigma^{2})\cdot e^{-\frac{\|\theta_{0}\|^{2}}{2\sigma^{2}}}\cosh\left(\frac{\theta_{0}^{\top}y}{\sigma^{2}}\right)}\cdot\d y\\
 & =e^{\frac{\|\theta_{0}\|^{2}}{2\sigma^{2}}}\int\varphi(y;0,\sigma^{2})\frac{\left[e^{-\frac{\|\theta_{1}\|^{2}}{2\sigma^{2}}}\cdot\cosh\left(\frac{\theta_{1}^{\top}y}{\sigma^{2}}\right)-e^{-\frac{\|\theta_{0}\|^{2}}{2\sigma^{2}}}\cdot\cosh\left(\frac{\theta_{0}^{\top}y}{\sigma^{2}}\right)\right]^{2}}{\cosh\left(\frac{\theta_{0}^{\top}y}{\sigma^{2}}\right)}\cdot\d y\\
 & \trre[\leq,a]e^{\frac{\|\theta_{0}\|^{2}}{2\sigma^{2}}}\cdot\int\varphi(y;0,\sigma^{2})\left[e^{-\frac{\|\theta_{1}\|^{2}}{2\sigma^{2}}}\cdot\cosh\left(\frac{\theta_{1}^{\top}y}{\sigma^{2}}\right)-e^{-\frac{\|\theta_{0}\|^{2}}{2\sigma^{2}}}\cdot\cosh\left(\frac{\theta_{0}^{\top}y}{\sigma^{2}}\right)\right]^{2}\cdot\d y\\
 & =e^{\frac{\|\theta_{0}\|^{2}}{2\sigma^{2}}}\cdot\int\varphi(y;0,\sigma^{2})e^{-\frac{\|\theta_{1}\|^{2}}{\sigma^{2}}}\cdot\cosh^{2}\left(\frac{\theta_{1}^{\top}y}{\sigma^{2}}\right)\cdot\d y\nonumber \\
 & \hphantom{==}-2e^{\frac{\|\theta_{0}\|^{2}}{2\sigma^{2}}}\cdot\int\varphi(y;0,\sigma^{2})e^{-\frac{\|\theta_{1}\|^{2}+\|\theta_{0}\|^{2}}{2\sigma^{2}}}\cosh\left(\frac{\theta_{1}^{\top}y}{\sigma^{2}}\right)\cosh\left(\frac{\theta_{0}^{\top}y}{\sigma^{2}}\right)\cdot\d y\nonumber \\
 & \hphantom{==}+e^{\frac{\|\theta_{0}\|^{2}}{2\sigma^{2}}}\cdot\int\varphi(y;0,\sigma^{2})e^{-\frac{\|\theta_{0}\|^{2}}{\sigma^{2}}}\cdot\cosh^{2}\left(\frac{\theta_{0}^{\top}y}{\sigma^{2}}\right)\cdot\d y,\label{eq: first bound on chi-square divergence btween two GM}
\end{align}
where $(a)$ follows since $\cosh(x)\geq1$ for all $x\in\mathbb{R}$.
We next evaluate the integral for each of the terms. First, for any
$y\in\mathbb{R}^{d}$,
\begin{align}
 & \int\varphi(y;0,\sigma^{2})\cosh^{2}\left(\frac{\theta_{1}^{\top}y}{\sigma^{2}}\right)\d y\nonumber \\
 & \trre[=,a]\int\varphi(y;0,1)\cdot\cosh^{2}\left(\frac{\theta_{1}^{\top}y}{\sigma}\right)\d y\\
 & \trre[=,b]\int\varphi(t;0,1)\cdot\cosh^{2}\left(\frac{\|\theta_{1}\|}{\sigma}t\right)\d t\\
 & =\int\varphi(t;0,1)\left[\frac{\exp\left(\frac{\|\theta_{1}\|}{\sigma}t\right)+\exp\left(-\frac{\|\theta_{1}\|}{\sigma}t\right)}{2}\right]^{2}\d t\\
 & =\int\varphi(t;0,1)\frac{\exp\left(2\frac{\|\theta_{1}\|}{\sigma}t\right)+2+\exp\left(-2\frac{\|\theta_{1}\|}{\sigma}t\right)}{4}\d t\\
 & \trre[=,c]\frac{1}{2}+\frac{1}{2}\cdot e^{2\frac{\|\theta_{1}\|^{2}}{\sigma^{2}}}\\
 & =e^{\frac{\|\theta_{1}\|^{2}}{\sigma^{2}}}\cosh\left(\frac{\|\theta_{1}\|^{2}}{\sigma^{2}}\right),
\end{align}
where $(a)$ follows from the change of variables $y\to\frac{y}{\sigma}$,
$(b)$ follows from the rotational invariance of the Gaussian
PDF $\varphi(y;0,\sigma^{2})$, we may assume that $\theta_{1}=(\|\theta_{1}\|,0,0,\ldots,0)$,
setting $t\in\mathbb{R}$ to be the first coordinate of $y$, and
integrating over all other $d-1$ coordinates, and $(c)$ follows
from the Gaussian moment-generating function formula. The third term
in the integral is similarly evaluated. For the second term, 

\begin{align}
 & \int\varphi(y;0,\sigma^{2})\left[2\cosh\left(\frac{\theta_{1}^{\top}y}{\sigma^{2}}\right)\cosh\left(\frac{\theta_{0}^{\top}y}{\sigma^{2}}\right)\right]\cdot\d y\nonumber \\
 & \trre[=,a]\int\varphi(y;0,1)\left[2\cosh\left(\frac{\theta_{1}^{\top}y}{\sigma}\right)\cosh\left(\frac{\theta_{0}^{\top}y}{\sigma}\right)\right]\cdot\d y\\
 & \trre[=,b]\int\varphi(y;0,1)\cosh\left(\frac{(\theta_{1}+\theta_{0})^{\top}y}{\sigma}\right)\cdot\d y+\int\varphi(y;0,1)\cosh\left(\frac{(\theta_{1}-\theta_{0})^{\top}y}{\sigma}\right)\cdot\d y\\
 & \trre[=,c]\int\varphi(t;0,1)\cosh\left(\frac{\|\theta_{1}+\theta_{0}\|t}{\sigma}\right)\cdot\d t+\int\varphi(t;0,1)\cosh\left(\frac{\|\theta_{1}-\theta_{0}\|t}{\sigma}\right)\cdot\d t\\
 & \trre[=,d]e^{\frac{\|\theta_{1}+\theta_{0}\|^{2}}{2\sigma^{2}}}+e^{\frac{\|\theta_{1}-\theta_{0}\|^{2}}{2\sigma^{2}}}\\
 & =e^{\frac{\|\theta_{0}\|^{2}+\|\theta_{1}\|^{2}}{2\sigma^{2}}}\cdot\left[e^{\frac{\theta_{1}^{\top}\theta_{0}}{\sigma^{2}}}+e^{\frac{-\theta_{1}^{\top}\theta_{0}}{\sigma^{2}}}\right]\\
 & =e^{\frac{\|\theta_{0}\|^{2}+\|\theta_{1}\|^{2}}{2\sigma^{2}}}\cdot2\cosh\left(\frac{\theta_{1}^{\top}\theta_{0}}{\sigma^{2}}\right).
\end{align}
where $(a)$ follows from the change of variables $y\to\frac{y}{\sigma}$,
$(b)$ follows from the identity $2\cosh(x)\cosh(y)=\cosh(x+y)+\cosh(x-y)$,
and $(c)$ follows from rotational invariance, and $(d)$ follows
from 
\[
\int\varphi(t;0,1)\cosh(at)\d t=\int\varphi(t;0,1)\frac{e^{at}+e^{-at}}{2}\d t=e^{\frac{a^{2}}{2}}.
\]
Continuing (\ref{eq: first bound on chi-square divergence btween two GM}),
we thus have
\begin{align}
 \Dchis\left(\tilde{P}_{\theta_{1}}\mid\mid\tilde{P}_{\theta_{0}}\right) 
 & \leq e^{\frac{\|\theta_{0}\|^{2}}{2\sigma^{2}}}\cdot\left[\cosh\left(\frac{\|\theta_{1}\|^{2}}{\sigma^{2}}\right)-2\cosh\left(\frac{\theta_{1}^{\top}\theta_{0}}{\sigma^{2}}\right)+\cosh\left(\frac{\|\theta_{0}\|^{2}}{\sigma^{2}}\right)\right].\label{eq: a bound on chi-square with cosh}
\end{align}
Now, $\frac{\d}{\d x}\cosh(x)=\sinh(x)$ and $\frac{\d^{2}}{\d x^{2}}\cosh(x)=\cosh(x)$.
In addition, if $0\leq x\leq1$ then it can be easily verified that
$\sinh(x)\leq2x$. Thus, if $0\leq x\leq y\leq1$ then 
\begin{equation}
\cosh(y)-\cosh(x)=\int_{x}^{y}\frac{\d}{\d r}\cosh(r)\d r=\int_{x}^{y}\sinh(r)\d r\leq\int_{x}^{y}2r\d r=y^{2}-x^{2}.\label{eq: cosh around zero is quadratic}
\end{equation}
Moreover, $e^{x}\leq1+2x$ for $x\in[0,1]$. Thus, if we assume $\|\theta_{0}\|=\|\theta_{1}\|\leq\sigma$
we may further upper bound (\ref{eq: a bound on chi-square with cosh})
as 
\begin{align}
 & \Dchis\left(\tilde{P}_{\theta_{1}}\mid\mid\tilde{P}_{\theta_{0}}\right)\nonumber \\
 & \le e^{\frac{\|\theta_{0}\|^{2}}{2\sigma^{2}}}\cdot\left[\cosh\left(\frac{\|\theta_{0}\|^{2}}{\sigma^{2}}\right)-\cosh\left(\frac{|\theta_{1}^{\top}\theta_{0}|}{\sigma^{2}}\right)+\cosh\left(\frac{\|\theta_{0}\|^{2}}{\sigma^{2}}\right)-\cosh\left(\frac{|\theta_{1}^{\top}\theta_{0}|}{\sigma^{2}}\right)\right]\\
 & \trre[\leq,a]2\cdot\left[\cosh\left(\frac{\|\theta_{0}\|^{2}}{\sigma^{2}}\right)-\cosh\left(\frac{|\theta_{1}^{\top}\theta_{0}|}{\sigma^{2}}\right)+\cosh\left(\frac{\|\theta_{0}\|^{2}}{\sigma^{2}}\right)-\cosh\left(\frac{|\theta_{1}^{\top}\theta_{0}|}{\sigma^{2}}\right)\right]\\
 & \trre[\leq,b]2\cdot\frac{\|\theta_{1}\|^{4}-(\theta_{1}^{\top}\theta_{0})^{2}+\|\theta_{0}\|^{4}-(\theta_{1}^{\top}\theta_{0})^{2}}{\sigma^{4}},\label{eq: second bound on chi suqare with  the quadratic term of cosh}
\end{align}
where
$(a)$
follows since $e^{\frac{\|\theta_{0}\|^{2}}{2\sigma^{2}}}\leq2$
under the assumption $\|\theta_{0}\|^{2}\leq\sigma^{2}$, $(b)$ follows
from (\ref{eq: cosh around zero is quadratic}). Now, the numerator
of (\ref{eq: second bound on chi suqare with  the quadratic term of cosh})
is further upper bounded as
\begin{align}
&\|\theta_{1}\|^{4}-2(\theta_{1}^{\top}\theta_{0})^{2}+\|\theta_{0}\|^{4} \\
&= (\theta_1^\top\theta_1)^2 - (\theta_1^\top\theta_0)(\theta_0^\top\theta_1) - (\theta_0^\top\theta_1)(\theta_1^\top\theta_0) + (\theta_0^\top\theta_0)^2 \\
& =\Tr\left[\theta_{0}\theta_{0}^{\top}\theta_{0}\theta_{0}^{\top}-\theta_{0}\theta_{0}^{\top}\theta_{1}\theta_{1}^{\top}-\theta_{1}\theta_{1}^{\top}\theta_{0}\theta_{0}^{\top}+\theta_{1}\theta_{1}^{\top}\theta_{1}\theta_{1}^{\top}\right]\\
&= \Tr[(\theta_{0}\theta_{0}^{\top} - \theta_{1}\theta_{1}^{\top})^2] \\
 & =\|\theta_{0}\theta_{0}^{\top}-\theta_{1}\theta_{1}^{\top}\|_{F}^{2}\\
 & =\|\theta_{0}\theta_{0}^{\top}-\theta_{0}\theta_{1}^{\top}+\theta_{0}\theta_{1}^{\top}-\theta_{1}\theta_{1}^{\top}\|_{F}^{2}\\
 & \leq2\|\theta_{0}\theta_{0}^{\top}-\theta_{0}\theta_{1}^{\top}\|_{F}^{2}+2\|\theta_{0}\theta_{1}^{\top}-\theta_{1}\theta_{1}^{\top}\|_{F}^{2}\\
 & =2\|\theta_{0}\|^{2}\cdot\|\theta_{0}-\theta_{1}\|^{2}+2\|\theta_{1}\|^{2}\cdot\|\theta_{0}-\theta_{1}\|^{2},
\end{align}
where the inequality follows from $\|A+B\|_{F}^{2}\leq2\|A\|_{F}^{2}+2\|B\|_{F}^{2}$.
Inserting this bound into (\ref{eq: second bound on chi suqare with  the quadratic term of cosh}),
and using $\|\theta_{0}\|=\|\theta_{1}\|=t$ results the bound (\ref{eq: a bound on chi square for GM in the Fano argument lemma}). 
\end{proof}
With the results of the three steps at hand, we may complete the proof
of Theorem \ref{thm: impossibility lower bound for mean estimation}.
\begin{proof}[Proof of Theorem \ref{thm: impossibility lower bound for mean estimation}.]
Recall that we assume the low-dimension regime $3\le d\leq\delta n$. 
The condition $ d\ge3 $ can be relaxed to $ d\ge2 $, as promised in Theorem \ref{thm: impossibility lower bound for mean estimation}, using a different construction of the packing set. 
We leave this refinement to the end of this section. 
If $t\geq\sqrt{\delta}$ then the lower bound for the Gaussian location
model (\ref{eq: minimax rates -- Gaussian location model}) implies
a bound of $\Theta(\sqrt{\frac{d}{n}})$. This can be verified by
separately checking the only two cases $t\geq\sqrt{\frac{d}{n}}\geq\sqrt{\delta}$
and $t\geq\sqrt{\delta}\geq\sqrt{\frac{d}{n}}$ possible in the low
dimension regime. We thus henceforth may concentrate on the regime
$t\leq\sqrt{\delta}$. To continue, we henceforth assume the slightly
stronger requirement $t\leq\sqrt{\frac{\delta}{\log n}}=\sqrt{\frac{1}{k}}$,
where $\frac{1}{k}$ is the variance in the genie-aided reduced model. 

Let $\overline{\delta}=\frac{1-\rho^{k}}{2}$ and recall that $\overline{P}_{\theta_{m}}^{(\ell)}$
is the Gaussian model with Markovian signs with flip probability $\overline{\delta}$.
Further let $\varphi(y_{1}^{\ell};\mu)$ be the Gaussian PDF for $\ell$
samples from a $d_{0}$-dimensional model with mean $\mu\in(\mathbb{R}^{d_{0}})^{\ell}$
and covariance matrix $\Sigma=\frac{1}{k}I_{d_{0}}\otimes I_{\ell}\in\mathbb{R}_{+}^{d_{0}\ell\times d_{0}\ell}$.
With this notation, it holds that 
\[
\overline{P}_{\theta_{m}}^{(\ell)}(y_{1}^{\ell})=\sum_{r_{1}^{\ell}\in\{-1,1\}^{\ell}}p_{\overline{\delta}}(r_{1}^{\ell})\cdot\varphi(y_{1}^{\ell};r_{1}^{\ell}\otimes\theta_{m}).
\]
Similarly, let $\tilde{P}_{\theta_{m}}^{(\ell)}$ be the Gaussian
model with Markovian signs with flip probability $1/2$, that is,
a Gaussian mixture model (which is, in fact, memoryless), 
\begin{align}
\tilde{P}_{\theta_{m}}^{(\ell)}(y_{1}^{\ell}) & =\sum_{r_{1}^{\ell}\in\{-1,1\}^{\ell}}p_{1/2}(r_{1}^{\ell})\cdot\varphi(y_{1}^{\ell};r_{1}^{\ell}\otimes\theta_{m})\\
 & =\prod_{i=1}^{\ell}\left[\frac{1}{2}\varphi(y_{i};\theta_{m})+\frac{1}{2}\varphi(y_{i};-\theta_{m})\right].\label{eq: Gaussian Markov with half as product measure}
\end{align}
Now, Lemma \ref{lem: ratio between Markov with flip close to half and GM}
implies that 
\[
1-\frac{1}{n}\leq\min_{r_{1}^{\ell}\in\{-1,1\}^{\ell}}\frac{p_{\overline{\delta}}(r_{1}^{\ell})}{p_{1/2}(r_{1}^{\ell})}\leq\frac{\overline{P}_{\theta_{m}}^{(\ell)}(y_1^\ell)}{\tilde{P}_{\theta_{m}}^{(\ell)}(y_1^\ell)}\leq\max_{r_{1}^{\ell}\in\{-1,1\}^{\ell}}\frac{p_{\overline{\delta}}(r_{1}^{\ell})}{p_{1/2}(r_{1}^{\ell})}\leq1+\frac{2}{n},
\]
and hence 
\begin{align}
\Dkl\left(\overline{P}_{\theta_{m}}^{(\ell)}\mid\mid\overline{P}_{\theta_{0}}^{(\ell)}\right) & \trre[\leq,a]\Dkl\left(\tilde{P}_{\theta_{m}}^{(\ell)}\mid\mid\tilde{P}_{\theta_{0}}^{(\ell)}\right)+\log\left[\left(1+\frac{2}{n}\right)\left(\frac{1}{1 - \frac{1}{n}}\right)\right]\\
&\trre[\le,b] \Dkl\left(\tilde{P}_{\theta_{m}}^{(\ell)}\mid\mid\tilde{P}_{\theta_{0}}^{(\ell)}\right)+2\log\left(1+\frac{2}{n}\right)\\
 & \leq\Dkl\left(\tilde{P}_{\theta_{m}}^{(\ell)}\mid\mid\tilde{P}_{\theta_{0}}^{(\ell)}\right)+\frac{4}{n}\\
 & \trre[=,c]\ell\cdot\Dkl\left(\tilde{P}_{\theta_{m}}\mid\mid\tilde{P}_{\theta_{0}}\right)+\frac{4}{n}\\
 & \trre[\leq,d]\ell\cdot\Dchis\left(\tilde{P}_{\theta_{m}}\mid\mid\tilde{P}_{\theta_{0}}\right)+\frac{4}{n}\\
 & \trre[\leq,e]8\frac{\ell t^{2}}{\sigma^{4}}\cdot\|\theta_{0}-\theta_{1}\|^{2}+\frac{4}{n}\\
 & \trre[\leq,f]32\log(n)\cdot\frac{nt^{4}\epsilon^{2}}{\delta}+\frac{4}{n}
\end{align}
where $(a)$ follows from Lemma \ref{lem: change of measure bound for KL },
$ (b) $ follows since $ \frac{1}{1 - \frac{1}{n}}\le1 + \frac{2}{n} $ for $ n\ge2 $, 
$(c)$ follows from the tensorization property of the KL divergence,
$(d)$ follows from the fact (cf.\ \citep[Eq.\ (2.27)]{tsybakov2008introduction}) that $ \Dkl(P\mid\mid Q)\le\Dchis(P\mid\mid Q) $ for any pair of probability measures $P$ and $Q$, 
$(e)$ follows from Lemma \ref{lem: chi-square for d dimensional Gaussian mixture},
and $(f)$ follows from property (iv) of the packing set $\|\theta_{0}-\theta_{1}\|^{2}\leq4t^{2}\epsilon^{2}$
and $\sigma^{2}=\frac{1}{k}=\frac{\delta}{\log n}$ and $\ell=\frac{n}{k}$. 

Recall that from Fano's argument (\ref{eq: Fano based lower bound for Markov case with KL}),
the largest $\epsilon>0$ so that $\Dkl(\overline{P}_{\theta_{m}}^{(\ell)}\mid\mid\overline{P}_{\theta_{0}}^{(\ell)})\leq\frac{d}{16}$
assures the bound $\M(n,d,\delta,t)\geq\frac{\epsilon t}{128}$. Assuming
that $\frac{4}{n}\leq\frac{d}{32}$, that is $n\geq\frac{128}{d}$,
this will occur if $32\log(n)\cdot\frac{nt^{4}\epsilon^{2}}{\delta}\leq\frac{d}{32}$.
This can be achieved by the choice $\epsilon=\frac{1}{32\sqrt{\log(n)}}\cdot\min\left\{ 1,\frac{1}{t^{2}}\sqrt{\frac{d\delta}{n}}\right\} $.
Using this value in the Fano's based bound then completes the proof
of the theorem.
\end{proof}

\paragraph*{Relaxing $d\ge3$ to $d\ge2$}
In the above proof, we assumed that $ d\ge3 $. 
This condition can be relaxed to $ d\ge2 $ by using a more careful construction of the packing set. 
\begin{lem}
\label{lem:projective-spherical-code}
Let $ d\ge2 $ be an integer and $ \alpha\in(0,\frac{\pi}{2}] $ be an angle. 
Then there exists a $ \left(2\sin\left(\frac{\alpha}{2}\right)\right) $-packing on $ \mathbb{S}^{d-1} $ of size at least $ \frac{\cos(\alpha)}{\sin^{d-1}(\alpha)} $. 
\end{lem}

\begin{proof}
We will greedily construct the desired packing set. 
Let $ \Delta = 2\sin\left(\frac{\alpha}{2}\right) $. 
Start with an arbitrary point on $ \mathbb{S}^{d-1} $. 
In each of the following steps, put into the packing set another arbitrary point that is $ \Delta $-far from any existing points in the packing set and their antipodal points. 
Repeat this process until no more points can be put without violating the distance guarantee. 
Note that this construction guarantees that for any pair of distinct points $ \theta,\theta' $ in the packing set, 
\begin{align}
\loss(\theta,\theta')
= \min\{\|\theta - \theta'\|, \|\theta + \theta'\|\}
> 2\sin\left(\frac{\alpha}{2}\right) = \Delta . 
\end{align}

We then lower bound the cardinality of the above packing set. 
For $ \alpha\in[0,\frac{\pi}{2}] $, let $ S(\alpha) $ denote the surface area of a spherical cap on $ \mathbb{S}^{d-1} $ of angular radius $ \alpha $. 
For $ \theta\in\mathbb{S}^{d-1} $ and $ \alpha\in[0,\frac{\pi}{2}] $, let $ \mathbb{K}(\theta,\alpha)\dfn\{\theta'\in\mathbb{S}^{d-1}:|\langle\theta,\theta'\rangle|\ge\cos(\alpha)\} $ denote the union of two spherical caps centered around $ \theta $ and $ -\theta $, respectively, of angular radius $ \alpha $ each. 
Note that each new point $ \theta $ in the construction induces a \emph{forbidden region} on $ \mathbb{S}^{d-1} $ of surface area at most $ |\mathbb{K}(\theta,\alpha)| = 2S(\alpha) $ in which following points cannot lie, 
since any point within $ \mathbb{K}(\theta,\alpha) $ has distance at most $ 2\sin\left(\frac{\alpha}{2}\right) = \Delta $ to either $ \theta $ or $ -\theta $. 
Also, the surface area of the forbidden region is \emph{upper} bounded by $ 2S(\alpha) $ since $ \mathbb{K}(\theta,\alpha) $ may overlap with $ \mathbb{K}(\theta',\alpha) $ for some previous $ \theta' $.
Therefore, by the step at which the greedy construction terminates, one must have put at least $ \frac{2S(\frac{\pi}{2})}{2S(\alpha)} = \frac{S(\frac{\pi}{2})}{S(\alpha)} $ many points, where the numerator is nothing but the surface area of $ \mathbb{S}^{d-1} $. 
\citep[Eq.\ (1)]{blachman-few} upper bounds the area ratio between two spherical caps on $ \mathbb{S}^{d-1} $ as follows:
\begin{align}
\frac{S(\alpha)}{S(\beta)} &= \frac{\int_0^\alpha\sin^{d-2}(x)\d x}{\int_0^\beta \sin^{d-2}(x)\d x} 
\le \sec(\alpha)\cdot\frac{\sin^{d-1}(\alpha)}{\sin^{d-1}(\beta)} . 
\end{align}
Using this bound, we conclude that the cardinality of the constructed packing set is at least $ \frac{\cos(\alpha)}{\sin^{d-1}(\alpha)} $, which finishes the proof. 
\end{proof}
By Lemma \ref{lem:projective-spherical-code}, a $ \Delta $-packing set can be obtained by setting $ 2\sin\left(\frac{\alpha}{2}\right) = \Delta $. 
Therefore, 
\begin{gather}
\sin\left(\frac{\alpha}{2}\right) = \frac{\Delta}{2} , \quad 
\cos\left(\frac{\alpha}{2}\right) = \sqrt{1 - \frac{\Delta^2}{4}} , \\
\sin(\alpha) = 2\sin\left(\frac{\alpha}{2}\right)\cos\left(\frac{\alpha}{2}\right) = \Delta\sqrt{1 - \frac{\Delta^2}{4}} , \quad 
\cos(\alpha) = 1 - 2\sin^2\left(\frac{\alpha}{2}\right) = 1 - \frac{\Delta^2}{2} , 
\end{gather}
and the cardinality of the packing set is at least
\begin{align}
\left(1 - \frac{\Delta^2}{2}\right)\left(\Delta\sqrt{1 - \frac{\Delta^2}{4}}\right)^{-(d-1)} . 
\end{align}
Setting $ \Delta\sqrt{1 - \frac{\Delta^2}{4}} = \frac{1}{8} $, we have $ \Delta = \frac{1}{2}\sqrt{8 - 3\sqrt{7}} \ge \frac{1}{8} $ and we get a $\frac{1}{8}$-packing set $ \Theta_M = \{\theta_m\}_{m\in[M]} $ of size at least 
$M \ge \frac{3\sqrt{7}}{8}\cdot 8^{-(d-1)} \ge 0.99\cdot 8^{-(d-1)}$.

The above construction can be used in place of the one described at the beginning of Appendix \ref{app:proof-converse-estimate-theta}. 
The resulting lower bound on the minimax error rate follows from similar reasoning with suitably adjusted numerical constants. 
We sketch the rest of the proof below. 
By Fano's method \citep[Proposition 15.2]{wainwright2019high} (see also \eqref{eqn:apply-fano-1}), 
\begin{align}
\M(n,d,\delta,t) &\ge \frac{\epsilon t}{16}\left[1 - \frac{I(J;X_1^n) + \log2}{\log M}\right] \\
&\ge \frac{\epsilon t}{16}\left[1 - \frac{\max_{m\in[M]}\Dkl(P_{\theta_m}^{(n)} \mid\mid P_{\theta_0}^{(n)}) + \log2}{(d-1)\log8+\log0.99}\right] \\
&\ge \frac{\epsilon t}{16}\left[\frac{1}{2} - 4\cdot\frac{\max_{m\in[M]}\Dkl(P_{\theta_m}^{(n)} \mid\mid P_{\theta_0}^{(n)})}{d}\right] , 
\end{align}
where the last inequality follows since $ d\ge2 $ implies (i) $ (d-1)\log8+\log0.99\ge\frac{d}{4} $, (ii) $ \frac{\log2}{(d-1)\log8+\log0.99}\le\frac{\log2}{\log8+\log0.99}\le\frac{1}{2} $. 
The proof of the bound $ \max_{m\in[M]}\Dkl(P_{\theta_m}^{(n)} \mid\mid P_{\theta_0}^{(n)})\le\frac{d}{16} $ can be completely reused which implies $ \M(n,d,\delta,t)\ge\frac{\epsilon t}{64} $.
The same choice of $ \epsilon $ then yields the desired lower bound on the minimax error rate in Theorem \ref{thm: impossibility lower bound for mean estimation} for any $ d\ge2 $. 

\section{Proofs for Section \ref{sec:Flip-probability-estimation}: Estimation
of $\delta$ for a given estimate of $\theta_{*}$ \label{sec:Proofs for delta estimation}}

\subsection{Proof of Theorem \ref{thm: estimation error of delta for mismatched delta}:
Analysis of the estimator}
\label{app:proof-achievability-estimate-delta}

\begin{proof}[Proof of Theorem \ref{thm: estimation error of delta for mismatched delta}.]
We first note that since the estimator \eqref{eq: estimator for rho} does not exploit the correlation between $ X_i$ and $X_j $ for $ |i - j|\ge2 $, we may assume that the $ n/2 $ pairs of random variables $ (S_{2i-1},S_{2i}) $ are independent by multiplying each pair of samples $ (X_{2i-1},X_{2i}) $ by an i.i.d.\ random sign. 
By explicitly using $X_{i}=S_{i}\theta_{*}+Z_{i}$ in the definition
of $\hat{\rho}$, and using the triangle inequality,  we obtain that
\begin{multline}
\left|\hat{\rho}-\rho\right|\leq\left|\frac{\|\theta_{*}\|^{2}}{\|\theta_{\sharp}\|^{2}}\frac{2}{n}\sum_{i=1}^{n/2}S_{2i}S_{2i-1}-\rho\right|+\left|\frac{2}{\|\theta_{\sharp}\|^{2}n}\sum_{i=1}^{n/2}S_{2i}\theta_{*}^{\T}Z_{2i-1}\right|\\
+\left|\frac{2}{\|\theta_{\sharp}\|^{2}n}\sum_{i=1}^{n/2}S_{2i-1}\theta_{*}^{\T}Z_{2i}\right|+\left|\frac{2}{\|\theta_{\sharp}\|^{2}n}\sum_{i=1}^{n/2}Z_{2i}^{\T}Z_{2i-1}\right|.\label{eq: estimator for rho explicit}
\end{multline}
We begin with the analysis of the first term. We note that 
\[
\frac{2}{n}\sum_{i=1}^{n/2}S_{2i}S_{2i-1}\eqd\frac{2}{n}\sum_{i=1}^{n/2}R_{i}
\]
where 
\[
R_{i}=\begin{cases}
1, & \text{w.p. }\frac{1+\rho}{2}\\
-1, & \text{w.p. }\frac{1-\rho}{2}
\end{cases}
\]
and $\{R_{i}\}_{i\in[n/2]}$ are i.i.d. with $\E[R_{i}]=1-2\delta=\rho$
and $\V[R_{i}]=4\delta(1-\delta)\leq4\delta$. Now, the first term
in (\ref{eq: estimator for rho explicit}) is bounded by 
\begin{align}
 & \left|\frac{\|\theta_{*}\|^{2}}{\|\theta_{\sharp}\|^{2}}\frac{2}{n}\sum_{i=1}^{n/2}R_{i}-\rho\right|\nonumber \\
 & =\left|\left(1+\frac{\|\theta_{*}\|^{2}-\|\theta_{\sharp}\|^{2}}{\|\theta_{\sharp}\|^{2}}\right)\frac{2}{n}\sum_{i=1}^{n/2}R_{i}-\rho\right|\\
 & \leq\left|\frac{2}{n}\sum_{i=1}^{n/2}R_{i}-\rho\right|+\frac{\left|\|\theta_{*}\|^{2}-\|\theta_{\sharp}\|^{2}\right|}{\|\theta_{\sharp}\|^{2}},
\end{align}
where the last inequality follows since $\left|\frac{2}{n}\sum_{i=1}^{n/2}R_{i}\right|\leq1$.
In this last equation, the last term can be considered a bias of the
estimator due to the mismatch between $\theta_{*}$ and $\theta_{\sharp}$.
Now, by Bernstein's inequality for a sum of independent, zero-mean,
bounded random variables $R_{i}-\E[R_{i}]\in[-2,2]$ (cf.\ \eqref{eqn:bernstein} from \citep[Propostion 2.14]{wainwright2019high})
\[
\P\left[\left|\frac{2}{n}\sum_{i=1}^{n/2}R_{i}-\E[R_{i}]\right|\geq t\right]\leq2\exp\left(-\frac{\frac{n}{2}t^{2}}{2\left(4\delta+\frac{2t}{3}\right)}\right),
\]
and so with probability larger than $1-\epsilon$
\begin{align}
\left|\frac{2}{n}\sum_{i=1}^{n/2}R_{i}-\E[R_{i}]\right| &\le \frac{4}{3n}\log\left(\frac{2}{\epsilon}\right) + \frac{4}{3n}\sqrt{9n\delta\log\left(\frac{2}{\epsilon}\right) + \log^2\left(\frac{2}{\epsilon}\right)} \\
&\leq\sqrt{\frac{16\delta\log\left(\frac{2}{\epsilon}\right)}{n}}+\frac{8\log\left(\frac{2}{\epsilon}\right)}{3n} \\ 
& \leq7\log\left(\frac{2}{\epsilon}\right)\sqrt{\frac{\delta}{n}},
\end{align}
since $\delta\geq\frac{1}{n}$. 

The second term in (\ref{eq: estimator for rho explicit}) (and similarly,
the third term in (\ref{eq: estimator for rho explicit})) is
\[
\frac{2}{\|\theta_{\sharp}\|^{2}n}\sum_{i=1}^{n/2}S_{2i}\theta_{\sharp}^{\T}Z_{2i-1}\sim N\left(0,\frac{2}{\|\theta_{\sharp}\|^{2}n}\right).
\]
Thus, by the standard Chernoff bound for Gaussian random variables
\[
\P\left[\left|\frac{2}{\|\theta_{\sharp}\|^{2}n}\sum_{i=1}^{n/2}S_{2i}\theta_{\sharp}^{\T}Z_{2i-1}\right|>t\right]\leq2e^{-\frac{\|\theta_{\sharp}\|^{2}nt^{2}}{4}},
\]
and so with probability larger than $1-\epsilon$
\[
\left|\frac{2}{\|\theta_{\sharp}\|^{2}n}\sum_{i=1}^{n/2}S_{2i}\theta_{\sharp}^{\T}Z_{2i-1}\right|\leq\sqrt{\frac{4}{n\|\theta_{\sharp}\|^{2}}\log\left(\frac{2}{\epsilon}\right)}.
\]
The fourth term in (\ref{eq: estimator for rho explicit}) satisfies
\[
\frac{2}{\|\theta_{\sharp}\|^{2}n}\sum_{i=1}^{n/2}Z_{2i}^{\T}Z_{2i-1}\eqd\frac{2}{\|\theta_{\sharp}\|^{2}n}\sum_{i=1}^{nd/2}W_{i}\tilde{W}_{i}
\]
where $\{W_{i}\}_{i\in[nd/2]}$ and $\{\tilde{W}_{i}\}_{i\in[nd/2]}$
are i.i.d. and $W_{i},\tilde{W}_{i}\sim N(0,1)$ are independent.
It is well known \citep[Lemma 2.7.7]{vershynin2018high} that the
product of two subGaussian random variables (even if they
are not independent) is sub-exponential. Here, we have a simpler and
exact characterization. Letting 
\[
W_{i}\tilde{W}_{i}=\left(\frac{W_{i}+\tilde{W}_{i}}{2}\right)^{2}-\left(\frac{W_{i}-\tilde{W}_{i}}{2}\right)^{2}\eqd V_{i}^{2}-\tilde{V}_{i}^{2}
\]
and since $V_{i}=\frac{1}{2}(W_{i}+\tilde{W}_{i})\sim N(0,1)$ and
$\tilde{V}_{i}=\frac{1}{2}(W_{i}-\tilde{W}_{i})\sim N(0,1)$ are uncorrelated,
they are independent. Letting $\chi_{nd/2}^{2}$ and $\tilde{\chi}_{nd/2}^{2}$
be a pair of independent chi-square random variables with $nd/2$
degrees of freedom, it then holds that
\begin{align}
\frac{2}{\|\theta_{\sharp}\|^{2}n}\sum_{i=1}^{n/2}Z_{2i}^{\T}Z_{2i-1} & \eqd\frac{2}{\|\theta_{\sharp}\|^{2}n}\left(\chi_{nd/2}^{2}-\tilde{\chi}_{nd/2}^{2}\right)\\
 & =\frac{2}{\|\theta_{\sharp}\|^{2}n}\left(\chi_{nd/2}^{2}-\E[\chi_{nd/2}^{2}]+\E[\tilde{\chi}_{nd/2}^{2}]-\tilde{\chi}_{nd/2}^{2}\right).
\end{align}
From the chi-square tail bound in (\ref{eq: chi-square upper tail bound})
and (\ref{eq: chi-square lower tail bound}) it holds that 
\[
\P\left[\left|\chi_{nd/2}^{2}-\E[\chi_{nd/2}^{2}]\right|\geq2\sqrt{nd/2t}+2t\right]\leq 2e^{-t}
\]
and so it holds with probability $1-\epsilon $ that 
\[
\left|\chi_{nd/2}^{2}-\E[\chi_{nd/2}^{2}]\right|\leq\sqrt{2nd\log\left(\frac{2}{\epsilon}\right)}+2\log\left(\frac{2}{\epsilon}\right)\leq4\sqrt{nd}\log\left(\frac{2}{\epsilon}\right).
\]
Hence, by the union bond 
\[
\frac{2}{\|\theta_{\sharp}\|^{2}n}\sum_{i=1}^{n/2}Z_{i}^{\T}Z_{i-1}\leq\frac{16\log\left(\frac{2}{\epsilon}\right)}{\|\theta_{\sharp}\|^{2}}\sqrt{\frac{d}{n}}.
\]
with probability $1-\epsilon $. The claim (\ref{eq: high probability bound on delta estimation})
follows from the analysis of the terms above, the choice of $ \epsilon = \frac{2}{n} $ and a union bound.
\end{proof}

\subsection{Proof of Proposition \ref{prop: impossibility lower bound for flip probability estimation}:
Impossibility lower bound}
\label{app:proof-converse-converse-estimate-delta}

\paragraph*{The effect of knowledge of $\theta_{*}$}

To begin, it is apparent that if $\|\theta_{*}\|^{2}=\|\theta_{\sharp}\|^{2}$,
the first term in (\ref{eq: high probability bound on delta estimation})
vanishes, and the loss is bounded by the remaining terms. Furthermore,
if $\|\theta_{*}\|\leq1$ then the dominant term in the brackets is
$\frac{1}{\|\theta_{*}\|^{2}}\sqrt{\frac{d}{n}}$, and evidently,
dominant error term suffers from a penalty of $\sqrt{d}$, even though,
essentially, $\rho$ is a one-dimensional parameter. If $\theta_{*}$
is known exactly up to a sign, that is $\theta_{\sharp}=\pm\theta_{*}$
(and it is not just their norms which are equal), then the estimation
error can be reduce to the case of $d=1$. This is a simple consequence
of the rotational invariance of the distribution of the Gaussian noise,
which implies that the projections $ (\pm\theta_{*}^{\T}X_{i})_{i = 1}^n $
are sufficient statistics for the estimation of $\rho$. The estimator
constructs the projections 

\begin{equation}
U_{i}\dfn\frac{\theta_{\sharp}^{\T}X_{i}}{\|\theta_{\sharp}\|}=\pm\|\theta_{*}\|\cdot S_{i}+\frac{\pm\theta_{*}^{\T}Z_{i}}{\|\theta_{*}\|}\eqd\|\theta_{*}\|\cdot S_{i}+W_{i},\label{eq: model for estimating delta given known theta_star}
\end{equation}
where $W_{i}\sim N(0,1)$. This is effectively a one-dimensional model
with parameter given by $\|\theta_{*}\|$, and so Corollary \ref{cor: estimation error of delta for mismatched delta known mean} immediately follows
from Theorem \ref{thm: estimation error of delta for mismatched delta}.

\paragraph*{Tightness of the impossibility lower bound }

Evidently, Corollary \ref{cor: estimation error of delta for mismatched delta known mean}
and Proposition \ref{prop: impossibility lower bound for flip probability estimation}
match in their dependence on the number of samples $\Theta(\frac{1}{\sqrt{n}})$,
but it is not clear what the optimal dependence on $\|\theta_{*}\|$ is.
The estimator we propose is based on the moment $\rho=\E[U_{i}U_{i-1}]$
and can be contrasted with likelihood based methods as follows. Letting
$p_{\delta}(s_{0}^{n})$ denote the probability of the sign sequence
$s_{0}^{n}$ with flip probability $\delta$, and letting $\varphi(x)$
denote the standard Gaussian density, the likelihood of $u_{1}^{n}$
is given by 
\[
P_{\delta,\theta_{*}}(u_{1}^{n})=\sum_{s_{0}^{n}\in\{\pm1\}^{n+1}}p_{\delta}(s_{0}^{n})\cdot\prod_{i\in[n]}\varphi(u_{i}-s_{i}\|\theta_{*}\|).
\]
This function is a large degree polynomial in $\delta$ on the order
of $n$. Even if one sums only over $s_{0}^{n}$ with the typical
number of flips, then this degree is $\Theta(\delta n)$, which means this
polynomial has a degree which blows up with $n$. Thus, the MLE may
indeed be sensitive to empirical errors. The update of the Baum-Welch
algorithm (or EM) can also be easily computed and contrasted with
our proposed estimator. Letting $P_{\delta,\theta_{*}}$ denote the
probability distribution of the corresponding model with
flip probability $\delta$ and mean parameter $\theta_{*}$, the Baum-Welch
estimate $\hat{\delta}^{(j)}$ at iteration $j$ is given by 
\[
\hat{\delta}^{(j)}=\frac{1}{n}\sum_{i=1}^{n}P_{\hat{\delta}^{(j-1)},\theta_{*}}(S_{i}\neq S_{i-1}\mid u_{1}^{n}).
\]
Evidently, the inner estimate is the probability that $S_{i}\neq S_{i-1}$
conditioned on the entire sample $u_{1}^{n}$, rather than just $u_{i},u_{i-1}$
as in our proposed estimator in (\ref{eq: estimator for rho}). However,
deriving sharp error rate bounds for this estimator seems to be a
challenging task. 

In terms of minimax lower bounds, the difficulty arises since a reduction
to an (almost) memoryless GMM as in lower bound for the estimation of $\theta_{*}$
in Theorem \ref{thm: impossibility lower bound for mean estimation}
does not seem fruitful. A standard application of Le-Cam's method
requires bounding the total variation between models $P_{\delta_{1},\theta_{*}}$and
$P_{\delta_{2},\theta_{*}}$ for some ``close'' $\delta_{1},\delta_{2}$
(e.g., $\delta_{1}=0$ and $\delta_{2}=\epsilon$). The total variation
is then typically bounded by the KL divergence. However, the KL divergence
does not tensorize (due to the memory), and using the chain rule requires
evaluating the KL divergence for the process $U_{1}^{n}$ which is
not a Markov process (but rather an HMM). Alternatively, further bounding
the $n$-dimensional KL divergence with a chi-square divergence, which
is a convenient choice for mixture models, leads to an excessively
large bound. 

\paragraph*{The proof's main ideas}

The proof is based on Le-Cam's method that is applied to a genie-aided
model in which the estimator knows every other sign $S_{0},S_{2},S_{4},\ldots,S_{n}$.
The main technical challenge is then to bound the total variation
for this genie-aided model. 
A complete proof is presented below.

\begin{proof}[Proof of Proposition \ref{prop: impossibility lower bound for flip probability estimation}.]
For a given $\delta\in[0,1]$ let $P_{\delta}^{(n)}$ denote the
probability distribution of $U_{1}^{n}$ under the model (\ref{eq: model for estimating delta given known theta_star}),
that is $U_{i}=t\cdot S_{i}+W_{i}$, $i\in[n]$, where $t\equiv\|\theta_{*}\|$
and $S_{i}$ is a binary symmetric Markov chain with flip probability
$\delta$, and $\P[S_{0}=1]=\frac{1}{2}$. It should be noted that
$P_{\delta}^{(n)}$ is a distribution on $\mathbb{R}^{n}$ which is
\emph{not} a product distribution. To lower bound the estimation error
of an estimator $\hat{\delta}(U_{1}^{n})$, we consider a genie-aided
estimator which is informed with the values of $S_{0},S_{2},S_{4},\ldots S_{n}$
(assuming for simplicity that $n$ is even). We then set $\epsilon\in(0,\frac{1}{2})$
and use Le-Cam's two point method \citep[Section 15.2]{wainwright2019high}
with $\delta_{0}=\frac{1}{2}$ and $\delta_{1}=\frac{1}{2}-\epsilon$.
Let us denote by $Q_{U_{1}^{n}S_{0}^{n}}$ (resp. $P_{U_{1}^{n}S_{0}^{n}}$)
the joint probability distribution of $U_{1}^{n}$ and $S_{0}^{n}$
under $\delta=\frac{1}{2}$ (resp. $\delta=\frac{1}{2}-\epsilon$),
and marginals and conditional versions by standard notation, e.g.,
$Q_{U_{2}U_{4}\mid S_{0}S_{2}}$. The proof of the proposition follows
from Le-Cam's two point method, which states that for any estimator
$\hat{\delta}(U_{1}^{n},S_{0},S_{2},S_{4},\ldots,S_{n})$, and thus
also for any less informed estimator $\hat{\delta}(U_{1}^{n})$,
\begin{equation}
\E\left[\left|\hat{\delta}-\delta\right|\right]\geq\frac{\epsilon}{2}\cdot\left(1-\dtv\left(P_{U_{1}U_{2}\cdots U_{n}S_{0}S_{2}S_{4}\cdots S_{n}},Q_{U_{1}U_{2}\cdots U_{n}S_{0}S_{2}S_{4}\cdots S_{n}}\right)\right).\label{eq: impossibility bound on the loss Le Cam}
\end{equation}
We next obtain a bound on the total variation distance in Lemma \ref{lem:A total variation bound for flip estimation}
of $4\sqrt{n}\epsilon$ for $t\leq\frac{1}{\sqrt{2}}$ and choosing
$\epsilon=\frac{1}{8\sqrt{n}}$ in (\ref{eq: impossibility bound on the loss Le Cam})
completes the proof of the lower bound. 
\end{proof}
\begin{lem}
\label{lem:A total variation bound for flip estimation}If $t\leq1/\sqrt{2}$
then 
\[
\dtv\left(P_{U_{1}U_{2}\cdots U_{n},S_{0}S_{2}S_{4}\cdots S_{n}},Q_{U_{1}U_{2}\cdots U_{n},S_{0}S_{2}S_{4}\cdots S_{n}}\right)\leq\sqrt{\frac{5}{2}n}\epsilon+\sqrt{8n}t\epsilon.
\]
\end{lem}

\begin{proof}
By Pinsker's inequality (e.g., \citep[Lemma 2.5]{tsybakov2008introduction}),
it holds for any pair of probability measures $P,Q$ that $\dtv(P,Q)\leq\sqrt{\frac{1}{2}\Dkl(Q,P)}.$
We next upper bound the KL divergence. To this end, recall the chain
rule, that for any joint distributions $P_{XY}$ and $Q_{XY}$ (with
conditional distributions $P_{Y\mid X},Q_{Y\mid X})$ the chain rule
for the KL divergence states that $\Dkl(P_{XY}\mid\mid Q_{XY})=\Dkl(P_{X}\mid\mid Q_{X})+\Dkl(P_{Y\mid X}\mid\mid Q_{Y\mid X}\mid P_{X})$
where $\Dkl(P_{Y\mid X}\mid\mid Q_{Y\mid X}\mid P_{X})\dfn\int\Dkl(P_{Y\mid X}\mid\mid Q_{Y\mid X})\d P_{X}$
is the conditional KL divergence \citep[Theorem 2.5.3]{Cover:2006:EIT:1146355}.
Thus, 
\begin{align}
 & \Dkl\left(P_{U_{1}U_{2}\cdots U_{n},S_{0}S_{2}S_{4}\cdots S_{n}}\mid\mid Q_{U_{1}U_{2}\cdots U_{n},S_{0}S_{2}S_{4}\cdots S_{n}}\right)\nonumber \\
 & =\Dkl\left(P_{S_{0}S_{2}S_{4}\cdots S_{n}}\mid\mid Q_{S_{0}S_{2}S_{4}\cdots S_{n}}\right)\label{eq: KL decomposition to signs and conditional samples 1}\\
 & \hphantom{==}+\Dkl\left(P_{U_{1}U_{2}\cdots U_{n}\mid S_{0}S_{2}S_{4}\cdots S_{n}}\mid\mid Q_{U_{1}U_{2}\cdots U_{n}\mid S_{0}S_{2}S_{4}\cdots S_{n}}\mid P_{S_{0}S_{2}S_{4}\cdots S_{n}}\right).\label{eq: KL decomposition to signs and conditional samples}
\end{align}
We next bound each of the two KL divergences appearing in \eqref{eq: KL decomposition to signs and conditional samples 1} and (\ref{eq: KL decomposition to signs and conditional samples}).
First, 
\begin{align}
\Dkl\left(P_{S_{0}S_{2}S_{4}\cdots S_{n}}\mid\mid Q_{S_{0}S_{2}S_{4}\cdots S_{n}}\right) & \trre[\leq,a]\Dkl\left(P_{S_{0}^{n}}\mid\mid Q_{S_{0}^{n}}\right)\\
 & =\sum_{s_{0}^{n}\in\{-1,1\}^{n+1}}P_{S_{0}^{n}}(s_{0}^{n})\log\frac{P_{S_{0}^{n}}(s_{0}^{n})}{Q_{S_{0}^{n}}(s_{0}^{n})}\\
 & \trre[=,b](n+1)\log2-\sum_{s_{0}^{n}\in\{-1,1\}^{n+1}}P_{S_{0}^{n}}(s_{0}^{n})\log\frac{1}{P_{S_{0}^{n}}(s_{0}^{n})}\\
 & \trre[=,c](n+1)\log2-H(S_{0},S_{1},\ldots, S_{n})\\
 & \trre[=,d](n+1)\log2-H(S_{0})-\sum_{i=1}^{n}H(S_{i}\mid S_{0}^{i-1})\\
 & \trre[=,e]n\cdot\log2-\sum_{i=1}^{n}H(S_{i}\mid S_{i-1})\\
 & \trre[=,f]n\cdot\left[\log2-h_{b}\left(\frac{1}{2}-\epsilon\right)\right],\label{eq: bound on the KL between signs process}
\end{align}
where $(a)$ follows from the convexity of the KL divergence (recall
$S_{0}^{n}=S_{0},S_{1},S_{2},\ldots,S_{n}$), $(b)$ since $Q_{S_{0}^{n}}(s_{0}^{n})=2^{-(n+1)}$
for any $s_{0}^{n}\in\{\pm1\}^{n+1}$, $(c)$ follows by defining
the entropy $H$ of $S_{0}^{n}$ (under the probability measure $P$),
$(d)$ follows from the chain rule of entropy and the definition of
conditional entropy \citep[Theorem 2.2.1]{Cover:2006:EIT:1146355},
$(e)$ follows from Markovity and $H(S_{0})=\log2$, $(f)$ follows
from $H(S_{i}\mid S_{i-1})=h_{b}(\delta_{1})=h_{b}(\frac{1}{2}-\epsilon)$
where $h_{b}(\delta)\dfn-\delta\log\delta-(1-\delta)\log(1-\delta)$
is the binary entropy function. Now, for $\epsilon\in(0,\frac{1}{2})$
the power series expansion of the binary entropy function results
the bound 
\begin{align}
h_{b}\left(\frac{1}{2}-\epsilon\right) & =\log2-\sum_{k=1}^{\infty}\frac{(2\epsilon)^{2k}}{2k(2k-1)}\\
 & \geq\log2-(2\epsilon)^{2}\sum_{k=1}^{\infty}\frac{1}{2k(2k-1)}\\
 & \geq\log2-4\epsilon^{2}\sum_{k=1}^{\infty}\frac{1}{(2k-1)^{2}}\\
 & =\log2-\frac{\pi^{2}}{2}\epsilon^{2}\\
 & \geq\log2-5\epsilon^{2}.
\end{align}
Inserting this bound into (\ref{eq: bound on the KL between signs process})
results that the first term in (\ref{eq: KL decomposition to signs and conditional samples 1})
is upper bounded as 
\begin{equation}
\Dkl\left(P_{S_{0}S_{2}S_{4}\cdots S_{n}}\mid\mid Q_{S_{0}S_{2}S_{4}\cdots S_{n}}\right)\leq5n\epsilon^{2}.\label{eq: bound on the first KL for lower bound delta estiamtion}
\end{equation}

We now move on to bound the second term in (\ref{eq: KL decomposition to signs and conditional samples}).
To this end, we note that the distribution of $U_{1}^{n}$ conditioned
on $S_{0},S_{2},S_{4},\ldots S_{n}$ can be decomposed in a simple
way. First, under $Q$, the signs $S_{0}^{n}$ are i.i.d., and so
$U_{1}^{n}$ is a vector of independent samples from a Gaussian mixture
model. Furthermore, $U_{i}$ depends on the sign $S_{i}$ but otherwise is
independent of all other $S_{1}^{n}\backslash\{S_{i}\}$. Hence,
\[
Q_{U_{1}U_{2}\cdots U_{n}\mid S_{0}S_{2}S_{4}\cdots S_{n}}=Q_{U_{1}}\cdot Q_{U_{2}\mid S_{2}}\cdot Q_{U_{3}}\cdot Q_{U_{4}\mid S_{4}}\cdots Q_{U_{n-1}}\cdot Q_{U_{n}\mid S_{n}},
\]
that is, a model of independent samples, where the odd samples are
drawn from a Gaussian mixture model, and the even samples from a Gaussian
location model with known sign $S_{2i}$. Second, under $P$, an application
of Bayes rule and the Markovity assumption results the decomposition
to pairs of samples given by 
\begin{align}
 & P_{U_{1}U_{2}\cdots U_{n}\mid S_{0}S_{2}S_{4}\cdots S_{n}}\nonumber \\
 & =P_{U_{1}U_{2}\mid S_{0}S_{2}S_{4}\cdots S_{n}}\cdot P_{U_{3}U_{4}\mid S_{0}S_{2}S_{4}\cdots S_{n}U_{1}U_{2}}\cdots P_{U_{n-1}U_{n}\mid S_{0}S_{2}S_{4}\cdots S_{n}U_{1}^{n-2}}\\
 & =P_{U_{1}U_{2}\mid S_{0}S_{2}}\cdot P_{U_{3}U_{4}\mid S_{2}S_{4}}\cdots P_{U_{n-1}U_{n}\mid S_{n-2}S_{n}}.\label{eq: Markovian decmposition of the distribution}
\end{align}
The chain rule for the KL divergence therefore implies that 
\begin{align}
 & \Dkl\left(P_{U_{1}U_{2}\cdots U_{n}\mid S_{0}S_{2}S_{4}\cdots S_{n}}\mid\mid Q_{U_{1}U_{2}\cdots U_{n}\mid S_{0}S_{2}S_{4}\cdots S_{n}}\mid P_{S_{0}S_{2}S_{4}\cdots S_{n}}\right)\nonumber \\
 & \trre[=,a]\Dkl\left(P_{U_{1}U_{2}\mid S_{0}S_{2}}\mid\mid Q_{U_{1}}\cdot Q_{U_{2}\mid S_{2}}\mid P_{S_{0}S_{2}}\right)+\Dkl\left(P_{U_{3}U_{4}\mid S_{2}S_{4}}\mid\mid Q_{U_{3}}\cdot Q_{U_{4}\mid S_{4}}\mid P_{S_{2}S_{4}}\right)\nonumber \\
 & \hphantom{==}+\cdots+\Dkl\left(P_{U_{n-1}U_{n}\mid S_{n-2}S_{n}}\mid\mid Q_{U_{n-1}}\cdot Q_{U_{n}\mid S_{n}}\mid P_{S_{n-2}S_{n}}\right)\\
 & \trre[=,b]\frac{n}{2}\cdot\Dkl\left(P_{U_{1}U_{2}\mid S_{0}S_{2}}\mid\mid Q_{U_{1}}\cdot Q_{U_{2}\mid S_{2}}\mid P_{S_{0}S_{2}}\right)\\
 & \trre[=,c]\frac{n}{2}\cdot\Dkl\left(P_{U_{1}\mid S_{0}S_{2}}\mid\mid Q_{U_{1}}\mid P_{S_{0}S_{2}}\right)+\frac{n}{2}\cdot\Dkl\left(P_{U_{2}\mid S_{0}S_{2}}\mid\mid Q_{U_{2}\mid S_{2}}\mid P_{S_{0}S_{2}}\right)\\
 & \trre[=,d]\frac{n}{2}\cdot\Dkl\left(P_{U_{1}\mid S_{0}S_{2}}\mid\mid Q_{U_{1}}\mid P_{S_{0}S_{2}}\right)\\
 & \trre[=,e]\frac{n}{2}\cdot\Dkl\left(P_{U_{1}\mid S_{0}=1,S_{2}}\mid\mid Q_{U_{1}}\mid P_{S_{2}\mid S_{0}=1}\right)\label{eq: KL divergence for delta estimation lower bound second term derivation}
\end{align}
where $(a)$ follows from (\ref{eq: Markovian decmposition of the distribution})
and the chain rule for KL divergence, and $(b)$ follows from the
stationarity of the Markov chain, $(c)$ follows again from the chain
rule, $(d)$ follows since $\Dkl(P_{U_{2}\mid S_{0}S_{2}}\mid\mid Q_{U_{2}\mid S_{2}}\mid P_{S_{0}S_{2}})=\Dkl(P_{U_{2}\mid S_{2}}\mid\mid Q_{U_{2}\mid S_{2}}\mid P_{S_{2}})=0$,
and $(e)$ follows since by symmetry, we may condition on $S_{0}=1$.
Thus, the last KL divergence in (\ref{eq: KL divergence for delta estimation lower bound second term derivation})
should be averaged over the cases $S_{2}=-1$ and $S_{2}=1$. For
the first case, it holds that $P_{S_{1}\mid S_{0}S_{2}}(\cdot\mid1,-1)$
is a uniform distribution on $\{\pm1\}$. Hence, conditioned on $S_{0}=1,S_{2}=-1$,
under $P$, $U_{1}$ is a sample from a balanced Gaussian mixture,
just as under $Q$. So, 
\[
\Dkl\left(P_{U_{1}\mid S_{0}=1,S_{2}=-1}\mid\mid Q_{U_{1}}\right)=0.
\]
So, the KL divergence is only comprised of the term in the second
case $S_{2}=1$. Continuing to evaluate the KL divergence in (\ref{eq: KL divergence for delta estimation lower bound second term derivation}),
we get 
\begin{align}
\Dkl\left(P_{U_{1}\mid S_{0}=1,S_{2}}\mid\mid Q_{U_{1}}\mid P_{S_{2}\mid S_{0}=1}\right) & =P_{S_{2}\mid S_{0}=1}(1)\cdot\Dkl\left(P_{U_{1}\mid S_{0}=1,S_{2}=1}\mid\mid Q_{U_{1}}\right)\\
 & \trre[\leq,a]\Dkl\left(P_{U_{1}\mid S_{0}=1,S_{2}=1}\mid\mid Q_{U_{1}}\right)\\
 & \trre[\leq,b]\Dchis\left(P_{U_{1}\mid S_{0}=1,S_{2}=1}\mid\mid Q_{U_{1}}\right)\label{eq: upper bounding KL with chi square for delta impossibility}
\end{align}
where $(a)$ follows since $P_{S_{2}\mid S_{0}}(1\mid1)=(1-\delta_{1})^{2}+\delta_{1}^{2}\leq1$,
and $(b)$ follows from the chi-square divergence bound on the KL
divergence (e.g., \citep[Eq. (2.27)]{tsybakov2008introduction}).
Now, recall that under $Q_{U_{1}}$, it holds that $U_{1}\sim\frac{1}{2}N(t,1)+\frac{1}{2}N(-t,1)$,
and under $P_{U_{1}\mid S_{0}=1,S_{2}=1}$ it holds that $U_{1}\sim(1-\alpha)\cdot N(t,1)+\alpha N(-t,1)$
where 
\[
\alpha=P_{S_{1}\mid S_{0}S_{2}}(-1\mid1,1)=\frac{\delta_{1}^{2}}{(1-\delta_{1})^{2}+\delta_{1}^{2}}.
\]
Thus, the chi-square divergence in (\ref{eq: upper bounding KL with chi square for delta impossibility})
is between a balanced Gaussian mixture (with probability that $S_{1}=1$
being $\frac{1}{2}$) and an unbalanced Gaussian mixture (with probability
that $S_{1}=1$ being $1-\alpha>\frac{1}{2}$). Let $\psi_{\beta}(u)$
denote the probability density function of $U=t\cdot S+W$ with $\P[S=-1]=1-\P[S=1]=\beta$
and $W\sim N(0,1)$. Then, if $\varphi(u)\equiv\frac{1}{\sqrt{2\pi}}e^{-u^{2}/2}$
is the standard Gaussian density function, it holds that 
\begin{align}
\psi_{\beta}(u) & =\beta\cdot\varphi(u+t)+(1-\beta)\cdot\varphi(u+t)\\
 & =\beta\cdot\frac{1}{\sqrt{2\pi}}e^{-(u+t)^{2}/2}+(1-\beta)\cdot\frac{1}{\sqrt{2\pi}}e^{-(u-t)^{2}/2}\\
 & =e^{-t^{2}/2}\cdot\varphi(u)\cdot\left[\beta\cdot e^{-ut}+(1-\beta)\cdot e^{ut}\right].
\end{align}
Specifically, for $\beta=\frac{1}{2}$ it holds that $\psi_{1/2}(u)=e^{-t^{2}/2}\cdot\varphi(u)\cdot\cosh(ut)$.
Therefore, the chi-square divergence from (\ref{eq: upper bounding KL with chi square for delta impossibility})
is upper bounded as 
\begin{align}
 & \Dchis\left(P_{U_{1}\mid S_{0}=1,S_{2}=1}\mid\mid Q_{U_{1}}\right)\nonumber \\
 & =\int_{-\infty}^{\infty}\frac{\left[P_{U_{1}\mid S_{0}S_{2}}(u\mid1,1)-Q_{U_{1}}(u)\right]^{2}}{Q_{U_{1}}(u)}\d u\\
 & =e^{-t^{2}/2}\cdot\int_{-\infty}^{\infty}\varphi(u)\frac{\left[(\alpha-\frac{1}{2})\cdot e^{-ut}+(1-\alpha-\frac{1}{2})\cdot e^{ut}\right]^{2}}{\cosh(ut)}\d u\\
 & \trre[\leq,a]e^{-t^{2}/2}\cdot\int_{-\infty}^{\infty}\varphi(u)\left[\left(\alpha-\frac{1}{2}\right)\cdot e^{-ut}+\left(1-\alpha-\frac{1}{2}\right)\cdot e^{ut}\right]^{2}\d u\\
 & =e^{-t^{2}/2}\cdot\int_{-\infty}^{\infty}\varphi(u)\left[\left(\alpha-\frac{1}{2}\right)^{2}\cdot e^{-2ut}+2\left(\alpha-\frac{1}{2}\right)\left(1-\alpha-\frac{1}{2}\right)+\left(1-\alpha-\frac{1}{2}\right)^{2}\cdot e^{2ut}\right]\d u\\
 & =e^{-t^{2}/2}\cdot\left[\left(\alpha-\frac{1}{2}\right)^{2}e^{2t^{2}}+2\left(\alpha-\frac{1}{2}\right)\left(1-\alpha-\frac{1}{2}\right)+\left(1-\alpha-\frac{1}{2}\right)^{2}e^{2t^{2}}\right]\\
 & \trre[\leq,b]\left[\left(\alpha-\frac{1}{2}\right)^{2}(1+4t^{2})+2\left(\alpha-\frac{1}{2}\right)\left(1-\alpha-\frac{1}{2}\right)+\left(1-\alpha-\frac{1}{2}\right)^{2}(1+4t^{2})\right]\\
 & =4t^{2}\cdot\left[\left(\alpha-\frac{1}{2}\right)^{2}+\left(1-\alpha-\frac{1}{2}\right)^{2}\right]\\
 & =2t^{2}\cdot\left[1-4\alpha(1-\alpha)\right]\\
 & \trre[=,c]2t^{2}\cdot\left[1-\left(\frac{2-8\epsilon^{2}}{2+8\epsilon^{2}}\right)^{2}\right]\\
 & =2t^{2}\cdot\left[1-\frac{2-8\epsilon^{2}}{2+8\epsilon^{2}}\right]\left[1+\frac{2-8\epsilon^{2}}{2+8\epsilon^{2}}\right]\\
 & =2t^{2}\cdot\left[\frac{64\epsilon^{2}}{(2+8\epsilon^{2})^{2}}\right]\\
 & \leq32t^{2}\epsilon^{2}, \label{eqn:dchisq-pu1-given-s01-s21-and-qu1}
\end{align}
where $(a)$ follows since $\cosh(x)\geq1$ for all $x\in\mathbb{R}$,
$(b)$ follows using the assumption $t\leq\frac{1}{\sqrt{2}}$, and
so $e^{2t^{2}}\leq1+4t^{2}$, and $(c)$ is by setting 
\[
4\alpha(1-\alpha)=4\frac{\delta_{1}^{2}(1-\delta_{1})^{2}}{\left[(1-\delta_{1})^{2}+\delta_{1}^{2}\right]^{2}}=4\frac{\left(\frac{1}{2}-\epsilon\right)^{2}\left(\frac{1}{2}+\epsilon\right)^{2}}{\left[\left(\frac{1}{2}+\epsilon\right)^{2}+\left(\frac{1}{2}-\epsilon\right)^{2}\right]^{2}}=\left(\frac{2-8\epsilon^{2}}{2+8\epsilon^{2}}\right)^{2}.
\]
Using the bound \eqref{eqn:dchisq-pu1-given-s01-s21-and-qu1} in (\ref{eq: upper bounding KL with chi square for delta impossibility}),
We thus conclude that 
\[
\Dkl\left(P_{U_{1}\mid S_{0}=1,S_{2}}\mid\mid Q_{U_{1}}\mid P_{S_{2}\mid S_{0}=1}\right)\leq32t^{2}\epsilon^{2}
\]
and then in (\ref{eq: KL divergence for delta estimation lower bound second term derivation})
that the second KL term of (\ref{eq: KL decomposition to signs and conditional samples})
is upper bounded as 
\[
\Dkl\left(P_{U_{1}U_{2}\cdots U_{n}\mid S_{0}S_{2}S_{4}\cdots S_{n}}\mid\mid Q_{U_{1}U_{2}\cdots U_{n}\mid S_{0}S_{2}S_{4}\cdots S_{n}}\mid P_{S_{0}S_{2}S_{4}\cdots S_{n}}\right)\leq16nt^{2}\epsilon^{2}.
\]
Combining this bound with the bound on the first KL term of (\ref{eq: KL decomposition to signs and conditional samples})
given in (\ref{eq: bound on the first KL for lower bound delta estiamtion})
we obtain 
\[
\Dkl\left(P_{U_{1}U_{2}\cdots U_{n},S_{0}S_{2}S_{4}\cdots S_{n}}\mid\mid Q_{U_{1}U_{2}\cdots U_{n},S_{0}S_{2}S_{4}\cdots S_{n}}\right)\leq5n\epsilon^{2}+16nt^{2}\epsilon^{2}.
\]
The aforementioned Pinsker bound on the total variation distance,
and the relation $\sqrt{a+b}\leq\sqrt{a}+\sqrt{b}$ then completes
the proof. 
\end{proof}

\section{Proofs for Section \ref{sec:Flip-probability-estimation}: Analysis
of Algorithm \ref{alg:Mean estimation with unknown delta} \label{sec: proofs for joint estimation}}

\begin{proof}[Proof of Theorem \ref{thm: Estimation error for algorithm}.]

Before presenting the analysis of Algorithm \ref{alg:Mean estimation with unknown delta}, we first specify the choices of the constants $ \lambda_\theta\ge1 $ and $ \lambda_\delta\ge1 $. 
Recall that from Theorem \ref{thm: mean estimation known delta upper bound}
for a Markov model $(\theta_{*},\delta)$ in low dimension, $d\leq\delta n$,
there exists a numerical constant $\lambda_{\theta}>0$,\footnote{The constants can be deduced from the proof, though they were not optimized.
We reiterate that both $ \lambda_\theta $ and $ \lambda_\delta $ below are universal numerical constants. 
In particular, they do \emph{not} depend on $ \theta_*,\delta $ or any other problem parameters. 
The subscripts are merely to emphasize that they are obtained in the estimation procedure for $ \theta_* $ and $ \delta $, respectively. }
so that the estimator $\hat{\theta}\equiv\hat{\theta}_{\text{cov}}(X_{1}^{n};k=\frac{1}{8\delta})$,
which assumes a perfect knowledge of $\delta$, achieves with probability
$1-O(\frac{1}{n})$, 
\begin{equation}
\loss(\hat{\theta},\theta_{*})\leq\lambda_{\theta}\cdot\log(n)\cdot\begin{cases}
\|\theta_{*}\|, & \|\theta_{*}\|\leq\left(\frac{\delta d}{n}\right)^{1/4}\\
\frac{1}{\|\theta_{*}\|}\sqrt{\frac{\delta d}{n}}, & \left(\frac{\delta d}{n}\right)^{1/4}\leq\|\theta_{*}\|\leq\sqrt{\delta}\\
\sqrt{\frac{d}{n}}, & \|\theta_{*}\|\geq\sqrt{\delta}
\end{cases}\label{eq: mean estimation error for joint estimation -- low dimension}
\end{equation}
if $d\leq\delta n$, and 
\begin{equation}
\loss(\hat{\theta},\theta_{*})\leq\lambda_{\theta}\cdot\log(n)\cdot\begin{cases}
\|\theta_{*}\|, & \|\theta_{*}\|\leq\sqrt{\frac{d}{n}}\\
\sqrt{\frac{d}{n}}, & \|\theta_{*}\|\geq\sqrt{\frac{d}{n}}
\end{cases}\label{eq: mean estimation error for joint estimation -- high dimension}
\end{equation}
if $\delta n\leq d\leq n$. We assume that $\lambda_{\theta}\geq1$,
and otherwise replace $\lambda_{\theta}$ by $1$. 

In addition, it holds from Theorem \ref{thm: estimation error of delta for mismatched delta}
that for the estimator $\hat{\delta}$ for $\delta$, which is based
on a mismatched mean $\theta_{\sharp}$, i.e.,
\[
\hat{\delta}\equiv\hat{\delta}_{\text{corr}}(X_{1}^{n};\theta_{\sharp})\dfn\frac{1}{2}\left(1-\hat{\rho}_{\text{corr}}(X_{1}^{n};\theta_{\sharp})\right),
\]
there exists a numerical constant $\lambda_{\delta}>0$, for which
with probability larger than $1-O(\frac{1}{n})$, 
\begin{equation}
\loss(\hat{\delta},\delta)\leq\lambda_{\delta}\left[\frac{\left|\|\theta_{*}\|^{2}-\|\theta_{\sharp}\|^{2}\right|}{\|\theta_{\sharp}\|^{2}}+\frac{\log(n)}{\|\theta_{\sharp}\|^{2}}\sqrt{\frac{d}{n}}\right]\label{eq: delta estimation error for joint estimation}
\end{equation}
assuming that $\|\theta_{\sharp}\|\leq2$ in order to simplify the
bound to the regime of interest. We assume here too that $\lambda_{\delta}\geq1$,
and otherwise replace $\lambda_{\delta}$ by $1$.

For the analysis we assume that all three steps of the algorithm
are successful estimation events. By the union bound, this occurs
with probability $1-O(\frac{1}{n})$.

\uline{Analysis of Step A:}

First, suppose that 
\[
\|\theta_{*}\|\leq\lambda_{\theta}\log(n)\cdot\left(\frac{d}{n}\right)^{1/4}.
\]
Then, it holds by (\ref{eq: mean estimation error for joint estimation -- low dimension})
and (\ref{eq: mean estimation error for joint estimation -- high dimension})
that 
\[
\|\hat{\theta}^{(A)}\|\leq\|\theta_{*}\|+\lambda_{\theta}\log(n)\cdot\left(\frac{\delta d}{n}\right)^{1/4}\leq2\lambda_{\theta}\log(n)\left(\frac{d}{n}\right)^{1/4}
\]
and thus the algorithm will stop and output $\hat{\theta}=0$, for
which it holds that $\loss(\theta_{*},\hat{\theta}^{(A)})=\|\theta_{*}\|$.
This agrees with (\ref{eq: high probability mean estimation error joint very very low dimension})
and (\ref{eq: high probability mean estimation error joint very low dimension}). 

Second, suppose that $\|\theta_{*}\|\geq1$. Now, it holds by (\ref{eq: mean estimation error for joint estimation -- low dimension})
and (\ref{eq: mean estimation error for joint estimation -- high dimension})
\[
\|\hat{\theta}^{(A)}\|\geq\|\theta_{*}\|-\lambda_{\theta}\log(n)\sqrt{\frac{d}{n}}\geq\frac{1}{2},
\]
by the assumption $d\leq\frac{1}{4\lambda_{\theta}^2\log^{2}(n)}\cdot n$.
Thus the algorithm will stop and output $\hat{\theta}=\hat{\theta}^{(A)}$,
for which it also holds that 
\[
\loss(\theta_{*},\hat{\theta}^{(A)})\leq\lambda_{\theta}\cdot\log(n)\cdot\sqrt{\frac{d}{n}}
\]
(irrespective of the value of $\delta$). This also agrees with (\ref{eq: high probability mean estimation error joint very very low dimension})
and (\ref{eq: high probability mean estimation error joint very low dimension}).

Thus, the algorithm will continue to Step B only if 
\[
\lambda_{\theta}\log(n)\cdot\left(\frac{d}{n}\right)^{1/4}\leq\|\theta_{*}\|\leq1
\]
which we henceforth assume. As a preparation for Step B, in which
$\hat{\theta}^{(A)}$ will play the rule of $\theta_{\sharp}$ (the
mismatched mean estimator of Section \ref{sec:Flip-probability-estimation})
in an estimation of $\delta$, we bound the absolute value $ \left|\|\hat{\theta}^{(A)}\|-\|\theta_{*}\|\right| $
and also show that $\|\hat{\theta}^{(A)}\|\geq\frac{1}{2}\|\theta_{*}\|$.
To this end, we may repeat the arguments of the proof of Theorem \ref{thm: mean estimation known delta upper bound},
and specifically, arguments similar to (\ref{eq: norm difference empirical bound})
to show that 
\begin{align}
\left|\|\hat{\theta}^{(A)}\|-\|\theta_{*}\|\right| & \lesssim\log(n)\cdot\psi\left(n,d,\delta=\frac{1}{2},k=1\right) \\
 & =\log(n)\cdot\left[2\sqrt{\frac{\delta}{n}}\cdot\|\theta_{*}\|^{2}+2\sqrt{\frac{d}{n}}\cdot\|\theta_{*}\|+13\sqrt{\frac{d}{n}}+10\frac{d}{n}\right] .
\end{align}
Note that the probability of this event is $1-O(\frac{1}{n})$, and is included
in the successful estimation event mentioned in the beginning of the
proof. Under the assumption $\|\theta_{*}\|\leq1$ it then holds
that
\begin{equation}
\left|\|\hat{\theta}^{(A)}\|-\|\theta_{*}\|\right|\leq\lambda_{\theta}\log(n)\cdot\sqrt{\frac{d}{n}} . \label{eq: norm error of Step A estimator}
\end{equation}
Note that we take $\lambda_{\theta}>0$ to be large enough so that this holds.
The assumed case $\|\theta_{*}\|\geq\lambda_{\theta}\log(n)\cdot\left(\frac{d}{n}\right)^{1/4}$,
and the assumption of the theorem $d\leq$$\frac{n}{16}$ then imply
that 
\begin{equation}
\frac{1}{2}\|\theta_{*}\|\leq\|\hat{\theta}^{(A)}\|\leq2\|\theta_{*}\|\leq2.\label{eq: scale of step A estimator}
\end{equation}

\uline{Analysis of Step B:}

Setting $\theta_{\sharp}=\hat{\theta}^{(A)}$ and utilizing the two
properties just derived in (\ref{eq: norm error of Step A estimator})
and (\ref{eq: scale of step A estimator}), (\ref{eq: delta estimation error for joint estimation})
implies that 
\begin{align}
\loss(\hat{\delta}^{(B)},\delta) & \leq\lambda_{\delta}\left[\frac{\left|\|\theta_{*}\|^{2}-\|\theta_{\sharp}\|^{2}\right|}{\|\theta_{\sharp}\|^{2}}+\frac{\log(n)}{\|\theta_{\sharp}\|^{2}}\sqrt{\frac{d}{n}}\right]\\
 & \leq4\lambda_{\delta}\lambda_{\theta}\frac{\log(n)}{\|\theta_{*}\|^{2}}\sqrt{\frac{d}{n}}.
\end{align}
There are two cases. First suppose that $\|\theta_{*}\|\geq\sqrt{8\lambda_{\delta}\lambda_{\theta}\log(n)}(\frac{d}{\delta^{2}n})^{1/4}$.
In this case, 
\[
\loss(\hat{\delta}^{(B)},\delta)\leq\frac{\delta}{2}
\]
and so 
\begin{equation}
\frac{1}{2}\delta\leq\hat{\delta}^{(B)}\leq2\delta.\label{eq: scale of step B estimator}
\end{equation}
In addition, from (\ref{eq: scale of step A estimator}) it holds
that 
\[
64\lambda_{\delta}\lambda_{\theta}\frac{\log(n)}{\|\hat{\theta}^{(A)}\|^{2}}\sqrt{\frac{d}{n}}\geq16\lambda_{\delta}\lambda_{\theta}\frac{\log(n)}{\|\theta_{*}\|^{2}}\sqrt{\frac{d}{n}}\geq2\delta\geq\hat{\delta}^{(B)}
\]
and so the algorithm will continue to Step C. Now, suppose that $\|\theta_{*}\|\leq\sqrt{8\lambda_{\delta}\lambda_{\theta}\log(n)}(\frac{d}{\delta^{2}n})^{1/4}$.
If the algorithm stops then $\hat{\theta}=\hat{\theta}^{(A)}$, and
the estimation rates of the Gaussian mixture model are achieved, which
agrees with (\ref{eq: high probability mean estimation error joint very very low dimension})
and (\ref{eq: high probability mean estimation error joint very low dimension}).
If the algorithm does not stop and proceeds to Step C then it holds
that 
\begin{align}
\hat{\delta}^{(B)} & \geq64\lambda_{\delta}\lambda_{\theta}\frac{\log(n)}{\|\hat{\theta}^{(A)}\|^{2}}\sqrt{\frac{d}{n}}
\geq16\lambda_{\delta}\lambda_{\theta}\frac{\log(n)}{\|\theta_{*}\|^{2}}\sqrt{\frac{d}{n}}
\geq\delta,
\end{align}
by (\ref{eq: scale of step A estimator}) and the assumption $\|\theta_{*}\|\leq\sqrt{8\lambda_{\delta}\lambda_{\theta}\log(n)}(\frac{d}{\delta^{2}n})^{1/4}$. 

To conclude, there are cases in which the algorithm proceeds to Step
C. If $\|\theta_{*}\|\geq\sqrt{8\lambda_{\delta}\lambda_{\theta}\log(n)}(\frac{d}{\delta^{2}n})^{1/4}$
then it proceeds Step C and (\ref{eq: scale of step B estimator})
holds. Otherwise, the algorithm might proceed to Step C, yet now only
$\hat{\delta}^{(B)}\geq\frac{1}{2}\delta$ is assured. 

\uline{Analysis of Step C:}

If the algorithm has proceeded to Step C, then it is guaranteed that
$\hat{\delta}^{(B)}\geq\frac{1}{2}\delta$ (in any event). Recall
that, ideally, had $\delta$ was known, the choice of the block length $k$
for the estimator $\hat{\theta}_{\text{cov}}(X_{2n+1}^{3n};k)$ is
$k_{*}\dfn\frac{1}{8\delta}$ (Section \ref{sec:Mean-estimation-for}).
It can be readily verified that the analysis of Section \ref{sec:Mean-estimation-for}
is valid when using any \emph{smaller} blocklength $k$, and that
the error rate improve as $k$ increases from $k=1$ to $k=k_{*}$.
In accordance to the analysis of Step B, there are two possible cases.
If $\|\theta_{*}\|\geq\sqrt{8\lambda_{\delta}\lambda_{\theta}\log(n)}(\frac{d}{\delta^{2}n})^{1/4}$
then it holds that $\frac{1}{2}\delta\leq\hat{\delta}^{(B)}\leq2\delta$.
Using $k=\frac{1}{16\hat{\delta}^{(B)}}$ then assures that $k\leq k_{*}$.
In addition, since the error bound scales linearly with $\psi(n,d,\delta,k)$
(see (\ref{eq: definition of psi})), the error can increase by a factor
at most $2$. Thus, $\hat{\theta}^{(C)}$ achieves the error rates in (\ref{eq: mean estimation error for joint estimation -- low dimension})
and (\ref{eq: mean estimation error for joint estimation -- high dimension}),
with a factor of $2$. There are again two cases to consider. If $d\geq\frac{1}{64\lambda_{\delta}^2\lambda_{\theta}^2\log^{2}(n)}\delta^{4}n$,
then the interval $(\sqrt{8\lambda_{\delta}\lambda_{\theta}\log(n)}(\frac{d}{\delta^{2}n})^{1/4},\sqrt{\delta})$
is empty (its right end point is smaller than its left end point),
and plugging $k$ in the error rates of (\ref{eq: mean estimation error for joint estimation -- low dimension})
and (\ref{eq: mean estimation error for joint estimation -- high dimension}),
the resulting error rates are as in the Gaussian mixture model. If
$d\leq\frac{1}{64\lambda_{\delta}^2\lambda_{\theta}^2\log^{2}(n)}\delta^{4}n$
then the aforementioned interval is non-empty, and the resulting error
rates agree with (\ref{eq: high probability mean estimation error joint very very low dimension})
and (\ref{eq: high probability mean estimation error joint very low dimension}).
Otherwise, if $\|\theta_{*}\|\leq\sqrt{8\lambda_{\delta}\lambda_{\theta}\log(n)}(\frac{d}{\delta^{2}n})^{1/4}$,
and the estimator $\hat{\theta}_{\text{cov}}(X_{2n+1}^{3n};k)$ operates
with $k\leq k_{*}$, but $k$ can be as low as $1$. Thus, error rates
of the Gaussian mixture are again achieved, and this agrees with (\ref{eq: high probability mean estimation error joint very very low dimension})
and (\ref{eq: high probability mean estimation error joint very low dimension}).
\end{proof}

\section{Useful results \label{sec:Useful-mathematical-tools}}

\paragraph*{Bernstein's inequality}
Let $ X_1,\cdots,X_\ell $ be independent random variables and $ |X_i|\le b $ almost surely for every $ i\in[\ell] $. 
Then \cite[Proposition 2.14]{wainwright2019high}\footnote{In the original form in \cite[Proposition 2.14]{wainwright2019high}, on the right hand side of the inequality, $ \V[X_i] $ is replaced with $ \E[X_i^2] $ which leads to a seemingly worse concentration bound. 
However, applying this form of Bernstein's inequality to $ Y_i = X_i - \E[X_i] $ allows us to conclude \eqref{eqn:bernstein}. 
} states
\begin{align}
\P\left[\left|\frac{1}{\ell}\sum_{i=1}^\ell(X_i - \E[X_i])\right|\ge\delta\right] &\le 2\exp\left(-\frac{\ell\delta^2/2}{\frac{1}{\ell}\sum_{i=1}^\ell\V[X_i] + b\delta/3}\right) . \label{eqn:bernstein}
\end{align}

\paragraph*{Norm of subGaussian random vectors}
A random vector $ X\in\mathbb{R}^d $ is said to be $ \sigma^2 $-subGaussian if $ \E[X] = 0 $ and 
\begin{align}
\E\left[\exp(t v^\top X)\right] &\le \exp\left(\frac{\sigma^2t^2}{2}\right)
\end{align}
for every $ t\in\mathbb{R} $ and every $ v\in\mathbb{S}^{d-1} $. 
Let $ X\in\mathbb{R}^d $ be a $ \sigma^2 $-subGaussian random vector. 
Then from \citep[Theorem 1.19]{rigollet2019high}, we have for any $ \delta>0 $, 
\begin{align}
\P\left[\max_{v\in\mathbb{B}^d}v^\top X\le4\sigma\sqrt{d}+2\sigma\sqrt{2\log\left(\frac{1}{\delta}\right)}\right]
=\P\left[\|X\|\le4\sigma\sqrt{d}+2\sigma\sqrt{2\log\left(\frac{1}{\delta}\right)}\right]
\ge 1 - \delta . \label{eqn:norm-subgaussian}
\end{align}

\paragraph*{Gaussian covariance estimation}

Let $ W\in\mathbb{R}^{\ell\times d} $ be a matrix with i.i.d.\ $ N(0,1) $ entries. 
Then \citep[Example 6.2]{wainwright2019high} implies that for any $ \delta>0 $, 
\begin{align}
\P\left[\left\|\frac{1}{\ell}W^\top W - I_d\right\|_{\text{op}} \le 2\left(\sqrt{\frac{d}{n}} + \delta\right) + \left(\sqrt{\frac{d}{n}} + \delta\right)^2\right] &\ge 1 - 2e^{-n\delta^2/2} . 
\end{align}

\paragraph*{Davis-Kahan's perturbation bound}

Let $ \Sigma,\hat{\Sigma} $ be symmetric matrices with the same dimensions. 
Let $ \lambda_i(\Sigma) $ and $ v_i(\Sigma) $
denote the $i$th largest eigenvalue and the associated eigenvector (of unit norm) of $ \Sigma $. 
Fix $ i $ and assume that $ \lambda_i(\Sigma) $ is well-separated from the rest of the spectrum of $ \Sigma $:
\begin{align}
\min_{j\ne i} |\lambda_i(\Sigma) - \lambda_j(\Sigma)| &= \delta > 0 . 
\end{align}
Then from \citep[Theorem 4.5.5]{vershynin2018high}, we have
\begin{align}
\loss(v_i(\Sigma),v_i(\hat{\Sigma})) &\le 4\frac{\|\Sigma - \hat{\Sigma}\|_{\text{op}}}{\delta} , \label{eqn:davis-kahan}
\end{align}
where $ \loss(\cdot,\cdot) $ was defined in \eqref{eq: loss function for means}. 

\paragraph*{Packing number}

Let $ M(\delta;\mathbb{B}^d,\|\cdot\|) $ and $ N(\delta;\mathbb{B}^d,\|\cdot\|) $ denote the $\delta$-packing number and $\delta$-covering number of $ \mathbb{B}^d $ w.r.t.\ $\|\cdot\|$, respectively. 
From \citep[Lemma 5.7 and Example 5.8]{wainwright2019high}, we have
\begin{align}
M(\delta;\mathbb{B}^d,\|\cdot\|) &\ge N(\delta;\mathbb{B}^d,\|\cdot\|)
\ge \left(\frac{1}{\delta}\right)^d . \label{eqn:packing}
\end{align}

\paragraph*{Chi-square tail bounds}

From the chi-square tail bound in \citep[Remark 2.11]{boucheron2013concentration}
\begin{equation}
\P\left[\chi_{\ell}^{2}-\ell\geq2\sqrt{\ell t}+2t\right]\leq e^{-t}\label{eq: chi-square upper tail bound}
\end{equation}
and 
\begin{equation}
\P\left[\chi_{\ell}^{2}-\ell\leq-2\sqrt{\ell t}\right]\leq e^{-t}.\label{eq: chi-square lower tail bound}
\end{equation}

\section{Numerical validation}
\label{app:numerical-validation}

In this section, we provide numerical validation of the performance of the estimators $ \hat{\theta}_{\text{cov}}(X_1^n;k) $ (cf.\ \eqref{eq: general PCA estimation rule}), $ \hat{\delta}_{\text{corr}}(X_1^n;\theta_\sharp) = \frac{1}{2}(1 - \hat{\rho}_{\text{corr}}(X_1^n;\theta_\sharp)) $ (cf.\ \eqref{eq: estimator for rho}) and Algorithm \ref{alg:Mean estimation with unknown delta} proposed and analyzed in Theorem \ref{thm: mean estimation known delta upper bound}, Theorem \ref{thm: estimation error of delta for mismatched delta} and Corollary \ref{cor: estimation error of delta for mismatched delta known mean}, and Theorem \ref{thm: Estimation error for algorithm}, respectively. 
Numerical results for these estimators/algorithm are plotted in Figures \ref{fig:plot-theta-estimation}, \ref{fig:plot-delta-estimation} and \ref{fig:plot-theta-estimation-joint}, respectively. 

\begin{figure}[htbp]
  \centering
  \includegraphics[width=0.5\textwidth]{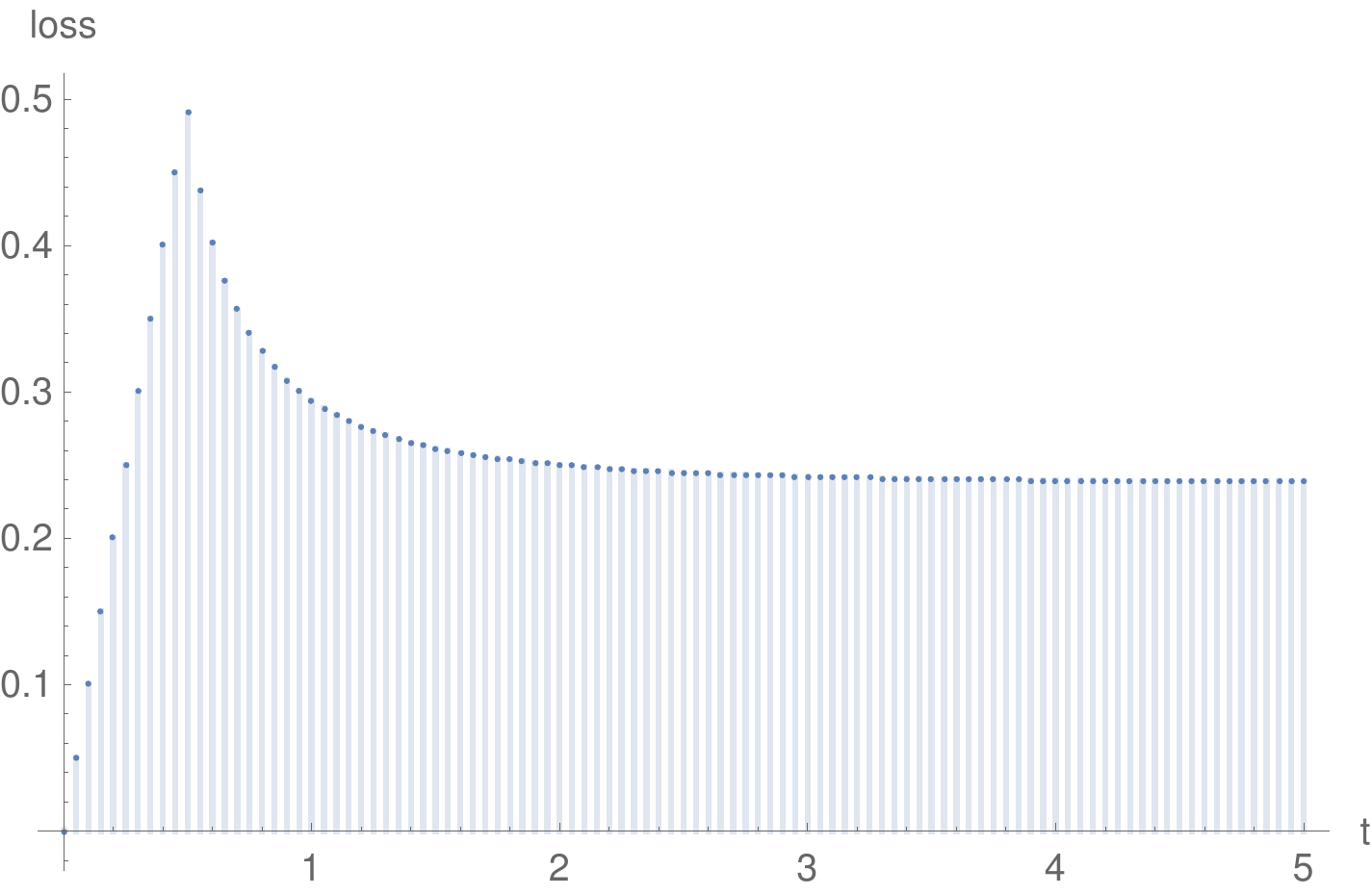}
  \caption{Numerical validation of the performance of the $ \theta_* $-estimator $ \hat{\theta}_{\text{cov}}(X_1^n;k) $ (cf.\ \eqref{eq: general PCA estimation rule}) used in the proof of Theorem \ref{thm: mean estimation known delta upper bound} which assumes a known $ \delta $. 
  We take the minimum loss (cf.\ \eqref{eq: loss function for means}) achieved by $ \hat{\theta}_{\text{cov}}\left(X_1^n;k=  \frac{1}{8\delta}\right) $ and the trivial estimator $ \hat{\theta}_0(X_1^n) = 0 $. 
  We plot the loss as a function of $ t = \|\theta_*\| $ for $ t\in[0,5] $ with step size $ 0.05 $.  
  We take $ n = 5000,d=250,\delta = 0.05 $. 
  Therefore, $ \delta>\frac{1}{n} $ and $ d< \delta n $. 
  }
  \label{fig:plot-theta-estimation}
\end{figure}

\begin{figure}[htbp]
  \centering
  \includegraphics[width=0.48\textwidth]{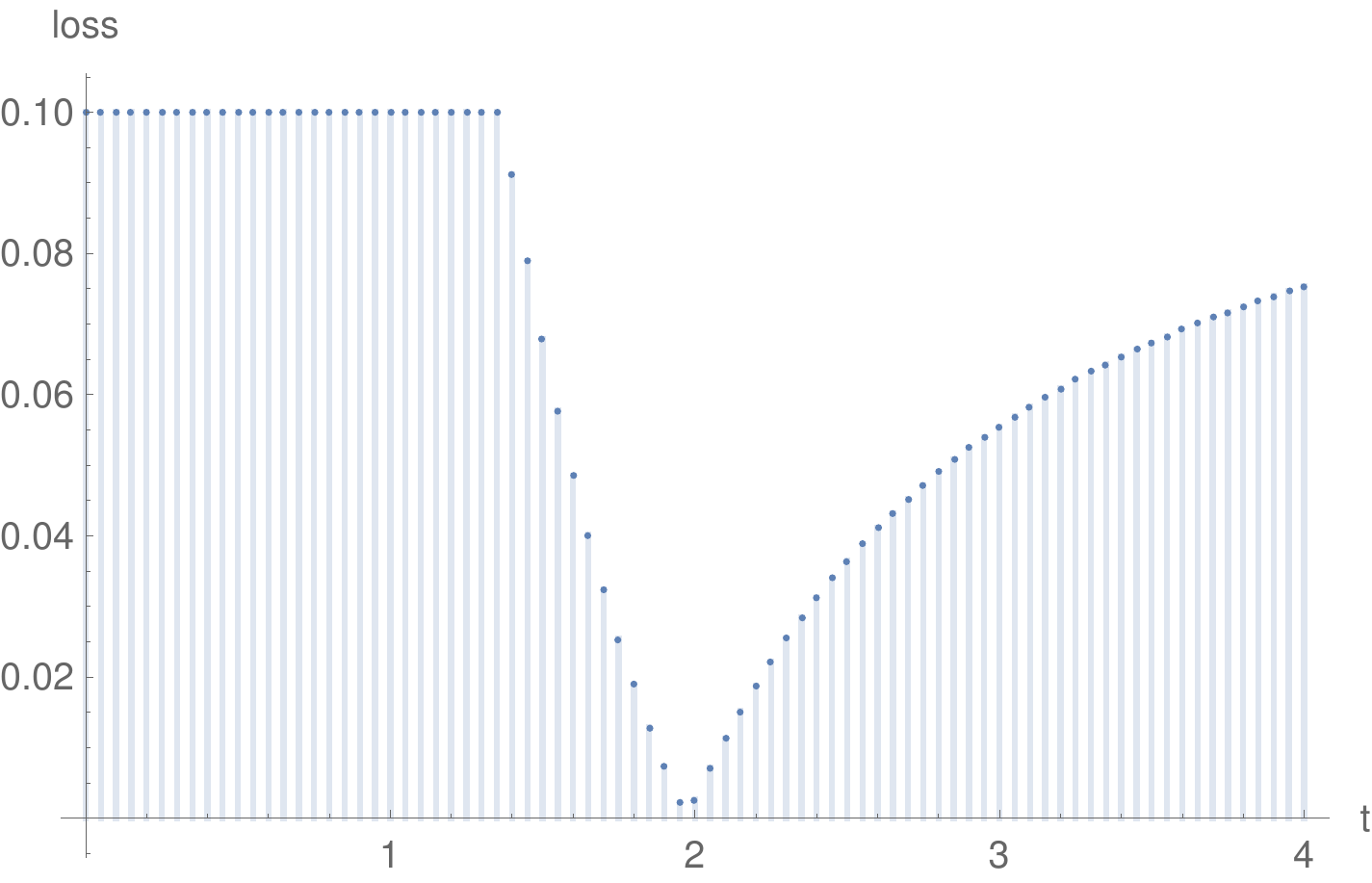} \,
  \includegraphics[width=0.48\textwidth]{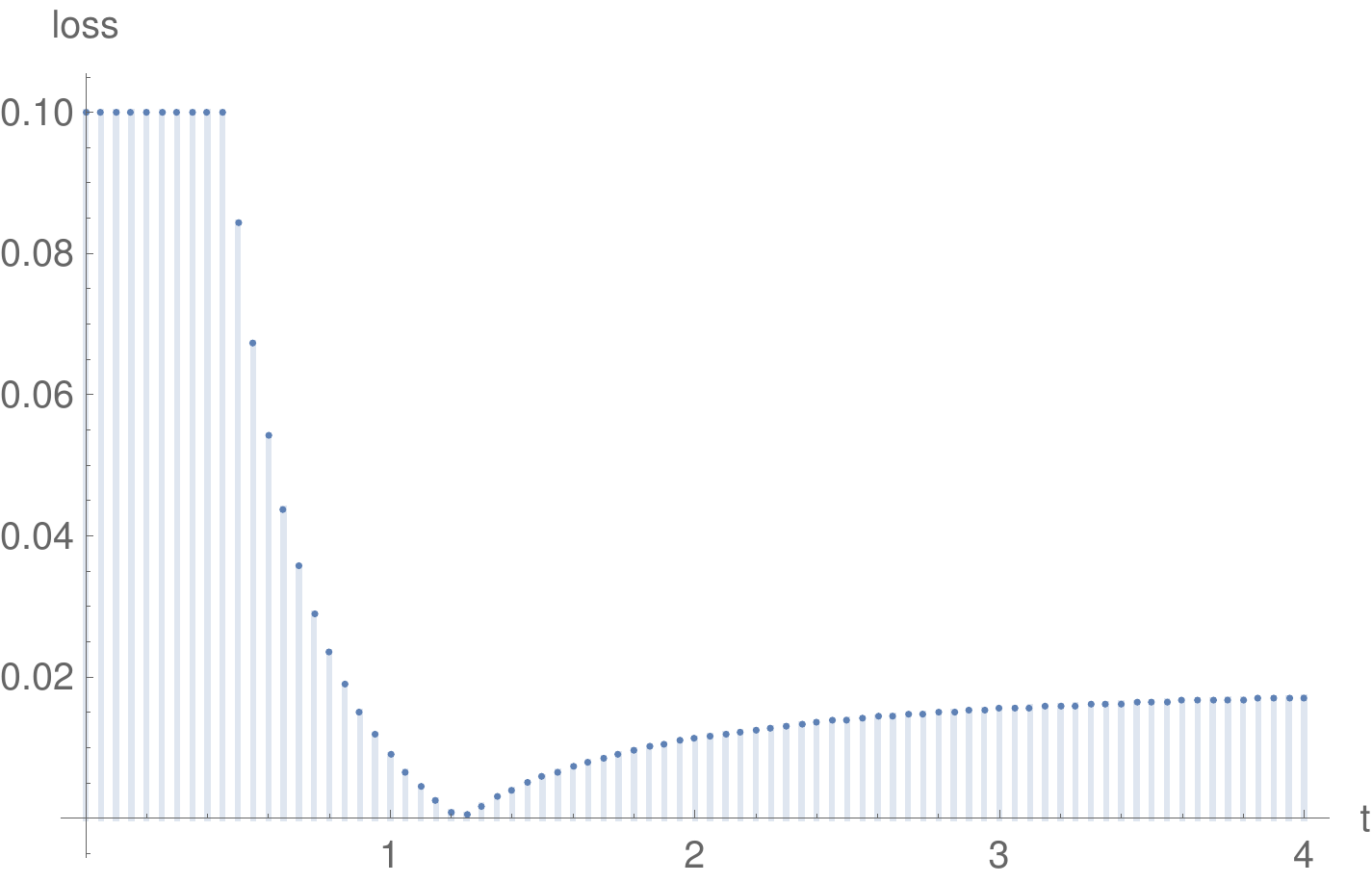}
  \caption{Numerical validation of the performance of the $ \delta $-estimator $ \hat{\delta}_{\text{corr}}(X_1^n;\theta_\sharp) = \frac{1}{2}(1 - \hat{\rho}(X_1^n;\theta_\sharp)) $ where $ \hat{\rho}_{\text{corr}}(X_1^n;\theta_\sharp) $ is given by \eqref{eq: estimator for rho}. 
  This estimator is used in Theorem \ref{thm: estimation error of delta for mismatched delta} (which assumes $ \theta_\sharp $ is a mismatched estimate of $ \theta_* $), Corollary \ref{cor: estimation error of delta for mismatched delta known mean} (which assume $ \theta_\sharp = \pm\theta_* $) and also Theorem \ref{thm: Estimation error for algorithm} (which concerns estimating $ \theta_* $ without the knowledge of $ \delta $). 
  We take the loss $ |\hat{\delta} - \delta| $ achieved by $ \hat{\delta}_{\text{corr}} $ and the trivial estimators $ \hat{\delta}_0(X_1^n) = 0,\hat{\delta}_1(X_1^n) = 1,\hat{\delta}_{1/2}(X_1^n) = \frac{1}{2} $. 
  We plot the loss as a function of $ t = \|\theta_*\| $ for $ t\in[0,1] $ with step size $ 0.05 $. 
  We take $ n = 500,d = 250,\delta = 0.1 $. 
  In the left panel, we assume the estimator has access to $ \theta_\sharp $ with $ \|\theta_\sharp\| = 1.2\cdot\|\theta_*\| $. 
  In the right panel, we assume $ \theta_\sharp = \theta_* $ and therefore the model \eqref{eq: Gaussian Markov model} is equivalent to the model \eqref{eq: model for estimating delta given known theta_star}.
  }
  \label{fig:plot-delta-estimation}
\end{figure}

\begin{figure}[htbp]
  \centering
  \includegraphics[width=0.5\textwidth]{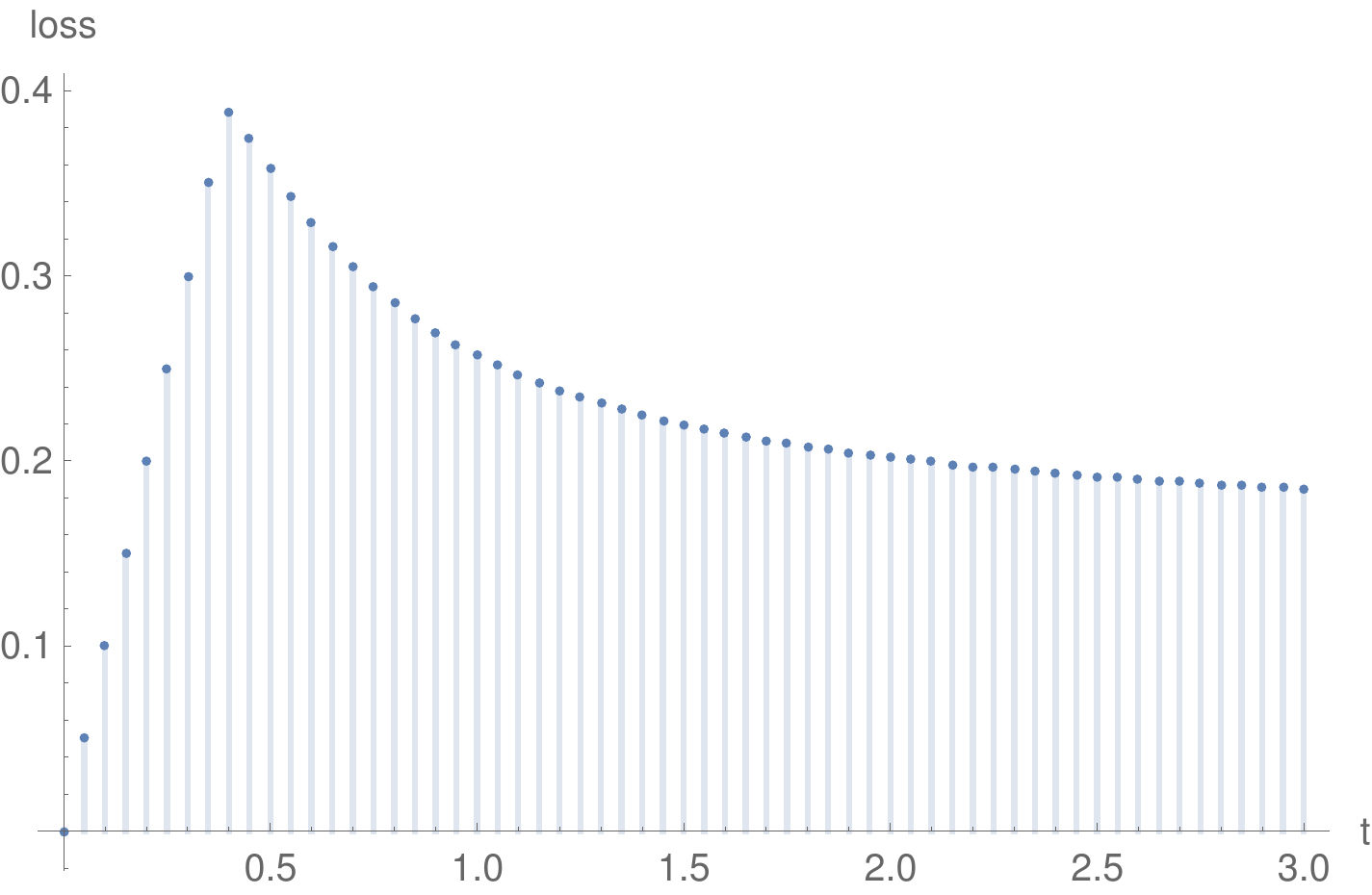}
  \caption{Numerical validation of the performance of Algorithm \ref{alg:Mean estimation with unknown delta} in Theorem \ref{thm: Estimation error for algorithm}. 
  We take the minimum loss (cf.\ \eqref{eq: loss function for means}) achieved by Algorithm \ref{alg:Mean estimation with unknown delta} and the trivial estimator $ \hat{\theta}_0(X_1^{3n}) = 0 $. 
  We plot the loss as a function of $ t = \|\theta_*\| $ for $ t\in[0,4] $ with step size $ 0.05 $. 
  We take $ n = 100,d = 5,\delta = 0.1 $. 
  }
  \label{fig:plot-theta-estimation-joint}
\end{figure}

\section{Open directions}
\label{app:open}

We discuss some open directions pertaining to Remark \ref{rk:extensions}. 
\begin{enumerate}
  \item Suppose $ Z_i\sim N(0,\Sigma) $ i.i.d.\ for some general covariance $ \Sigma\succ 0 $. 
  If $ \Sigma $ is \emph{unknown}, in contrast to the observation above, then the problem becomes significantly more delicate and challenging. A well-known and intuitive example (see, e.g., \citep{Ferguson}) shows that the maximum likelihood estimator (MLE) does not exist even for estimating the mean $\mu$ and variance $\sigma^2$ of a Gaussian mixture with two components $N(0,1)$ and $N(\mu,\sigma^2)$ where $\mu\in\mathbb{R}, \sigma\in\mathbb{R_+}$. 
  In fact, fitting both mean and scale parameters was studied in \citep{Dwivedi2019Challenges,em-overspecified} in the context of the EM algorithm, in a rather restricted setting: The distribution of the samples is standard Gaussian $N(0,I_d)$, yet the estimator is allowed to (over)fit a two-component Gaussian mixture, with symmetric means $\pm \theta$ and a covariance matrix $\sigma^2\cdot I_d$ with any $\sigma>0$. Even in this restricted setting, the result is rather delicate and there are differences, e.g., between one- and multi-dimensional models. 
  We finally remark that even method-of-moments based estimators (as the one we use in our paper) the analysis is also typically made for isotropic noise, e.g., \citep{hsu-moments,wu2020optimal}. 
  Addressing these issues in the context of Markovian model or models with more sophisticated dependence structures is left as an important yet challenging future task. 

  \item Another direction beyond Gaussian noise is to look at noise with a heavy-tailed distribution. 
  There has been some recent progress on this topic in high-dimensional statistics \citep{hopkins-heavy,hopkins-heavy-simple,hopkins-heavy-stats}.
  The estimation error rate is expected to depend on the decay rate of the tail. 
  Additional ideas and techniques will most likely be needed in order to handle heavy tails. 
  We leave it for future research. 
\end{enumerate}

\end{document}